\documentclass[11pt,oneside,english,reqno]{amsart}

\usepackage{subfiles}
\usepackage{amsmath,amssymb,amsthm}
\usepackage[left=2.5cm, right=2.5cm, top=1in, bottom=1in]{geometry}
\usepackage{framed}
\usepackage{mdframed}
\usepackage{graphicx}
\usepackage{bm}
\usepackage{theoremref}
\usepackage{enumerate}
\usepackage{enumitem}
\usepackage{fancyhdr}
\usepackage{color}
\usepackage{float}
\usepackage{dsfont}
\usepackage{mathtools}
\usepackage{breqn}
\usepackage{bbm}
\usepackage{mathrsfs}
\usepackage{amstext}
\usepackage{hyperref}
\usepackage{accents}
\usepackage{tikz}
\usepackage{subcaption}
\usepackage{setspace}
\usepackage{cases}
\usepackage{parskip}
\usepackage{nicefrac}
\usepackage{stmaryrd}

\hypersetup{
	colorlinks  = true,
	citecolor  = blue,
	linkcolor  = blue
}
\usepackage{pdfsync}
\usepackage{amsthm}
\usepackage{amssymb}
\usepackage{mathrsfs}

\usepackage{color}

\usepackage{verbatim}
\usepackage[utf8]{inputenc}
\usepackage[english]{babel}

\newcommand{\notinsubfile}[1]{}

\DeclarePairedDelimiter\floor{\lfloor}{\rfloor}

\numberwithin{equation}{section}

\theoremstyle{remark}

\let\oldproofname=\proofname
\renewcommand{\proofname}{\rm\bf{\oldproofname}}

\newcommand{\mcA}{\mathcal{A}}
\newcommand{\mcB}{\mathcal{B}}
\newcommand{\mcC}{\mathcal{C}}

\newcommand{\mcE}{\mathcal{E}}
\newcommand{\mcF}{\mathcal{F}}
\newcommand{\mcG}{\mathcal{G}}
\newcommand{\mcH}{\mathcal{H}}
\newcommand{\mcI}{\mathcal{I}}

\newcommand{\mcL}{\mathcal{L}}
\newcommand{\mcM}{\mathcal{M}}
\newcommand{\mcN}{\mathcal{N}}

\newcommand{\mcP}{\mathcal{P}}
\newcommand{\mcQ}{\mathcal{Q}}

\newcommand{\mcS}{\mathcal{S}}
\newcommand{\mcT}{\mathcal{T}}

\newcommand{\mcW}{\mathcal{W}}
\newcommand{\mcX}{\mathcal{X}}

\newcommand{\indic}{\mathds{1}}

\newcommand{\mbE}{\mathbb{E}}
\newcommand{\mbF}{\mathbb{F}}

\newcommand{\mbI}{\mathbb{I}}

\newcommand{\mbN}{\mathbb{N}}

\newcommand{\mbP}{\mathbb{P}}
\newcommand{\mbQ}{\mathbb{Q}}
\newcommand{\mbR}{\mathbb{R}}

\newcommand{\bP}{\mathbf{P}}

\newcommand{\bW}{\mathbf{W}}

\newcommand{\mfD}{\mathfrak{D}}

\newcommand{\mfH}{\mathfrak{H}}

\newcommand{\mfJ}{\mathfrak{J}}
\newcommand{\mfK}{\mathfrak{K}}

\newcommand{\mfh}{\mathfrak{h}}

\newcommand{\mfv}{\mathfrak{v}}

\newcommand{\msA}{\mathscr{A}}
\newcommand{\msB}{\mathscr{B}}

\newcommand{\msD}{\mathscr{D}}

\newcommand{\msH}{\mathscr{H}}
\newcommand{\msI}{\mathscr{I}}

\newcommand{\msM}{\mathscr{M}}

\newcommand{\msP}{\mathscr{P}}

\newcommand{\msR}{\mathscr{R}}

\newcommand{\msW}{\mathscr{W}}
\newcommand{\msX}{\mathscr{X}}

\newcommand{\cA}{\mathcal{A}}

\newcommand{\cC}{\mathcal{C}}

\newcommand{\cF}{\mathcal{F}}

\newcommand{\cS}{\mathcal{S}}
\newcommand{\cT}{\mathcal{T}}

\newcommand{\EE}{\mathbb{E}}
\newcommand{\E}{\EE}
\newcommand{\FF}{\mathbb{F}}

\newcommand{\NN}{\mathbb{N}}

\newcommand{\PP}{\mathbb{P}}

\newcommand{\RR}{\mathbb{R}}

\DeclareMathOperator*{\esssup}{ess\,sup}

\newcommand{\RN}[1]{%
	\textup{\uppercase\expandafter{\romannumeral#1}}%
}

\newcommand{\supp}{\text{supp}}
\newcommand{\eps}{\varepsilon}
\newcommand{\dd}{\mathop{}\!\mathrm{d}}

\newcommand{\var}{\text{var}}

\newcommand{\tint}{\mcI}
\newcommand{\llmws}{L^m_{\Omega;w,s}}
\newcommand{\llmw}{L^m_{\Omega;w}}
\newcommand{\lldmw}{L^{2m}_{\Omega;w}}
\newcommand{\llm}{L^{m}_{\Omega}}
\newcommand{\wh}{W^H}

\newcommand{\ls}{\lesssim}
\newcommand{\inv}[1]{\frac{1}{#1}}
\newcommand{\absv}[1]{\left| #1 \right|}

\makeatletter
\numberwithin{equation}{section}
\numberwithin{figure}{section}
\theoremstyle{plain}
\newtheorem{thm}{\protect\theoremname}
\theoremstyle{definition}
\newtheorem{defn}[thm]{\protect\definitionname}
\theoremstyle{remark}
\newtheorem{rem}[thm]{\protect\remarkname}
\theoremstyle{plain}
\newtheorem{lem}[thm]{\protect\lemmaname}
\theoremstyle{plain}
\newtheorem{prop}[thm]{\protect\propositionname}
\theoremstyle{plain}
\newtheorem{cor}[thm]{\protect\corollaryname}
\theoremstyle{plain}

\theoremstyle{plain}
\newtheorem{ass}[thm]{\protect\assumptionname}
\numberwithin{thm}{section}

\usepackage{babel}
\providecommand{\corollaryname}{Corollary}
\providecommand{\definitionname}{Definition}
\providecommand{\lemmaname}{Lemma}
\providecommand{\propositionname}{Proposition}
\providecommand{\remarkname}{Remark}
\providecommand{\theoremname}{Theorem}
\providecommand{\examplename}{Example}
\providecommand{\assumptionname}{Assumption}

\newcommand{\vertiii}[1]{{\left\vert\kern-0.25ex\left\vert\kern-0.25ex\left\vert #1 
		\right\vert\kern-0.25ex\right\vert\kern-0.25ex\right\vert}}
\newcommand{\norm}[1]{\left\lVert #1\right\rVert}

\newcommand{\TV}{\text{TV}}

\title[Ergodicity of Singular fBm Driven SDEs]{Ergodic Theory for Fractional SDE with Singular Coefficients}

\author{{\L}ukasz M\c{a}dry }
\address{Univ. Brest, LMBA - UMR 6205, Brest, France}
\email{lmadry@univ-brest.fr}

\author{Avi Mayorcas}
\address{Department of Mathematical Sciences, University of Bath, UK}
\email{am2735@bath.ac.uk}

\date{\today}

\begin{document}
\begin{abstract}
We show existence and uniqueness of invariant measures for SDE of the form
\begin{equation}\label{eq:SDE}\tag{SDE}
    \dd X_t = g(X_t)\dd t + u(X_t)\dd t + \dd W^H_t,\quad X_0\in \mbR^d,
\end{equation}
     where $W^H$ is a fractional Brownian motion (fBm) with Hurst parameter $H\in (0,\nicefrac{1}{2})$, $u$ is a linearly dispersive term and $g$ is any $B^\alpha_{\infty,\infty}(\mbR^d)$ is distribution in the class treated by Catellier--Gubinelli `16, i.e. $\alpha>1-\frac{1}{2H}$. The significant challenge is to combine the regularizing effect of the fBm with an ergodic theory suited to non-Markovian SDE. Concerning the latter our first main contribution is to construct a bona fide stochastic dynamical system (SDS) (Hairer `05 and Hairer--Ohashi `07) associated to \eqref{eq:SDE}. Since the solution map is only continuous in the support of the stationary noise process we weaken the definitions introduced by Hairer `05 and Hairer--Ohashi `07 but manage to retain the Doob--K'hashminksii provided by Hairer--Ohashi `07. Our second innovation is to introduce a family of flexible local-global stochastic sewing lemmas, in the vein of L\^e `20, which allows us to efficiently treat small and large scales simultaneously. By tuning the local scale as a function of $\|g\|_{B^\alpha_{\infty,\infty}}$ we are able to obtain the necessary continuity of the semi-group and stability estimates to show unique ergodicity for all $g\in B^{\alpha}_{\infty,\infty}(\mbR^d)$. We believe that these local-global sewing lemmas may be of independent interest.\\[1ex]
	\textbf{MSC2020 Subject Classification:} 37A50, 37A25, 60H10, 60H50, 60G22.\\[1ex]
\end{abstract}
\maketitle
\setcounter{tocdepth}{1}
\tableofcontents
\vspace{-1em}

\section{Introduction}
Existence, uniqueness and convergence towards ergodic invariant measures of diffusion equations are questions of significant importance in machine learning \cite{roberts_stramer_02_langevin,jastzebski_kenton_arpit_ballas_fischer_bengio_storkey_18_factors,simsekli_gurbuzbalaban_nguyen_richard_sagun_19_heavy_tailed}, homogenization of fast/slow systems \cite{pavliotis_stuart_08_multiscale,hairer_li_20_averaging,hairer_li_22_generating,LS22} and statistical inference problems \cite{daems_opper_crevecouer_birdal_23_variatonal,amorino_nourdin_shevchenko_25_fractional,amorino_nualart_panloup_sieber_24_fast}. We are concerned with analytic criteria on the drift and noise which ensure existence and uniqueness of ergodic invariant measures for additive noise SDEs of the form
\begin{equation}\label{eq:intro_sde}
X_t = x + \int_0^t b(X_r) \dd r + w_t, \quad t\geq 0,
\end{equation}
where $b:\mbR^d\to \mbR^d$ is a possibly generalised map and $t\mapsto w_t$ is an $\mbR^d$ valued stochastic process. In the case of degenerate and possibly multiplicative noise with regular drift and variance, \cite{hormander_67,malliavin_78_stochastics} give criteria under which the strong Feller property holds, a key ingredient in uniqueness of the invariant measure. This technique has been extended to equations driven by jump processes \cite{cass_09_smooth_jumps,takeuchi_10_BEL_jumps}, infinite dimensional equations \cite{hairer_mattingly_06_ergodicity,bakhtim_mattingly_07_malliavin,gerasimovics_hairer_19_hormander,schonbauer_23_malliavin} and equations driven by continuous Gaussian processes outside the Brownian setting \cite{baudoin_hairer_07_hormander_fractional,cass_friz_10_densities,hairer_pillai_11_ergodicity,hairer_pillai_11_regularity,cass_hairer_litterer_tindel_15_smoothness}. We refer the reader to \cite[Sec.~1]{hairer_16_advanced} for further general references as well as \cite{daprato_zabczyk_96_ergodicity,meyn_tweedie_09_markov,khasminskii_12_stochastic} for surveys of more classical ergodic theory, random dynamical systems and Markov processes.

The focus of this work is instead on the relationship between singularity of the drift coefficient $b$ and irregularity of the noise term, leading to unique ergodicity. The first question therefore, is even well-posedness of the SDE \eqref{eq:intro_sde}. In the case of Brownian noise, this is a long standing question and we only mention a small selection of more general works here, \cite{Zvo74,Ver81,krylov_rockner_05_strong,krylov_21_strong_critical,flandoli_issoglio_russo_17_multidimensional,zhang_zhao_18_singular}, referring to the introductions of \cite{galeati_gerencser_22_subcritical,butkovsky_le_mytnik_23_stochastic,butkovsky_mytnik_24_weak} for a more comprehensive overview. Before addressing ergodic theorems, one can go beyond pure well-posedness to look for strong completeness and regularity of the associated flow. In this direction we refer the reader to \cite{chen_li_14_completeness} for results on strong completeness in the case of non-Lipschitz drift (and non-degenerate diffusion coefficients), \cite{federizzi_flandoli_13_holder,zhang_16_sobolev}  for H\"older regularity of the stochastic flow and Malliavin differentiability in the setting of \cite{krylov_rockner_05_strong} (the second also treating Sobolev regular diffusion coefficients), \cite{meoukeupamen_meyerbrandis_nilssen_proske_zhang_13_variational,mohammed_nilssen_proske_15_sobolev} for proofs of H\"older and Sobolev  regularity of the flow and Malliavin differentiability in the case of bounded measurable drifts, see also \cite{davie_07} for path-by-path uniqueness in this case. Finally, \cite{zie_zhang_16_sobolev} obtained the strong-Feller property for Brownian SDE with locally Sobolev coefficients of super-linear growth.

Concerning existence, uniqueness and convergence towards invariant measures, one typically requires structural assumptions on $b$ as well as regularity conditions. To this end, we specify \eqref{eq:intro_sde} to the case 
\begin{equation}\label{eq:b_decomp}
	b = g+u,
\end{equation}
where $g$ is potentially singular but non-explosive at infinity and $u$ is regular and coercive. \cite{wang_17_integrability} obtains existence, uniqueness and integrability of invariant measures under Novikov type conditions on $g$ and the diffusion, while \cite{zhang_zhao_18_singular} obtains existence and uniqueness of invariant measures under either a  Bessel regularity assumption on $g$, generalizing those of \cite{krylov_rockner_05_strong} beyond the $L^p$ scale, or in a Kato class regime, see also \cite{legall_84_one_dimensional,bass_chen_01_dirichlet,bass_chen_05_one_dimensional} for related well-posedness results. Beyond the scope of this work are results concerning SDE driven by general martingales, for some references on ergodic properties of such equations we refer to \cite{masuda_07_ergodicity,masuda_09_erratum,kulik_09_exponential,xie_zhang_20_ergodicity}.

In recent years it has become well understood that replacing either Brownian or general martingale noise by more oscillatory processes can have an improved regularizing effect on equations with singular coefficients, \cite{nualart_ouknine_02,nualart_ouknine_03, catellier_gubinelli_16,HP21,galeati_gubinelli_20_Prevalence,galeati_gubinelli_20_Noiseless,butkovsky_le_mytnik_23_stochastic,butkovsky_mytnik_24_weak}. This work asks the same question at the level of ergodic invariant measures; asking whether for a given singularity of the drift $g$ (recall \eqref{eq:b_decomp}) there exists a sufficiently oscillatory noise such that \eqref{eq:intro_sde} possesses a unique  ergodic invariant measure. The work \cite{marie_2015_singular} gives an instructive \emph{no-go} theorem in the one dimensional case. Considering the deterministic equation
\begin{equation}\label{eq:marie_de}
    \dd x_t = g(x_t)\dd t + f_t, \quad x_0 \in (0,+\infty), \quad t\in [0,T],
\end{equation}
with $f:[0,T]\to \mbR$ a $\gamma$-H\"older continuous path and $g:(0,+\infty)\to (0,+\infty)$ that is continuously $1+ \floor{\nicefrac{1}{\gamma}}$ differentiable on $(0,+\infty)$ and has bounded derivatives on every interval $[\eps,+\infty)$ for $\eps>0$, \cite[Prop.~2.2]{marie_2015_singular} shows that \eqref{eq:marie_de} has a unique $(0,+\infty)$ valued solution if
\begin{equation}\label{eq:marie_condition}
    \lim_{a\searrow 0} \frac{1}{a^\gamma} \int_0^a g(\iota a^\gamma)\dd a = +\infty\quad \text{for all }\, \iota >0.
\end{equation}
The equivalent result on $(-\infty,0)$ clearly holds by symmetry. As such, if \eqref{eq:marie_condition} holds then an SDE of the form \eqref{eq:marie_de}, with an additional confining potential and $f$ replaced by a stochastic process with almost surely $\gamma$-H\"older continuous paths, cannot have a unique invariant measure on $\mbR$. Notably, if we specify to $g(x) = x|x|^{\alpha-1}$ for $p<0$ the condition \eqref{eq:marie_condition} resolves to $\alpha < 1-\frac{1}{\gamma}$. That is, for more oscillatory driver $f$, one requires the drift to be more singular in order to prevent the solution \emph{tunneling} from $(0,+\infty)$ to $(-\infty,0)$. In fact, if $t\mapsto f_t$ is self-similar of order $\gamma$ (either pathwise or in law) this condition is exactly the scaling super-criticality condition. One has
\begin{equation}\label{eq:scalings}
    \begin{cases}
        \alpha>1-\frac{1}{\gamma}, & \text{scaling sub-critical},\\
        \alpha= 1-\frac{1}{\gamma}, & \text{scaling critical},\\
        \alpha<1-\frac{1}{\gamma}, & \text{scaling super-critical}.
    \end{cases}
\end{equation}
In the scaling sub-critical regime the small scale behavior is dominated by the driver $f$, in the super-critical regime it is dominated by the drift and they are equally balanced in the critical regime. The heuristics are reversed for the large scale behaviors. See also \cite[Sec.~1.1]{galeati_gerencser_22_subcritical} for a discussion of the same ideas in the context of SDEs with drifts in a wide variety of function spaces. See also \cite[Thm.~2.6]{butkovsky_le_mytnik_23_stochastic} for a more general statement of non-uniqueness for a particular family of super-critical drifts.

The fractional Brownian motions (fBm) $W^H$ on a given probability space $(\Omega,\mcF,\mbP)$ are a family of Gaussian processes, indexed by $H\in (0,+\infty)\setminus \mbN$ which generalize the standard Brownian. When $H=\nicefrac{1}{2}$, $W^H$ is a standard Brownian motion, while for $H\neq \nicefrac{1}{2}$ the process $\{t\mapsto W^H_t\}$ is neither a semi-martingale nor a Markov process. As a result one does not have access to PDE analytic methods such as the Feynman--Kac formula without introducing infinitely many degrees of freedom, see \cite{VZ19} for example. However, the processes are self-similar in law at scale $H$, i.e. the family $\{t\mapsto c^{-H} W^H_{ct}\}_{c> 0}$ are all equal in law. In addition, due to Gaussianity and the Kolmogorov continuity criterion one has $\mbP\left(W^H\in \mcC^{\gamma}([0,T];\mbR^d)\right) =1$ for all $\gamma\in (0,H)$. Replacing $f$ with $W^H$ and considering a generic $g\in B^{\alpha}_{\infty,\infty}(\mbR^d;\mbR^d)$ (for discussion of the H\"older--Besov spaces see Section~\ref{sec:notations}) then the criticality conditions \eqref{eq:scalings} are exactly mirrored with $\gamma \mapsto H$, see \cite[Ex.~1.1]{galeati_gerencser_22_subcritical}.  While \eqref{eq:scalings} provides an indication of the regimes where one may expect well-posedness, at the time of writing sharp results at these boundaries are not available in full generality. Some special cases are treated for example in \cite[Ch.~II]{sznitman_91_topics}, \cite{bossy_talay_96}. 

 Existence and uniqueness results, in the Brownian case are extremely extensive. Of most relevance here are those works which treat only regularity assumptions on the drift $b$ (equally $g$) without extensive structural or geometric assumptions. The closest works in this regard begin with \cite{Zvo74,Ver81} and have been more recently developed by \cite{krylov_rockner_05_strong,bass_chen_03_singular} and related works. Crucially, \cite{krylov_rockner_05_distribution} obtained a sharp (see \cite{krylov_21_Ld_drift,beck_flandoli_gubinelli_maurelli_19_critical}) criterion for strong well-posedness, that is $b \in L^q([0,T];L^p(\mbR^d;\mbR^d))$ for $\frac{d}{p}+\frac{2}{q}<1$ for $p\geq 2$. The critical cases were treated by \cite{bass_chen_05_one_dimensional} ($d=1$) and \cite{krylov_21_Ld_drift} ($d\geq 3$). Very recently \cite{lu_rockner_25_sharp,lu_rockner_zhu_25_lagrangian,rockner_zhao_25_critical} studied examples of both uniqueness and non-uniqueness of strong solutions in the critical case $\frac{d}{p}+\frac{2}{q}=1$ for $d=2$. Results in the Kato class have been obtained by \cite{bass_chen_03_singular} and have recently been comprehensively treated by \cite{zhang_zhao_18_singular} along with a generalization of the $L^q_T L^p_x$ condition from \cite{krylov_rockner_05_strong} to Bessel potential spaces $\mathcal{H}^{-\alpha,p}(\mbR^d;\mbR^d)$ (Bessel potential class) for $\alpha \in (0, \inv{2}], p \in \big( \frac{d}{1-\alpha}, +\infty \big)$. For drifts measured in the H\"older--Besov scale, $b\in \mcC^{\alpha}(\mbR^d;\mbR^d)$, \cite{flandoli_issoglio_russo_17_multidimensional} obtained weak well-posedness for $\alpha>-\nicefrac{1}{2}+$ while \cite{delarue_diel_16_rough,cannizzaro_chouk_18_multidimensional} obtained a weaker notion of well-posedness for special classes of drifts $b\in \mcC^{\alpha}(\mbR^d;\mbR^d)$ for $\alpha>-\nicefrac{2}{3}$. See \cite{kremp_perkowski_22_multidimensional,kremp_perkowski_23_rough} for similar result in the case of Lévy noise. Concerning ergodic results in the case of singular drifts, \cite{zhang_zhao_18_singular} gives the most general results to date, in particular showing uniqueness of ergodic invariant measures for $g$ in the class treated by \cite{krylov_rockner_05_strong}.

Concerning well-posedness of \eqref{eq:intro_sde} the most pertinent results in our context are those concerning singular drifts and fractional Brownian noise. In this regard, \cite{nualart_ouknine_02,catellier_gubinelli_16} are the earliest results in this direction, with our work being more in the spirit of \cite{catellier_gubinelli_16}. Without focussing on specific cases (e.g. one dimensional problems or specific structural assumptions \cite{anzeletti_R._tanre_23_regularisation}) weak well-posedness of \eqref{eq:intro_sde} with $w$ replaced by an fBm $W^H$ has been shown by \cite{butkovsky_mytnik_24_weak} provided $b\in \mcC^{\alpha}(\mbR^d;\mbR^d)$ with $\alpha >\frac{1}{2}-\frac{1}{2H}$. On the other hand, strong well-posedness, in the same setting has been obtained by \cite{galeati_gerencser_22_subcritical} provided $\alpha>1-\frac{1}{2H}$ (given as a subcase of their more general condition which allows for time dependent drifts and $H>1$, see \cite[Thm.~1.4]{galeati_gerencser_22_subcritical}). This latter work provides a significant number of results which we appeal to here, namely existence of a regular stochastic flow and Malliavin differenitability of the solution map. There exist a number of equally general results beyond the $L^q_T\mcC^\alpha_x$ scale, we refer the reader, for example, to \cite[Sec.~6]{le_20_stochSewing} \cite[Rem.~1.10]{galeati_gerencser_22_subcritical} and \cite{butkovsky_le_mytnik_23_stochastic}. More broadly, McKean--Vlasov equations and associated propagation of chaos have been treated in \cite{galeati_harang_mayorcas_23_fractional,galeati_le_mayorcas_24_quantitative}, convergence rates of numerical schemes in \cite{BDG19, GHR23}, equations with multiplicative noise by \cite{dareiotis_gerencser_24_multiplicative, catellier_duboscq_25_regularisation}, extensions to SPDEs in \cite{ABLM24} and  zero noise limits by \cite{GM24}.

With these well-posedness results available, it is natural to ask for ergodic properties of diffusions under the same conditions. However, this introduces new challenges. As fBm is not a Markov process (as well as failing to be a semi-martingale) the random dynamical systems (RDS) framework, \cite{arnold_98_rds,daprato_zabczyk_96_ergodicity}, is ill suited to obtaining uniqueness of ergodic invariant measures.\footnote{For existence results using this approach we refer to \cite{maslowski_schmalfuss_04_random,marie_2015_singular}.}   Informally speaking, the deterministic perspective of an RDS approach, where sampled noise is fed into the update of a system forwards in time, does not reliably exclude the invariant measures which are not \emph{measurable with  respect to the past} (see \cite[Sec.~1.4]{arnold_98_rds} for this terminology). Instead, the works \cite{Hai05,hairer_ohashi_07_ergodic} developed a more probabilistic approach based on the notion of a stochastic dynamics system (SDS). Here, the noise updates are carried along with the dynamics appropriate measurability of the invariant measures. We refer to the introductions of \cite{hairer_ohashi_07_ergodic} for more discussion. The notion of an SDS was first introduced in \cite{Hai05} where an additive noise equation, of the form \eqref{eq:intro_sde} with locally Lipschitz and contractive drift was treated by a hands on coupling approach, obtaining existence and uniqueness of an ergodic invariant measure (understood in the appropriate sense) as well as algebraic ergodicity. As a follow up work, \cite{hairer_ohashi_07_ergodic} treated a multiplicative noise equation and obtained a generalization of the Doob--K'hashminksii theorem for SDS, \cite[Thm.~3.10]{hairer_ohashi_07_ergodic}. This gives uniqueness of ergodic invariant measures for \emph{quasi-Markovian systems} (see \cite[Def.~3.7]{hairer_ohashi_07_ergodic}) enjoying natural generalizations of the strong Feller and irreducibility properties. Our work takes this latter approach, building on the regularity of solutions obtained in \cite{galeati_gerencser_22_subcritical}, see Theorem~\ref{th:main_intro} for an informal statement of our setting and main result. Beyond these initial results, \cite{li_panloup_seiber_23_non_stationary} obtained Gaussian upper and lower bounds on the densities associated to fractional diffusions, under appropriate conditions, along with a number of other results. While \cite{panloup_R._20_sub_exponential} demonstrated sub-exponential convergence to the unique invariant measure for fractional diffusions with  semi-contractive and locally Lipschitz drifts. This was upgraded to exponential convergence by \cite[Thm.~1.3]{LS22} in the case of globally Lipschitz semi-contractive drifts with small non-contractive perturbation. Meanwhile, \cite{amine_coffie_harang_proske_20_bismut} obtained a Bistmut--Ellworthy--Li formula  for SDE driven by additive fBm with bounded and measurable drifts. Our Section~\ref{subsec:strong_feller} extends this formula to the case of distributional drifts.

To informally state our main result, we fix the following equation
\begin{equation}\label{eq:intro_fbm_sde}
     X_t =x + \int_0^t g(X_r)\,\dd r + \int_0^t u(X_r)\, \dd r + W^H_t,
\end{equation}
for $W^H$ an $\mbR^d$ valued fBm of Hurst parameter $H\in (0,\nicefrac{1}{2})$. Since we will allow $g$ to be a generalised function well-posedness, even well-definition, of the dynamics \eqref{eq:intro_fbm_sde} is somewhat delicate and we postpone this discussion until Section~\ref{sec:sds_construction}. Equally, the precise notion of unique ergodicity is postponed here to Section~\ref{sec:sds}, although it remains essentially unchanged from that of \cite{Hai05,hairer_ohashi_07_ergodic}.
Our main results can are summarized as follows, with reference to precise statements given below.
\begin{thm}\label{th:main_intro}
Let $H\in (0,\nicefrac{1}{2})$, $\lambda>0$, $g \in B^\alpha_{\infty,\infty}(\mbR^d;\mbR^d)$ and $u:\mbR^d\to \mbR^d$ be such that 
\begin{equation*}
    |\nabla u(x)| \,\leq\,  \lambda \quad \text{and}\quad \langle u(x)-u(y),x-y\rangle \leq -\lambda |x-y|^2\quad \text{for all }  x,\,y\in \mbR^d.
\end{equation*}
Then the following statements hold:
\begin{enumerate}
    \item \label{it:intro_existence} If $\alpha >\frac{1}{2}-\frac{1}{2H}$ there exists a weakly stationary solution to \eqref{eq:intro_fbm_sde} (Corollary~\ref{cor:weak_existence}). 
    \item  \label{it:intro_uniqueness} If $\alpha >1-\frac{1}{2H}$ then there exists unique global strong solution to \eqref{eq:intro_fbm_sde} and associated  stochastic dynamical system. In addition  there exists a unique ergodic invariant measure (Theorem~\ref{th:sds_well_defined} and Theorem~\ref{th:unique_ergodicity}). 
    \item \label{it:intro_gaussian} In either case, for any $\gamma \in (0,H)$, there exists a $\kappa_0\coloneqq \kappa_0(\gamma,\alpha,H,\|g\|_{\mcC^{\alpha}_x},\lambda,d)>0$ (which vanishes as $\gamma \to H$ or $\|g\|_{\mcC^{\alpha}_x} \to +\infty$) such that (Theorem~\ref{th:tightness_main})
    \begin{equation*}
     \sup_{t \geq 0} \EE \left[ \exp\left( \kappa \llbracket X\rrbracket_{\mcC^{\gamma}_{[t,t+1]}}^2 \right) \right] \ls 1 \quad \text{for all } \kappa < \kappa_0.
    \end{equation*}
\end{enumerate}
\end{thm}
\begin{rem}\label{rem:holder_besov_embedding}
    Note that due to the compactness of the embedding $B^\alpha_{\infty,\infty}(\mbR^d) \hookrightarrow  B^{\alpha'}_{\infty,\infty}(\mbR^d)$ for $\alpha'<\alpha$, given $g\in B^\alpha_{\infty,\infty}(\mbR^d)$ there exists a smooth and bounded sequence $g^n$ such that $\|g-g^n\|_{B^{\alpha'}_{\infty,\infty}} \to 0$ for any $\alpha'<\alpha$. See Section~\ref{sec:notations} for a recap of these spaces and their definitions.
\end{rem}
For simplicity and concision we have decided only to treat the case $H\in (0,\nicefrac{1}{2})$, a treatment of $H \in (\nicefrac{1}{2},1)$ is given by \cite[Sec.~3.2]{li_panloup_seiber_23_non_stationary}. However, it seems reasonable to expect that the same result should hold for all $H\in (0,+\infty)\setminus\mbN$, referring to \cite{galeati_gerencser_22_subcritical} for a discussion of the process $W^H$ with $H>1$ and the required a priori estimates without the unbounded coefficient $u$. Regularization by noise results similar to \cite{galeati_gerencser_22_subcritical} have been studied for multiplicative analogues of \eqref{eq:intro_fbm_sde} in the setting of regular variance and singular drifts by \cite{dareiotis_gerencser_24_multiplicative,dareiotis_gerencser_le_ling_24_gaussian,catellier_duboscq_25_regularisation} and in the case of singular variance only  \cite{matsuda_mayorcas_25_pathwise}. It seems reasonable to expect unique ergodicity for a multiplicative variant of \eqref{eq:intro_fbm_sde}, in the spirit of \cite{hairer_ohashi_07_ergodic,hairer_pillai_11_ergodicity}. 

Concerning the rate of convergence to the ergodic invariant measure, \cite{Hai05} obtained an algebraic rate by constructing a direct coupling between solutions. Since the fundamental argument relies on stability of the solution map, which we also obtain in our setting (Section~\ref{sec:local_stability}), we expect that the same arguments could be adapted to our singular case giving the same geometric convergence. While preparing this manuscript we became aware of the work \cite{dareiotis_haress_le_25_uniform} which obtains exponential ergodicity for \eqref{eq:intro_fbm_sde} under similar assumptions as ours but restricted to the case $\|g\|_{B^\alpha_{\infty,\infty}}$ small compared to $\lambda$. See also \cite{LS22} where a geometric rate of convergence was obtained in the regular coefficient setting for small perturbations of the dispersive term.
\subsection{Main Challenges, Ideas and Organization}

The initial challenge in treating \eqref{eq:intro_fbm_sde} when $g$ is allowed to be a generalised function is to even define the right hand side. This is by now a well understood issue, \cite{catellier_gubinelli_16,galeati_gubinelli_20_Noiseless}. The remedy is to notice that we are not required to evaluate $g$ at any point in $\mbR^d$ but along $\{t\mapsto X_t\}$, any putative solution to \eqref{eq:intro_fbm_sde}. To this end, we expect that any solution can be decomposed as  $X = x_0 + \theta + W^H$ where $\{t\mapsto \theta_t\}$ is \emph{slow} in comparison to $W^H$. As a result we are really attempting to evaluate
\begin{equation}\label{eq:local_averaging}
    t\mapsto \int_0^t g(\theta_r +x_0 + W^H_r)\dd r.
\end{equation}
This is much more manageable due to the \emph{local non-determinism} of $\{t\mapsto W^H_t\}$, see Lemma~\ref{lem:lnd}, which gives rise to an averaging over $g$ in \eqref{eq:local_averaging}, see Lemma~\ref{lem:lnd_to_heat_kernel} and Lemma~\ref{lem:integral_estimate}. This insight is at the core of \cite{davie_07,catellier_gubinelli_16} and many subsequent works. Technically, one first takes a smooth and bounded $g\in C^\infty_b(\mbR^d)$ and, via stochastic sewing methods, obtains sufficient a priori bounds on the process $\theta$ which only depend on $\|g\|_{B^\alpha_{\infty,\infty}}$. These are fed back into expressions like \eqref{eq:local_averaging} and their analogues with $g$ replaced by $\nabla g$, see Lemma~\ref{lem:averaged_field_tightness}. Putting these estimates together we obtain uniqueness of strong solutions, in an appropriate sense given by Definition~\ref{def:singular_sde}, to \eqref{eq:intro_fbm_sde} with genuinely distributional $g$. This is carried out in Proposition~\ref{prop:sde_well_posed}.

The inability to evaluate $\mbR^d \ni x\mapsto g(x)$ has consequences in our proof of tightness and the strong Feller property. Concerning the former, we broadly follow the proof of \cite[Prop.~3.12]{Hai05} with appropriate modifications.  For simplicity let us set $\lambda =1$, $g$ to be smooth and introduce the Ornstein-Uhlenbeck process
\begin{equation*}
    \dd Y_t = - Y_t \dd t + \dd \wh_t, \quad Y_0 = 0.
\end{equation*} 
Setting $\rho_t \coloneqq X_t - Y_t$, a direct and standard calculation yields the inequality
\begin{equation}\label{eq:rho_intro_bound}
    \absv{\rho_t}^2 \leq \absv{\rho_0}^2 e^{- t} + \left|\int_0^t e^{-(t-r)} \rho_r \cdot g(X_r) \dd r\right| + \left|\int_0^t e^{-(t-r)} u(Y_r) \dd r\right|.
\end{equation}
If in reality our $g$ was not a distribution but a continuous function we could attempt to bring the absolute values inside the integral of the second term on the right hand side. Instead, in order to obtain estimates uniform in $\|g\|_{B^\alpha_{\infty.\infty}}$ we follow a similar strategy as used to define the right hand side of \eqref{eq:intro_fbm_sde} and seek to use stochastic sewing and the averaging principle developed in Lemma~\ref{lem:integral_estimate}. Herein lies one of our main innovations, to develop a \emph{local-global} sewing lemma which allows one to combine \emph{small scale} properties of $\{t\mapsto X_t\}$ with the \emph{large scale} damping effect of $\{r\mapsto e^{-(t-r)}\}$ on large scales, see Lemma~\ref{lem:sewing_decay_stochastic_integral}. Crucially we are able to tune the \emph{small/large scale} threshold as a function of the problem parameters, in our case the size of $\|g\|_{B^\alpha_{\infty,\infty}}$. Similar statements can be found as \cite[Cor.~3.7]{butkovsky_le_mytnik_23_stochastic} and \cite[Lem.~A.3.1]{haress_24_numerical_thesis}, however, the former is purely local while the latter does not provide a tunable separation of scales required for our arguments. An application of Young's inequality, after obtaining suitable uniform in time bounds, along with the linear growth of $\{y\mapsto u(y)\}$ gives us a bound on the right hand side of \eqref{eq:rho_intro_bound} which is at most quadratic in $\rho$. By carefully choosing the balance between \emph{local} and \emph{global} scales in our \emph{local-global} sewing lemma we absorb the quadratic terms into the left hand side and close by a standard Gr\"onwall argument, Lemma~\ref{lem:my_gronwall}. These arguments are collected in the proof of Theorem~\ref{th:tightness_main} which gives uniform in time Gaussian moments of $X$. This almost directly gives us existence of solutions, weakly stationary measures and Gaussian tails of any invariant measure. Note that the standard sewing lemma only gives exponential moments due to the use of a discrete BDG inequality. We bypass this step by considering germs which expand as finite variation and martingale part, allowing us to apply a continuous BDG inequality, and retain Gaussian moments. For technical reasons we additionally allow for deterministic continuous perturbations of the noise and conditioning on its history for $t <0$, see the beginning of Section~\ref{sec:tightness}.

The above philosophy carries us through much of the rest of the paper. With tightness in hand, and appropriate continuity of the solution in drift and deterministic perturbations of the noise we are able to construct a bona fide stochastic dynamical system, Theorem~\ref{th:sds_well_defined} and obtain sufficient estimates on the associated Malliavin derivative and Jacobian, Section~\ref{subsec:strong_feller}. A key point here, alluded to above, is that our solution map will not be continuous in the entirety of the noise space $\mcW \subset \mcC^{H-}(\mbR;\mbR^d)$, since we require specific stochastic properties of the fBm to even define the right hand side of our equation. To this end we slightly weaken the definition of a stochastic dynamical system as compared with \cite{Hai05,hairer_ohashi_07_ergodic}, see Definition~\ref{def:sds}. However, we check that with minimal modification the same abstract unique ergodicity results apply, Theorem~\ref{th:abstract_ergodic_result} and yields the same analogue of the Doob--K'hasminksii theorem, see Theorem~\ref{th:doob_khasminskii} 

In verifying the assumptions of Theorem~\ref{th:doob_khasminskii}, we provide our second major contribution, a Bismut--Elworthy--Li formula for \eqref{eq:intro_fbm_sde} with distributional $g \in B^\alpha_{\infty,\infty}(\mbR^d)$. Here it is crucial that we are able to obtain Gaussian estimates on processes of the form
\begin{equation*}
	t\mapsto  \int_0^t \nabla g(X_s)\dd s,
\end{equation*}
since the Jacobian and Malliavin derivative are controlled by exponential moments of these processes. These strong estimates, reminiscent of \cite[Lem.~6.3]{galeati_gerencser_22_subcritical} which do not allow for the unbounded term $x\mapsto u(x)$ allow us to exchange integration and expectation in order to apply Malliavin integration by parts and avoid a technical stopping time argument used in \cite{hairer_ohashi_07_ergodic}. Similar estimates were also obtained by \cite{hu_nualart_07_differential} in the setting of regular coefficients and $H>\nicefrac{1}{2}$. These arguments are collected in the proof of Theorem~\ref{th:strong_feller_proof}.

Finally, we verify topological irreducibility through a Girsanov argument, locally in time (Section~\ref{subsec:irreduc}) and quasi-Markovianity by the same arguments as in the proof of \cite{hairer_ohashi_07_ergodic}, only modified to allow for $H<\nicefrac{1}{2}$ (Section~\ref{subsec:quasimarkov}). This concludes the proof of unique ergodicity. 

A number of appendices collect useful results on Gr\"onwall-type inequalities, Appendix~\ref{app:gronwall}, moment bounds on integrals of Gaussian processes, Appendix~\ref{sec:gaussian_bounds}, our modified sewing results, Appendix~\ref{app:stoch_sewing_proofs}, extensions of various proofs from \cite{hairer_ohashi_07_ergodic} to $H<\nicefrac{1}{2}$ and the abstract ergodicity result \cite[Thm.~3.2]{hairer_ohashi_07_ergodic} to our slightly weakened definition of an SDS, Appendix~\ref{app:doob_khasminski_proof} and Appendix~\ref{app:A_operator_properties}.
\section*{Acknowledgments}
A large part of this work was done while ŁM was employed as an ATER at Université Paris-Cité, and the project was started during last months of his PhD thesis at Université Paris-Dauphine. We also acknowledge generous funding from ERC \emph|{Noisy Fluids} (Advanced Grant no. 101053472) which allowed the authors to work on the project during the conference \emph{Turbulence on the banks of Arno} in Feb. 2024, the research stay of ŁM in Pisa in June 2024 and the extended stay of AM from Feb. 2024 to Aug. 2024.
\section{Preliminaries}\label{sec:preliminaries}
\subsection{Notations}\label{sec:notations}
We collect some more notation used throughout the paper. We define the sets, $\mbN = \{0,1,\ldots\}$, $\mbN_{>0} \coloneqq \mbN\setminus \{0\}$, $\mbR_+ \coloneqq \mbR\cap [0,+\infty)$ and $\mbR_- \coloneqq \mbR\cap (-\infty,0]$. Given a measurable space $(E,\mcE)$ we write $\msM(E)$ for the space of Radon measures on $E$, $\msM_1(E)$ for the space of probability measures on $E$ and $\msM_+(E)$ for the set of positive finite, Radon measures on $E$. All three sets are endowed with the topology of weak convergence. If otherwise specified, for $\mu,\,\nu \in \msM(E)$, we write $\mu\sim \nu$ to denote that $\mu$ and $\nu$ are mutually absolutely continuous.

For $0\leq S < T \leq +\infty$, we define the two and three dimensional simplexes by
	\begin{equation}\label{eq:simplex}
    \begin{aligned}
        [S,T)^2_{\leq} \coloneqq & \{ (s,t) \in [S,T)^2 \,:\, s\leq t\},\\
		[S,T)^3_{\leq} \coloneqq & \{ (s,u,t) \in [S,T)^3\,:\, s\leq u\leq t\}
    \end{aligned}
	\end{equation}
with analogous notation for the closed interval $[S,T]$ when $T<+\infty$. For $s,t \in [S,T)_\leq$ and $Y:[S,T)\to E$, we write $Y_{s,t}\coloneq Y_t-Y_s$. Given $A:[S,T)^2_{\leq}\to E$, we define a map $\delta A:[S,T)^3_{\leq}\to E$ by setting
	\begin{equation*}
	    \delta A_{s,u,t} = A_{s,t}-A_{s,u}-A_{u,t}.
	\end{equation*}
Whenever the specification of the domain and range are not important or clear from the context, for $\ell,\,n \in \mbN$, we denote by $C^0_x$ the Banach space $C_b(\RR^\ell;\RR^n)$ of continuous, bounded functions $f:\RR^\ell\to\RR^n$, endowed with the supremum norm $\| f\|_{C^0_x}=\sup_{x\in\RR^\ell} |f(x)|$. Similarly, for $k\geq 1$, we denote by $C^k_x$ the Banach space of continuous functions $f$, with continuous and bounded derivatives up to order $k$, with norm
	\begin{equation*}
		\|f\|_{C^k_x} \coloneqq \max_{l\leq k}\|D^l_x f\|_{C^0_x}.
	\end{equation*}
We denote by $C^\infty_b=C^\infty_b(\RR^\ell;\RR^n)$ the space of infinitely differentiable functions $f:\RR^\ell\to\RR^n$, with all bounded derivatives; similarly for $C^\infty_b([0,T]\times \RR^\ell;\RR^n)$.

For $\alpha\in\RR$, we denote by $B^\alpha_{\infty,\infty}=B^\alpha_{\infty,\infty}(\RR^\ell;\RR^n)$ the inhomogeneous Besov--H\"older space, as defined for instance in \cite[Def.~2.6.7]{bahouri_chemin_danchin_11}. For $\alpha\in (0,+\infty)\setminus\NN$, $B^\alpha_{\infty,\infty}$ corresponds to the more classical H\"older space $C^\alpha_x=C^\alpha_b(\RR^\ell;\RR^n)$ of bounded functions $f$ whose derivatives of order $|l|\leq \lfloor\alpha\rfloor$ are bounded and $\{\alpha\}$-H\"older continuous; here $\lfloor\cdot\rfloor$ and $\{\cdot\}$ respectively stand for integer and fractional parts. However, for $\alpha=k\in\NN_{\geq 0}$, only the strict embedding $C^k_x \subsetneq B^k_{\infty,\infty}$ holds (cf. \cite[p.~99]{bahouri_chemin_danchin_11}). 

For $\alpha\in (-\infty,0)$ and $E$ a Banach space, we define the Banach space $\mcC^\alpha (\RR^\ell;E)\coloneqq \mcC^\alpha_x E$ as 
	\begin{equation}\label{eq:defn_cC_alpha}
		\mcC^\alpha_xE =	\{
			\text{ Closure of $C^\infty_b(\mbR^\ell;E)$ under the norm } \|\,\cdot\,\|_{B^\alpha_{\infty,\infty}(\mbR^\ell;E)}\}.
	\end{equation}
	As a consequence of the definition and Besov embeddings, for any $\alpha\in\RR$, $\delta>0$ and $n\in \mbN$ one has 
\begin{equation}\label{eq:holder_complete_embed}
		\mcC^\alpha_x\mbR^n\hookrightarrow B^\alpha_{\infty,\infty}(\mbR^\ell;\mbR^n) \hookrightarrow \mcC^{\alpha-\delta}_x\mbR^n.
	\end{equation}
We let $L^\infty_x = L^\infty(\mbR^\ell;\mbR^n)$ denote the space of essentially bounded functions $f:\mbR^\ell \to \mbR^j$, equipped with the usual norm $\|f\|_{L^\infty_x} \coloneqq \esssup_{x\in \mbR^{\ell}}|f(x)|$. $W^{1,\infty}_x = W^{1,\infty}(\mbR^\ell;\mbR^j)$ denotes the Sobolev space of essentially bounded functions with essentially bounded first, weak derivatives, equipped with the norm
	\begin{equation*}
		\|f\|_{W^{1,\infty}_x} = \|f\|_{L^\infty_x} + \|\nabla f\|_{L^\infty_x} \eqcolon \|f\|_{L^\infty_x} + \|f\|_{\dot W^{1,\infty}_x}.
	\end{equation*}
	Recall that $W^{1,\infty}_x$ is equivalent to the space of bounded, Lipschitz continuous functions, and the optimal Lipschitz constant of $f$ is given by $\|f\|_{\dot W^{1,\infty}_x}$.
	
    For a Banach space $E$, we use $C([0,T];E) = C_T E$ to denote the space of all continuous maps $f:[0,T]\to E$, equipped with the usual supremum norm. When $E = \mbR^\ell$ and there is no cause of confusion, we suppress it and simply write $L^q_T,\, C_T$ in place of $L^q_T  \mbR^\ell,\, C_T  \mbR^\ell$.

Given a Banach space $E$ and an interval $\mcI\subseteq \mbR_+$, we say that a process $\varphi: \mcI \times \mbR^d \to E$ is uniformly $\beta$-H\"older continuous on $\mcI$ if
\begin{equation}\label{eq:uniform_holder_moment}	\llbracket\varphi\rrbracket_{\mcC^{\beta}_{\mcI} E} \coloneqq \sup_{\substack{(s,t)\in \mcI^2_\leq\\ |t-s|\leq 1}} \frac{ \norm{ \varphi_u - \varphi_s }_{E} }{ \absv{u-s}^{\beta}}<\infty
\end{equation}
and define the norm
\begin{equation*}
    \|\varphi\|_{\mcC^\beta_{\mcI}E} \, \coloneqq \|\varphi\|_{L^\infty_{\mcI}E} + \llbracket\varphi\rrbracket_{\mcC^{\beta}_{\mcI} E}.
\end{equation*}
Note that for $0<\beta'<\beta<1$ one has
\begin{equation}\label{eq:global_holder_ordering}  \llbracket\varphi\rrbracket_{\mcC^{\beta'}_{\mcI} E}\,  \leq \llbracket\varphi\rrbracket_{\mcC^{\beta}_{\mcI} E} \quad \text{and}\quad \|\varphi\|_{\mcC^{\beta'}_{\mcI} E} \, \leq \|\varphi\|_{\mcC^{\beta}_{\mcI} E}.
\end{equation}
We will primarily consider these spaces with $E= L^m(\Omega;\mbR^d)$ for some $m\geq 1$.

    For $f\in C^\infty(\mbR;\mbR)$  and $\alpha \in (0,1)$ we define the fractional integral and differential operators:
	\begin{align}
		\msI^\alpha f(t) = &\frac{1}{\Gamma(\alpha)}\int_0^t (t-s)^{\alpha-1} f(s) \dd s, \label{eq:fractional_integral}\\
		\msD^\alpha f(t) = & \frac{\dd}{\dd t} \msI^{1-\alpha} f(t) \label{eq:fractional_derivative}
	\end{align}
	referring \cite[Sec. 2]{Pic11} for a review focused on applications of this formalism to fBm. By default we extend these operators by density and duality to the space of all distributions and set
	\begin{equation*}
		\msI^0 f = f\quad \text{and}\quad \msD^0 f =f.
	\end{equation*}
	Moreover, we have the following semi-group property for $\alpha, \beta \in (0,1)$, namely $\msI^{\alpha} \msI^{\beta} = \msI^{\alpha+\beta}$ (e.g. \cite[Eq. (19)]{Pic11}, proved in Thm. 5 therein). In Section~\ref{subsec:irreduc} we require an extension of these operators to the regime $\alpha \in (-1,0)$, which is also standard and described in \cite{Pic11}. We postpone a discussion of them and their properties until that section. 
\subsection{Stochastic Dynamical Systems}\label{sec:sds}
Our main ergodic results rely on a version of the Doob–Khas’minskii theorem obtained by \cite{hairer_ohashi_07_ergodic} for a class of non-Markovian SDEs, \cite[Thm.~3.10]{hairer_ohashi_07_ergodic}. The requisite framework for this result is that of quasi-Markovian stochastic dynamical systems (SDS). For the sake of completeness, we briefly summarize these notions and some of their important properties. A fuller account is given by the works \cite{Hai05,hairer_ohashi_07_ergodic}.

The first object required is that of a stationary noise process.
\begin{defn}[Stationary Noise Process]\label{def:stat_noise}
    A quadruple $(\mcW,\{\msP_t\}_{t\geq 0}, \bP_\omega, \{\theta_t\}_{t\geq 0})$ is called a \emph{stationary noise process} if
    \begin{enumerate}[label=\roman*)]
        \item $\mcW$ is a Polish space with Borel $\sigma$-algebra $\msW$;
        \item $\{\msP_t\}_{t\geq 0}$ is a Feller transition semi-group on $\mcW$ for which $\bP_\omega$ is its unique invariant measure;
        \item the family $\{\theta_t\}_{t\geq 0}$ is a semiflow of measurable maps on $\mcW$ such that for all $w\in \mcW$ and $t>0$,
        \begin{equation}\label{eq:stat_noise_flow_return}
            \theta_t^\ast\msP_t(w,\,\cdot\,) = \delta_w. 
        \end{equation}
    \end{enumerate}
\end{defn}
A \emph{stochastic dynamical system} is built atop a stationary noise process and is the fundamental object studied in \cite{Hai05,hairer_ohashi_07_ergodic}. We make a slight adjustment to the definition given in \cite{hairer_ohashi_07_ergodic} by only requiring that the flow map be continuous in the noise argument at points contained in the support of the stationary noise process.

\begin{defn}[Stochastic Dynamical System]\label{def:sds}
    A \emph{continuous stochastic dynamical system (SDS)} on a Polish space $\mcX$ (with Borel $\sigma$-algebra $\msX$) over a stationary noise process $(\mcW,\{\msP_t\}_{t\geq 0}, \bP_\omega, \{\theta_t\}_{t\geq 0})$ is a map
    \begin{align}
        \Lambda: \mbR_+ \times \mcX \times \mcW &\to \mcX\\
        (t,x,w) &\mapsto \Lambda_t(x,w),
    \end{align}
which enjoys the following properties.
\begin{enumerate}[label=\roman*)]
    \item \label{it:sds_path_reg}\emph{Regularity of paths:} For every $T>0$, $x\in \mcX$ and $w\in \mcW$, the map 
    \begin{align}
        \Phi_T(x,w):[0,T]&\to \mcX,\\
        t&\mapsto \Phi_T(x,w)(t) \coloneqq  \Lambda_t(x,\theta_{T-t}w),
    \end{align}
is continuous, i.e it belongs to the space $C([0,T];\mcX)$.
    \item \label{it:sds_continuous} \emph{Continuous dependence:} The map
    \begin{equation}
          \mcX\times \mcW \ni(x,w) \to \Phi_T(x,w)\in C([0,T];\mcW), 
    \end{equation}
    is jointly continuous at every $(x,w)\in \mcX\times \supp(\bP_\omega)\subset \mcX\times \mcW$ with respect to the natural topologies, for every $T>0$.
    \item \label{it:sds_cocycle} \emph{Cocyle property:} The family of maps $\{\Lambda_t\}_{t\geq 0}$ is such that for all $s,\,t>0$, $x\in \mcX$ and $w\in \mcW$ one has
    \begin{align}
        \Lambda_0(x,w) &= x,\\
        \Lambda_{s+t}(x,w) &= \Lambda_s(\Lambda_t(x,\theta_sw),w).
    \end{align}
\end{enumerate}
\end{defn}
\begin{rem}\label{rem:flow_map_continuity}
	As mentioned above, we slightly modify the definition of a stochastic dynamical system built atop a stationary noise process from that given in \cite{hairer_ohashi_07_ergodic}. The only change we make is to slightly relax condition \ref{it:sds_continuous} to only require that the flow map be continuous in its noise argument around all points in the support of the stationary noise process invariant measure. Since we are concerned with singular equations where statistical properties of the noise are required to obtain well-posedness of solutions we cannot expect to have a globally continuous flow map in the noise argument. However, as we shall see, this restriction has no impact on the ergodic results of \cite{hairer_ohashi_07_ergodic}.
\end{rem}
Given an SDS $\Lambda$ and a generalised initial condition $\mu \in \mcM_\Lambda$ (see \cite[Def.~2.3]{hairer_ohashi_07_ergodic}) we construct a natural $\mcX$-valued stochastic process induced by $\Lambda$. Given $t \geq 0$ and $(x,w)\in \mcX \times \mcW$ we define a new measure $\mcQ_t(x,w;\,\cdot\,) \in \mcM_1(\mcX\times \mcW)$ by setting
\begin{equation}\label{eq:Q_semi_group}
\mcQ_t(x,w;A \times B) \coloneqq \int_B \delta_{\Lambda_t(x,w')}(A) \msP_t(w,\dd w'),
\end{equation}
where $A \in \msX$ and $B\in \msW$. As usual we define the action of $\{\mcQ_t\}_{t\geq 0}$ on measures by setting
\begin{equation}
    \mcQ_t\mu(A\times B) \coloneqq \int_{\mcX \times \mcW} \mcQ_t(x,w;A\times B) \mu(\dd x, \dd w),\quad \text{for all }\, A\times B \in \msX\times \msW.
\end{equation}
From \cite[Lem.~2.12]{Hai05} we see that the family $\{\mcQ_t\}_{t\geq0}$ defines a Feller transition semi-group on $\mcX\times \mcW$ and given $\mu \in \mcM_\Lambda$ (i.e. $\Pi^\ast_\mcW \mu =\bP_\omega$) then $\Pi^\ast_\mcW \mcQ_t\mu = \bP_\omega$ for all $t>0$.
Given a generalised initial condition $\mu\in \mcM_\Lambda$, we define a stochastic process $\mbR_+\ni t\mapsto (x_t,w_t) \in \mcX\times \mcW$ by taking it to be the canonical process under the evolution $\mbR_+\ni t\mapsto \mcQ_t\mu$. The marginal process $t\mapsto x_t$ whose law is given by the measure $t\mapsto \Pi^*_\mcX \mcQ_t \mu $ is referred to as the \emph{process generated by $\Lambda$}. We let $\bar{\mcQ}\mu \in \mcM_1(C(\mbR_+;\mcX))$ denote the law of this process, which naturally induces a map
\begin{equation}\label{eq:Q_bar_def}
  \mcM_1(\mcX \times \mcW) \ni \mu \mapsto \bar{\mcQ}\mu \in \mcM_1(C(\mbR_+;\mcX)).
\end{equation}
\begin{defn}[Invariant and Ergodic Measures]\label{def:invariant_ergodic_measures}
Given an SDS $\Lambda$, we say that a generalised initial condition $\mu \in \mcM_\Lambda$ is invariant for $\Lambda$ if
\begin{equation}
    \mcQ_t \mu = \mu \quad \text{for all }\, t\geq 0.
\end{equation}
    We say that $\mu \in \mcM_\Lambda$ is ergodic for $\Lambda$ if the law of the stationary Markov process on $\mcX\times \mcW$ with transition probabilities $\mcQ_t$ and marginal law at each $t\geq 0$ equal to $\mu$ is ergodic for the shift map on $C(\mbR_+;\mcX\times \mcW)$. 
\end{defn}
The following result, \cite[Lem.~2.5]{hairer_ohashi_07_ergodic}, is key to interpreting ergodic results concerning SDS in terms of ergodicity for the \emph{process generated by $\Lambda$}.
It follows from \cite[Lem.~2.5]{hairer_ohashi_07_ergodic} that if $\mu,\,\nu \in \mcM_\Lambda$ are both ergodic invariant measures for $\Lambda$ then either $\bar{\mcQ}\mu = \bar{\mcQ}\nu$ or $\bar{\mcQ}\mu $ and $\bar{\mcQ}\nu$ are mutually singular, see \cite[Rem.~2.6]{hairer_ohashi_07_ergodic}.

Given a map $f:\mbR_+\to \mcE$ for some Polish space $\mcE$ and $t\geq 0$ we define the restriction map
\begin{equation}
    R_t f \coloneqq f|_{[t,+\infty)}.
\end{equation}
In line with our modification to the definition of an SDS we also modify the notion of strong Feller SDS, topological irreducibility (the latter we do not spell out in detail).
\begin{defn}[Stochastic Strong Feller]\label{def:strong_feller}
    We say that an SDS $\Lambda$ is \emph{strong Feller at time $t\geq 0$} if there exists a map $\ell : \mcX^2 \times \mcW \to \mbR_+$ which is jointly continuous at every point $(x,y,w)\in \mcX^2\times \supp(\bP_\omega) \subset \mcX^2 \times \mcW$,  such that $\ell(x,x,w)=0$ for every $x\in \mcX$, $\bP_\omega$-almost every $w\in \mcW$ and
    \begin{equation}
        \left\|R^\ast_t \bar{\mcQ}\delta_{x,w} - R^\ast \bar{\mcQ} \delta_{y,w} \right\|_{\TV} \leq \ell(x,y,w), \quad \text{for all }\, x,\,y \in \mcX\,\, \text{and } \bP_w\text{-a.e.}\,\, w\in \mcW.
    \end{equation}
\end{defn}
\begin{defn}[Topologically Irreducible]\label{def:irreducible}
    We say that an SDS $\Lambda$ is  \emph{topologically irreducible at time $t\geq 0$} if for every $x\in \mcX$, $\bP_\omega$-a.s. $w\in \mcW$ and every non-empty open set $U \subset \mcX$, one has $\mcQ_t(x,w;U\times \mcW) >0$.
\end{defn}
We have a version of \cite[Thm.~3.10]{hairer_ohashi_07_ergodic} under our slightly relaxed notions of an SDS, the strong Feller property, topological irreducibility and quasi-Markovianity (\cite[Def.~3.7]{hairer_ohashi_07_ergodic}). Since the proof is almost identical to that given in \cite{hairer_ohashi_07_ergodic} we postpone its sketch to Appendix~\ref{app:doob_khasminski_proof}.
\begin{thm}\label{th:doob_khasminskii}
	If there exist times $s>0$ and $t>0$ such that a quasi-Markovian SDS $\Lambda$ (in the sense of \cite[Def.~3.7]{hairer_ohashi_07_ergodic} with the map $w\mapsto \mcP_s^{V,U}(w,\,\cdot\,)$ only defined on $\supp(\bP_\omega)$) is strong Feller (in the sense of Definition \ref{def:strong_feller}) at time $t$ and topologically irreducible (in the sense of Definition~\ref{def:irreducible}) at time $s$ then $\Lambda$ can have at most one invariant measure (in the sense of Definition~\ref{def:invariant_ergodic_measures}) up to the equivalence relation of \cite[Def.~2.4]{hairer_ohashi_07_ergodic}.
\end{thm}
\begin{proof}
    See Appendix.~\ref{app:doob_khasminski_proof}.
\end{proof}
\subsection{Construction and Properties of the Stationary Noise Process}\label{subsec:sds_construction}
We fix a probability space $(\Omega,\mcF,\mbP)$ carrying a two sided standard Brownian motion $B:(-\infty,+\infty)\to \mbR^d$ for some $d\geq 1$. Then, given $H\in (0,1)$, we define the two sided $
\mbR^d$ valued fBm compontentwise by setting\footnote{Note that by change of variables this is equal to the expression $t\mapsto W^H_t \coloneqq c_H \int_{-\infty}^0 (-r)^{H-\nicefrac{1}{2}}\left(\dd B_{t+r} - \dd B_r\right)$, see for comparison  \cite[Eq.~(4.1)]{hairer_ohashi_07_ergodic}.}
\begin{equation}\label{eq:fbm_mvn}
   \mbR \ni t\mapsto W^H_t \coloneqq c_H \int_0^t (t-r)_+^{H-1/2} \dd B_r + c_H \int_{-\infty}^0 (t-r)_+^{H-1/2} - (-r)_+^{H-1/2} \dd B_r,
\end{equation}
which is a Gaussian process $W^H: \Omega \times (-\infty,+\infty)\to \mbR^d$ and is such that $W^H_0 = 0$ and $(\,\cdot\,)_+ = \max(0, \cdot)$

The parameter $c_H >0$ is chosen such that 
\begin{equation}\label{eq:fbm_covariance}
    \mbE\left[W^H_t W^H_s\right] = \frac{1}{2}\left(|t|^{2H} + |s|^{2H} - |t-s|^{2H}\right)I_d.
\end{equation}
Note that \eqref{eq:fbm_mvn} is often referred to as a non-canonical representation of the fBm since its canonical filtration $\mbF^H\coloneqq (\mcF^H_t)_{t\geq 0} \coloneqq \left(\sigma \left\{ W^H_s\,:\, s\leq t \right\}\right)_{t\geq 0}$ is strictly smaller than the filtration generated by the two sided Brownian motion $\mbF^B\coloneqq (\mcF^B_t)_{t\in \mbR} \coloneqq \left(\sigma\{B_s\,:\, s\leq t \} \right)_{t\in \mbR}$.

It follows from Gaussianity of the fBm, its explicit covariance structure and an exponential version of Kolmogorov's continuity criterion (e.g. \cite[Lem.~A.2]{galeati_gerencser_22_subcritical}), that there exists a $a_0\coloneqq a_0(H,d,\eps^{-1}) >0$ such that for all $a\in (0,a_0)$,
\begin{equation}\label{eq:fbm_gaussian_moments}
    \sup_{t\in \mbR_+}\mbE\left[ \exp\left(a \left\llbracket W^H\right\rrbracket^2_{\mcC^{H-\eps}_{[t,t+1]}}\right)\right] <\infty.
\end{equation}
To define the stationary noise process we introduce the space of paths $\mcW^-_{\gamma,\delta}$, for any pair $(\gamma, \delta)\in (0,H)\times (0,1-\gamma)$, as the completion of $C^{\infty}_0(\mbR_-,\mbR^d)$ under the norm
\begin{equation}\label{eq:w_gamma_delta_norm}
    \|w\|_{\gamma,\delta} = \sup_{t,\,s\in \mbR_-} \frac{|w_t-w_s|}{|t-s|^\gamma(1+|t|+|s|)^\delta}.
\end{equation}
Note that the norm defined by \eqref{eq:w_gamma_delta_norm} captures both local regularity of paths and allows for some growth at infinity. We define the following related sets
\begin{equation}\label{eq:w_tilde_gamma_delta}
    \mcW^+_{\gamma,\delta} = \Big\{ w:\mbR_+\to \mbR^d\,:\, w_{-t} \in \mcW^-_{\gamma,\delta}\Big\},
\end{equation}
and for $T>0$
\begin{equation}\label{eq:w_gamma_delta_T}
    \mcW^-_{T;\gamma,\delta} = \Big\{ w:\mbR_-\to \mbR^d\,:\, w|_{[-T,0]} \in \mcW^-_{\gamma,\delta}\Big\},
\end{equation}
\begin{equation}\label{eq:w_tilde_gamma_delta_T}
    \mcW^+_{T;\gamma,\delta} = \Big\{ w:\mbR_+\to \mbR^d\,:\, w|_{[0,T]} \in \mcW^+_{\gamma,\delta}\Big\}.
\end{equation}
The following lemma is obtained as a combination of \cite[Lem.~3.5, Lem.~3.8]{Hai05} and \cite[Lem.~4.1]{hairer_ohashi_07_ergodic}.
\begin{lem}\label{lem:w_gamma_delta_wiener}
    For any $\gamma,\,\delta>0$ the space $\mcW^-_{\gamma,\delta}$ is separable. Furthermore, given $H\in (0,1)$, $\gamma \in (0,H)$ and $\delta \in (H-\gamma,1-\gamma)$, there exists a Borel probability measure $\bP^-_H$ on $\mcW^-_{\gamma,\delta}$ such that the canonical process associated to $\bP^-_H$ is an $\mbR^d$ valued fBm with Hurst parameter $H$ on $\mbR_-$. In addition, there is a Borel measure $\bP_H$ on $\mcW^-_{\gamma,\delta}\times \mcW^+_{\gamma,\delta}$ whose canonical process is the two sided fBm $W^H$.
\end{lem}

\begin{proof}
Separability follows from the same argument given in the proof of \cite[Lem.~3.5]{Hai05}. That is, since we define $\mcW^-_{\gamma,\delta}$ as completion under the given norm it suffices to use the fact that $\|w\|_{\star}\coloneqq \sup_{t<0}|t\dot{w}_t| \geq \|w\|_{\gamma,\delta}$ for all $\gamma,\,\delta \in (0,1)$ and the completion of $C^\infty_0(\mbR_-,\mbR^d)$ under $\|\,\cdot\,\|_{\star}$ is itself separable.

\noindent The remaining claims of the lemma follow in exactly the same manner as the proof of \cite[Lem.~4.1]{hairer_ohashi_07_ergodic} and following arguments of \cite[Lem.~3.8]{Hai05}
\end{proof}
As in \cite{hairer_ohashi_07_ergodic} and \cite{hairer_pillai_11_regularity} we define the operator 
\begin{equation}\label{eq:A_op_def_1}
   \mbR_+\ni t\mapsto  ( \msA w)_t \coloneqq (H-\nicefrac{1}{2})c_H c_{1-H} \int_0^\infty \frac{1}{r}\, f\bigg(\frac{t}{r}\bigg) w_{-r}\dd r,
\end{equation}
with 
\begin{equation}\label{eq:A_op_def_2}
    f(x) \coloneqq x^{H-\nicefrac{1}{2}}+ (H-\nicefrac{3}{2})x \int_0^1 \frac{(u+x)^{H-\nicefrac{5}{2}}}{(1-u)^{H-\nicefrac{1}{2}}
    }\dd u.
\end{equation}
We refer to Appendix~\ref{app:A_operator_properties} for some useful analytical properties of $\mcA$. As described by \cite[Eq.~(4.2)]{hairer_pillai_11_regularity} and \cite[Lem.~4.2]{hairer_ohashi_07_ergodic} the operator $\mcA$ allows us to disintegrate the measure $\bP_H$ over $\mcW^-_{\gamma,\delta}\times \mcW^+_{\gamma,\delta}$ with respect to $\bP^-_H$. More prosaically, given $w^- :\mbR_- \to \mbR^d$, the variable $(\msA w^-)_t$ coincides with the conditional expectation of a two sided fBm, conditioned to agree with $w^-$ on $\mbR_-$.  Given $h : \mbR \to \mbR^d$, we define the map
\begin{align}
 \mcW^+_{\gamma,\delta}  \ni w^+\mapsto  \tau_h w^+ \coloneqq  w^+ +h.
\end{align}
The following is an analogue of \cite[Lem.~4.2]{hairer_ohashi_07_ergodic}.
\begin{lem}\label{lem:stat_noise_measure_disintigration}
    Given $H\in (0,\nicefrac{1}{2})$, define the transition kernel $ \mcW^-_{\gamma,\delta} \ni w \mapsto \msH(w,\,\cdot\,) \in \msM_1\big(\mcW^+_{\gamma,\delta}\big)$ by setting
    \begin{equation}\label{eq:disintegration_measure}
        \msH(w,\,\cdot\,) \coloneqq \left(\tau_{\msA w} \circ \msD^{\nicefrac{1}{2}-H}\right)^*\bW,
    \end{equation}
    where $\bW$ is the Wiener measure over $C(\mbR_+;\mbR^d)$ and we recall the definition of $\msD^\alpha$ from \eqref{eq:fractional_derivative}. Then $\msH$ is the disintegration of $\bP_H$ over $\mcW^{-}_{\gamma,\delta}\times \mcW^{+}_{\gamma,\delta}$ with respect to $\bP^-_H$.
\end{lem}
\begin{proof}[Sketch of Proof]
This follows from the identity 
\begin{equation}\label{eq:w_+_disintigration}
 \mbR_+  \ni t\mapsto   w^+_t = (\msA w^-)_t + c_H (\msD^{\nicefrac{1}{2}-H} B^+)_t,
\end{equation}
 where $B^+$ is a standard Wiener process on $\mbR_+$, $w^- \in\mcW^-_{\gamma,\delta}$ and $w^+$ is an fBm conditioned to equal $w^-$ for $t<0$. Identity \eqref{eq:w_+_disintigration} can be obtained from the Mandelbrot van-Ness representation of the fBm \eqref{eq:fbm_mvn}, see \cite[Eq.~4.2]{hairer_pillai_11_regularity}.
\end{proof}
We are now able to define our stationary noise process, following the same procedure as outlined in \cite[Sec.~4]{hairer_ohashi_07_ergodic}. For $w \in \mcW^-_{\gamma,\delta}$, $t\in \mbR_+$ and $s\in \mbR_-$, we define the map
\begin{equation}
    \theta_t w_s \coloneqq w_{s-t} - w_{-t}
\end{equation}
and the \emph{concatenation function} $M: \mbR_+ \times \mcW^-_{\gamma,\delta} \times \mcW^+_{\gamma,\delta} \to \mcW^-_{\gamma,\delta}$
\begin{equation}
    M_t(w^-,w^+) \coloneqq \begin{cases}
        w^+_{t+s}-w^+_t, & \text{if }\, -t <s,\\
        w^-_{t+s} - w^+_t, & \text{if }\, s\leq -t \leq 0.
    \end{cases}
\end{equation}
We then set
\begin{equation}
    \msP_t(w^-,\,\cdot\,) \coloneqq M_t(w^-,\,\cdot\,)^* \msH(w^-,\,\cdot\, ),\quad \text{for all }\, w^- \in \mcW^-_{\gamma,\delta},\quad t\in \mbR_+.
\end{equation}
We refer to \cite[Fig.~1]{neamtu_varzaneh_25_negativity} for an illustration of this construction. The following is the equivalent, in our setting, of \cite[Lem.~4.3]{hairer_ohashi_07_ergodic}.
\begin{lem}\label{lem:stationary_noise}
    The tuple $(\mcW^-_{\gamma,\delta},\{\msP_t\}_{t\in \mbR_+}, \bP_H,\{\theta_t\}_{t\in \mbR_+})$ defines a stationary noise process in the sense of Definition~\ref{def:stat_noise}.
\end{lem}
\begin{proof}
    The proof is identical to that of \cite[Lem.~4.3]{hairer_ohashi_07_ergodic}, only replacing \cite[Prop.~A.2]{hairer_ohashi_07_ergodic} with our version Proposition~\ref{prop:A_map_bounded}.
\end{proof}
For every $T>0$, let us define the shift map $\msR_T :\mcW^-_{\gamma,\delta} \to \mcW^+_{\gamma,\delta}$ such that
\begin{equation}\label{eq:negative_to_positive_flip}
	\msR_T f(\,\cdot\,) = f(-T) - f(\,\cdot\, - T).
\end{equation}
\begin{cor}\label{cor:negative_to_positive_fbm}
   Given the tuple $(\mcW^-_{\gamma,\delta},\{\msP_t\}_{t\in \mbR_+}, \bP_H,\{\theta_t\}_{t\in \mbR_+})$ as in Lemma~\ref{lem:stationary_noise} and for any $T>0$ the map $\msR_T$ defined by \eqref{eq:negative_to_positive_flip}, for any $w \sim \bP_H$ the process $[0,T] \ni t\mapsto \msR_t w(t)$ is an $\mbR^d$ valued fBm which is equal in law to the process given by \eqref{eq:fbm_mvn} with respect to $B$ a standard $\mbF^B$ adapted Brownian motion on a filtered probability space $(\Omega,\mcF,\mbP,\mbF^B)$. Moreover, in the decomposition of $W^H$ given by \eqref{eq:w_+_disintigration} the process $\{\mbR_+\ni t\mapsto (\msA w)_t\}$ ($w \sim \bP_H$) is $\mcF_0$-measurable.
\end{cor}
\begin{proof}
    The proof is direct since $w \sim \bP_H$ is a time reversed fBm, $\msR_t$ is a continuous map and $w(0)= 0$ by construction. The second claim is immediate from the fact that $\msA w = \int_{-\infty}^0 (t-r)^{H-1/2} - (-r)^{H-1/2} \dd B_r$ which is a measurable map of $\{\mbR_{-} \ni t\mapsto B_t\}$.
\end{proof}
We require the Cameron--Martin spaces associated to $\bP_H^{-}$, $\bP_H^{+}$ and $\bP_H$. To this end, for $H\in (0,\nicefrac{1}{2})$ we define
\begin{align}
	\mfH_H &\coloneqq \{ h:\mbR \to \mbR^d\,:\, \mfh = \msD^{\nicefrac{1}{2}-H} \tilde{\mfh},\, \tilde{\mfh}\in \mcH^1(\mbR;\mbR)\},\label{eq:cameron_martin_two_sided}\\
	\mfH^+_{H} & \coloneqq \{ \mfh:\mbR_+ \to \mbR^d\,:\, \mfh = \msD^{\nicefrac{1}{2}-H} \tilde{\mfh},\, \tilde{\mfh}\in \mcH^1(\mbR_+;\mbR^d)\},\label{eq:cameron_martin_positive}\\
	\mfH^-_{H} & \coloneqq \{ \mfh:\mbR_- \to \mbR^d\,:\, \mfh = \msD^{\nicefrac{1}{2}-H} \tilde{\mfh},\, \tilde{\mfh}\in \mcH^1(\mbR_-;\mbR^d)\}. \label{eq:cameron_martin_negative}
\end{align}
where $\mcH^k(\mbR;\mbR^d) =W^{k,2}(\mbR;\mbR^d)$ denotes the usual Sobolev spaces of square integrable maps with square integrable distributional derivatives. To see that these are correct characterizations of the Cameron--Martin spaces see \cite[Thm.~19]{Pic11}. 
It will be necessary for us to consider an fBm $W^H$ conditioned to equal some $w\in \mcW^{-}_{(\gamma,\delta)}$ on $\mbR_-$. Using the decomposition \eqref{eq:w_+_disintigration} we define
\begin{equation}\label{eq:restricted_w_h}
    \mbR_+ \ni t\mapsto (W^H|_w)_t \coloneqq  c_H \int_0^t (t-r)^{H-\nicefrac{1}{2}}_+ \dd B_t + (\msA w)_t \eqqcolon \tilde{W}^H_t + (\msA w)_t.
\end{equation}
Note that $W^H|_w$ is the canonical process associated to the measure $\msH(w,\,\cdot\,)$ and $\tilde{W}^H$ is proportional a Riemann-Liouville process (see \cite[Sec.~5]{Pic11}).

Then, given an $\mbR^d$ valued random variable $X:\Omega\to \mbR^d$ and $w\in \mcW^-_{\gamma,\delta}$ we set 
\begin{equation}\label{eq:disintegrated_expectation}
\E_{w}\left[ X \right] \coloneqq \E\left[ X \lvert W^{H}\lvert_{\mbR_-} = w\right],
\end{equation}
where the right hand side should be understood as expectation with respect to the conditional distribution given by \eqref{eq:disintegration_measure}. In addition, for any $s\in \mbR_+$, by a slight abuse of notation we set
\begin{equation}\label{eq:disintegrated_cond_expectation}
    \mbE^B_{w,s} \coloneqq \mbE\left[X|\mcF^B_s, W^H|_{\mbR_-} = w\right],
\end{equation}
which should be understood as the random variable $\mbE \left[ X \lvert \mcF^B_s \right]$ restricted to the event $W^{H}\lvert_{\mbR_-} = w$. Moreover, for $m\in [1,+\infty)$ and $s\in \mbR_+$, we define the conditional moments
\begin{equation*}
   \norm{ X }_{\llmw} = \left( \E_{w}\left[ |X|^m \right] \right)^{1/m}  \quad \text{and}\quad \norm{ X }_{\llmws} = \left( \E^B_{w,s}\left[ |X|^m \right] \right)^{1/m}.
\end{equation*} 

Analogously we define conditional variances 
\[ \text{Var}_w[X] = \mbE_w \left[\left( X - \mbE_w  \left[ X\right] \right)^2 \right] \quad \text{and}\quad \text{Var}_w[X \lvert \mcF_s ] = \mbE_{w,s} \left[\left( X - \mbE_{w,s}  \left[ X\right] \right)^2 \right] \]
where $\mbE_w, \mbE_{w,s}$ are as in \eqref{eq:disintegrated_expectation} and \eqref{eq:disintegrated_cond_expectation}.

A key property of the fBm is its local non-determinism, in particular we have the following lemma. For a similar statement see for example \cite[Eq.~(3.6)]{galeati_gubinelli_20_Noiseless}, which i.
\begin{lem}\label{lem:lnd}
Let $(\Omega,\mcF,\mbP,\mbF^B)$ be a filtered probability space carrying a standard $\mbF^B$ adapted Brownian motion $B$, $H\in (0,1)$, $W^H$ be fBm given by \eqref{eq:fbm_mvn} with natural filtration $\mbF^H$. Then, for any $0\leq s <t <+\infty$, it holds that
    \begin{equation}\label{eq:fbm_lnd}
        {\rm Var} (W^H_t|\mcF^H_s) \land {\rm Var}_w (\wh_t | \mcF^H_s) \geq \frac{c^2_H}{2H} |t-s|^{2H} I_d \eqqcolon \tilde{c}_H|t-s|^{2H} I_d.
    \end{equation}
    Moreover, for $\psi :\mbR_+\to \mbR$  continuous and deterministic,  the estimate \eqref{eq:fbm_lnd} holds without modification if $W^H$ is replaced by $W^H + \psi$.
\end{lem}
\begin{proof}
Since $\wh$ is identified with the right hand side of \eqref{eq:w_+_disintigration}, using \eqref{eq:restricted_w_h}, for any $0\leq s< t<+\infty$ and $w \in \mcW^{-}_{\gamma,\delta}$, we have
\begin{equation}\label{eq:wh_decomp_with_awop}
      (\wh|_{w})_t = W^{H;1}_{s,t} + W^{H;2}_{s,t} + (\msA w)_t - (\msA w)_s,
  \end{equation}
  where
  \begin{equation*}
      W^{H;1}_{s,t} = c_H \int_s^t (t-r)^{H-\nicefrac{1}{2}}\dd B_r \quad \text{and}
      \quad W^{H;2}_{s,t} = \int_0^s \left( (t-r)_+^{H-1/2} - (-r)_+^{H-1/2} \right) \dd B_r .
  \end{equation*}
    Since $W^{H;1}_{s,t}$ is independent of $\mcF^B_s$, $W^{H;2}_{s,t}$ is $\mcF^B_s$ measurable and $\msA w$ is deterministic, using that $\mbF^H\subset \mbF^B$ and the standard It\^o isometry, for any $0\leq s<t<+\infty$ we have
  \begin{equation*}
      {\rm Var}_w (\wh_t|\mcF^H_s) \geq {\rm Var}_w (\wh_t|\mcF^B_s) = {\rm Var}_w (W^{H;1}_{s,t}) = \frac{c_H^2}{2H} |t-s|^{2H}I_d.
  \end{equation*}
  For the full variance, ${\rm Var}$, without conditioning on $w$, the same computation holds, except that one can refer to \eqref{eq:fbm_mvn}, therefore replacing the increment of $(\msA w)$ with an integral similar to the first term in $W^{H;2}$, but from $-\infty$ to 0, see also \cite[Eq.~(3.6)]{galeati_gubinelli_20_Noiseless}. The proof of the last statement follows by the same steps and observation that ${\rm Var}(W^{H;1}_{s,t} + \psi_{s,t}) = {\rm Var}(W^{H;1}_{s,t})$.
\end{proof}
\begin{rem}\label{rem:H_implies_B_measurable}
    Note that since we have the inclusion of $\sigma$-algebras $\mbF^H \subset \mbF^B$, if $Z$ is an $\mcF^H_s$ measurable random variable for some $s\geq 0$ then it is also $\mcF^B_s$ measurable.
\end{rem}
We introduce some compact notation, for $H\in (0,1)$, $t\in \mbR_+$ and  $f:\mbR^d \to \mbR^{d}$, we set
\begin{equation*}
    \mcG^H_{t} f \coloneqq G_{\tilde{c}_H t^{2H}} f \quad \text{where} \quad G_t f = g_t\ast f \quad \text{and} \quad g_t(x)=(2\pi t)^{-\frac{d}{2}} e^{-\frac{|x|^2}{2t}}.
\end{equation*}
It follows (for example from \cite[Prop.~5]{mourrat_weber_17_GWP} or \cite[Lem.~A.3]{ABLM24}) that for any $\eta < \kappa$ there exists a constant $C\coloneqq C(\eta,\kappa,d)>0$ such that
\begin{equation}\label{eq:heat_kernel_fbm}
\norm{ \mcG^H_t f }_{B^{\kappa}_{\infty,\infty}} \leq C t^{-H\left(\kappa-\eta\right)} \norm{ f }_{B^{\eta}_{\infty,\infty}}.
\end{equation}
In our case, the semi-group action above enters through the conditional expectation, which is formalized in the following result.

\begin{prop}\label{prop:condexp_heat_kernel}
    Let $0<s<t<+\infty$, $f \in C^{\infty}_b(\mbR^d;\mbR^d)$. Then, in the setting of Lemma~\ref{lem:lnd} it holds that 
\begin{equation}\label{eq:heat_kernel_fbm_identity}
     \mbE^B_s \left[f(\wh_t) \right]= (\mathcal{G}^H_{t-s} f)( \bar{W}^{H,2}_{s,t} )
    \end{equation},
   where $\bar{W}^{H,2}_{s,t} \coloneqq \int_{-\infty}^s \big( (t-r)_+^{H-1/2} - (-r)_+^{H-1/2} \big) \dd B_r$.
   
   \noindent Furthermore, for $w \in \mcW^-_{\gamma,\delta}$, $\mbE^B_{w,s}$ given by \eqref{eq:disintegrated_cond_expectation} and $W^{H,2}$ as in \eqref{eq:wh_decomp_with_awop}, it holds that
\begin{equation}\label{eq:disintegrated_heat_kernel_fbm}
    \mbE^B_{w,s} \left[f(W^{H,1}_{s,t} + W^{H,2}_{s,t} + (\msA w^-)_t - (\msA w^-)_s)\right] = (\mathcal{G}^H_{t-s}f)(W^{H,2}_{s,t} + (\msA w^-)_t - (\msA w^-)_s).
    \end{equation}
\end{prop}
\begin{proof}
   The proof of \eqref{eq:heat_kernel_fbm} is standard and follows from integrating out the part of $\wh$ that is independent from $\mcF_s$, using its Gaussianity, see \cite[Lem.~6.4]{le_20_stochSewing} for details. The identity \eqref{eq:disintegrated_heat_kernel_fbm} is proven in a similar way and it follows from recalling \eqref{eq:wh_decomp_with_awop} and using the fact every term inside $f$ on left hand side of \eqref{eq:disintegrated_heat_kernel_fbm} apart from $W^{H,1}_{s,t}$ is $\mcF^B_s$ and $\msH(w,\,\cdot\,)$ measurable.
\end{proof}
Fixing $(s,t)\in [0,+\infty)^2_\leq$ and an $\mcF^B_t$-measurable $\mbR^d$ valued Malliavin differentiable random variable  $Y$ on $(\Omega,\mcF,\mbF^B,\mbP)$ with Malliavin derivative $\mfD_{\,\cdot\,} Y$, we recall the Clark--Ocone formula (\cite[Prop.~1.3.14]{Nualart_malliavin_book}).
\begin{equation}\label{eq:clark_ocone}
    Y = \mbE^B_s[Y] + \int_s^t \mbE^B_r [\mfD_r Y] \cdot \dd B_r. 
\end{equation}
Moreover, we can prove that the formula \eqref{eq:clark_ocone} has a variant that holds under the disintegrated expectation of  where we recall the notation of \eqref{eq:disintegrated_cond_expectation}.
\begin{prop}
  Assume that we are in the setting of Lemma \ref{lem:lnd} and $w \in \supp(\bP^-_H)\subset \mcW^-_{\gamma,\delta}$. Then, for any $t\geq 0$ and an $\mcF^B_t$-measurable $\mbR^d$ valued Malliavin differentiable random variable $Y$, it holds that
\begin{equation}\label{eq:clark_ocone_disintegrated}
    Y = \mbE_{w,s}^B [Y] + \int_s^t \mbE_{w,r}^B [ \mfD_r Y_t ]\dd B_r, \quad \text{for all }\, s \in [0,t].
    \end{equation}
\end{prop}
\begin{proof}
One may follow the same steps as in the proof of \cite[Prop.~1.3.14]{Nualart_malliavin_book}. If $Y$ is a $\mcF^B_t$-measurable, then it means that for all $w$ defined in the support of $\bP^{-}_H$, it can be written as a functional of form $h(B|_{[0,t]}, w)$ under $\msH(w,\cdot)$. Then the proof follows mutatis mutandis, using the independence of increments of Brownian motion on the positive line from its past.
\end{proof}
\begin{rem}
To highlight the rather surprising equality between \eqref{eq:clark_ocone_disintegrated} and \eqref{eq:clark_ocone}, this only holds due to the assumed measurability of $Y$. In particular, this implies that we can write $Y = h(B|_{t\geq 0}, B|_{t< 0}$ for some measurable functional $h : C_{\mbR_+}\times C_{\mbR_-}\to \mbR^d$.
\end{rem}
Let $f \in C^{1}(\mbR^d)$, $x\in \mbR^d$, $(s,t)\in [0,+\infty)^2_{\leq}$ and recall that $\wh$ can be written by \eqref{eq:fbm_mvn} as 
\begin{equation}\label{eq:wh_as_rl_with_past}
\wh_t = \int_0^t (t-r)^{H-1/2} \dd B_r + P_t
\end{equation} 
for some $\mcF^B_0$-measurable path $P_t$. Then, for $x\in \mbR^d$ and $Y = f(x+ \int_0^t (t-r)^{H-1/2} \dd B_r + P_t )$, which is clearly $\mcF^B_t$-measurable, we have, by the chain rule of Malliavin derivatives and the fact that $\mfD_{\cdot} P = 0$:
\[ \mfD_{r} Y = \nabla f(x+\wh_t) (t-r)^{H-1/2} \]
 and as such one obtains
\begin{equation}\label{eq:malliavinfdx}
	 f(x+W^H_t) = \mbE^B_s[f(x+W^H_t)] + c_H \int_s^t (t-r)^{H-\nicefrac{1}{2}} \, \mbE^B_r \left[\nabla f(x+W^H_t)\right]  \dd B_r.
\end{equation}
We will now extend the preceding discussion to deal with the disintegrated case, i.e. when the past part of fBm is fixed.

The following lemma sums up the discussion above. 
\begin{lem}\label{lem:lnd_to_heat_kernel}
In the setting of Lemma~\ref{lem:lnd}, let $H\in (0,1)$, $(s,t)\in [0,+\infty)^2_{\leq}$, $f:\RR^d \to\RR^d$ be a bounded measurable function and $Z$ be an $\mbR^d$ valued $\cF^B_s$-measurable random variable. Then, for any $(s,t)\in [0,+\infty)^2_{\leq}$
\begin{equation}\label{eq:lnd_to_heat_kernel_stoch_integral}
     f(W^H_t+Z) = \mcG^H_{|t-s|} f\left(\mbE^B_s W^H_{t}+Z\right) + \int_s^t (t-r)^{H-\nicefrac{1}{2}}\nabla \mcG^H_{|r-s|} f(\mbE^B_r W^{H}_{t}+Z) \cdot \dd B_r
\end{equation}
and for $w\in \mcW^-_{\gamma,\delta}$ 
\begin{equation}\label{eq:disintegrated_lnd_to_heat_kernel_stoch_integral}
    f(\wh_t+Z) = \mcG^H_{|t-s|} f\left(\mbE^B_{w,s} W^H_{t}+Z\right) + \int_s^t (t-r)^{H-\nicefrac{1}{2}}\nabla \mcG^H_{|r-s|} f(\mbE^B_{w,r} W^{H}_{t}+Z) \cdot \dd B_r,
\end{equation}
   where one has $\mbE^B_{w,s} W^H_{t} = \mbE^B_s (\msI^{H-1/2} B)_t + \msA w_t$.
\end{lem}
\begin{proof} Recall that for $\bar{W}^{H,2}$ as in \eqref{eq:heat_kernel_fbm} we have $\bar{W}^{H,2}_{s,t} = \EE_s \wh_t$. The first identity follows almost directly from Lemma~\ref{lem:lnd}, see \cite[Lem.~3.10]{galeati_gubinelli_20_Noiseless}, or \eqref{eq:malliavinfdx} with Proposition \ref{prop:condexp_heat_kernel} applied to $f(\cdot +Z)$.

\noindent Analogously, for $W^{H,2}_{s,t}$ as in \eqref{eq:wh_decomp_with_awop} we recall that
\[ W^{H,2}_{s,t} + \left( (\msA w)_t - (\msA w)_s \right) = \mbE_{w,s} \wh_t. \] 
The proof of \eqref{eq:disintegrated_lnd_to_heat_kernel_stoch_integral} is then an application of \eqref{eq:clark_ocone_disintegrated} for $Y = f(\wh_t + Z)$, along with \eqref{eq:disintegrated_heat_kernel_fbm} and the fact that the Malliavin derivative of any measurable $P$ in \eqref{eq:wh_as_rl_with_past} is equal to zero, and $w$ is measurable. 
\end{proof}
\section{A Priori Estimates and Tightness}\label{sec:tightness}
We establish suitable a priori estimates, leading to Gaussian tightness, for solutions to the family of perturbed SDEs
\begin{equation}\label{eq:psi_sde}
    X^\psi_t = x + \int_0^t g(X^\psi_s)\dd s + \int_0^t u(X^\psi_s)\dd s + W^H_t + \psi_t,
\end{equation}
where $\psi :\mbR_+\to \mbR^d$ is a given continuous path and in places we will also condition $W^H$ to be equal to some $w\in \mcW^-_{\gamma,\delta}$ (recall \eqref{eq:w_gamma_delta_norm}) for $t\in \mbR_-$. Note that without this conditioning and for $\psi =0$ the dynamics of \eqref{eq:psi_sde} are nothing but \eqref{eq:intro_fbm_sde}. These modifications are required to prove continuity of the solution map and the modulus of continuity of the semi-group around the support of the noise process, see Theorem~\ref{th:sds_well_defined} and Theorem~\ref{th:strong_feller_proof}. Our main tool in this section, and those that follow, is local non-determinism and its consequences, Lemma~\ref{lem:lnd} and Lemma~\ref{lem:lnd_to_heat_kernel} both of which hold for the perturbed and restricted process $W^H +\psi$.

We make the following standing assumptions throughout this section. Note that \eqref{eq:tightness_singular_parameters} puts us in the situation of \ref{it:intro_existence} of Theorem~\ref{th:main_intro}.
\begin{ass}\label{ass:confining_assumption}
Let $u :\mbR^d\to \mbR^d$ be such that there exists some $\lambda>0$ for which
\begin{equation}\label{eq:confining_assumption}
	\|\nabla u\|_{L^{\infty}_x} \leq \lambda\quad \text{and}\quad \left(u(x)-u(y) \right) \cdot (x-y) \leq - \lambda |x-y|^2.
\end{equation}
\end{ass}

\begin{ass}\label{ass:tightness_singularity}
	Let $(\alpha,H) \in (-\infty,0)\times (0,\nicefrac{1}{2})$ be such that 
\begin{equation}\label{eq:tightness_singular_parameters}
	\alpha >\frac{1}{2}-\frac{1}{2H},
	\end{equation}
\end{ass}
To properly fix ideas concerning the perturbation $\psi$, given an interval $\mcI \subseteq \mbR_+$ and $\gamma$  as in \eqref{eq:w_gamma_delta_norm} we define the set
\begin{equation}\label{eq:psi_bound_longterm}
    \mcC^\gamma_{\mcI;\lambda} \coloneqq \bigg\{\psi \in \mcC^{\gamma}(\mbR;\mbR^d)\,:\, \sup_{t\in \tint} \bigg( \, \absv{ \int_0^t e^{-\lambda(t-r)}  \dd \psi_r } \lor \sup_{\substack{(u,v)\in \mcI^2_\leq\\
    |u-v|\leq 1}}  \frac{\absv{ \psi_u - \psi_v}}{ \absv{v-u}^{\gamma}  } \, \bigg) <+\infty  \bigg\}.
\end{equation} 
We write $\norm{ \,\cdot\, }_{\mcC^\gamma_{\tint;\lambda}}$ for the smallest upper bound on the supremum inside \eqref{eq:psi_bound_longterm}.

\noindent Note that as $\{r\mapsto e^{-\lambda(t-\cdot)}\}$ is smooth, the integral above is understand in the usual Riemann-Stjeltes sense and by integrating by parts one has
\begin{equation*}
    \absv{ \int_0^t e^{-\lambda(t-r)}  \dd \psi_r } \lesssim_\lambda \|\psi\|_{L^\infty_\mcI}.
\end{equation*}
 For brevity, when $\mcI =[0,T]$ for some $T>0$ we will write $\norm{ \,\cdot\, }_{\mcC^\gamma_{T;\lambda}}$
Throughout this section we fix a standard probability space $(\Omega,\mcF,\mbP)$ carrying a two sided Brownian motion $B$ and equipped with the augmented natural filtration $\{\mcF^B_{t\in \mbR}\}_{t\in \mbR}$ and recall the notations \eqref{eq:disintegrated_expectation}, \eqref{eq:disintegrated_cond_expectation} and attendant discussion. 
We are now in a position to state the main result of this section.
\begin{thm}\label{th:tightness_main}
	Let $(\alpha,H)\in (-\infty,0)\times (0,\nicefrac{1}{2})$ satisfy Assumption~\ref{ass:tightness_singularity},$\gamma\in (0,H)$ be as in \eqref{eq:w_gamma_delta_norm} and $u:\mbR^d \to \mbR^d$ satisfy Assumption~\ref{ass:confining_assumption}, $g\in C^\infty(\mbR^d;\mbR^d)$, $x \in L^m_\Omega \mbR^d$ for all $m\geq 1$ such that $\mbE[\exp(\tilde{\kappa}|x|^2)] <\infty$ for some $\tilde{\kappa}>0$. Moreover, let $\tint \subseteq \RR_+$ be a given interval, $\psi \in \mcC^{\gamma}_{\tint;\lambda}$ and $w \in \mcW^-_{\gamma,\delta}$. Then, there exists a $\kappa_0 \coloneqq \kappa_0\left( (H-\gamma)^{-1},\tilde{\kappa},\alpha,H,d,\lambda^{-1},\|g\|_{\mcC^\alpha_x}\right) > 0$
\begin{equation}\label{eq:sol_local_holder_bound}
\sup_{t \in \tint} \EE_{w}\Bigg[ \exp\left( \kappa \|X^{\psi}\|_{\mcC^{\gamma}_{[t,t+1] \cap \tint}}^2 \right)\Bigg] <+\infty,\quad \text{for all } \kappa \in (0,\kappa_0],
\end{equation}
satisfying $\lim \kappa_0 \to 0$ as $\norm{g}_{\mcC^{\alpha}_x} \to \infty$ for each fixed $\lambda > 0$. Moreover, if instead $\psi \in \mcC^{\gamma}_{\mbR_+;\lambda}$, then, for some possibly different $\kappa_0$ depending on the same parameters as above, it holds that
\begin{equation}\label{eq:sol_global_tightness}
		\sup_{t > 0} \EE \Bigg[ \exp\left( \kappa \|X^{\psi}\|_{\mcC^{\gamma}_{[t,t+1]}}^2 \right)\Bigg] < +\infty,\quad \text{for all } \kappa \in (0,\kappa_0],
	\end{equation}
where $\kappa_0$ in \eqref{eq:sol_global_tightness} also vanishes with increasing $\norm{g}_{\mcC^{\alpha}_x}$.
\end{thm}
We prove Theorem~\ref{th:tightness_main} at the end of this section after collecting necessary preliminary lemmas.
\begin{lem}\label{lem:integral_estimate}  
Let $(\alpha,H) \in (-\infty,0)\times (0,\nicefrac{1}{2})$ satisfy Assumption~\ref{ass:tightness_singularity}, $m\in [2,+\infty)$, $\varphi \, \in \mcC^{1+\alpha H}_{\mbR_+} L^m(\Omega;\mbR^d)$ and $g\in C^\infty(\mbR^d;\mbR^d)$. Then, there exists a constant $C\coloneqq C(\alpha,H,d) >0$ such that for any $T>0$, $\psi \in \mcC^\gamma_{T;\lambda}$ and $(s,t)\in [0,T]^2_\leq$ such that $|t-s|\, \leq 1$ , 
\begin{equation}\label{eq:decay_integral_holder_estimate}
	  \left\llbracket\int_s^{\,\cdot\,} g(\varphi_r +W^H_r + \psi_r)\dd r \right\rrbracket_{\mcC^{1+\alpha H}_{[s,t]}L^m_{\Omega}} \leq C\|g\|_{\mcC^\alpha_x}  \left(1+\sqrt{m}+ \llbracket\varphi\rrbracket_{\mcC^{1+\alpha H}_{[s,t]}L^m_{\Omega}} |t-s|^{1+H(\alpha-1)} \right).
	\end{equation}
The statement holds without modification if $\llbracket\,\cdot\, 
\rrbracket_{\mcC^{1+\alpha H}_{[s,t]}L^m_\Omega}$ is replaced with $\llbracket\,\cdot\, 
\rrbracket_{\mcC^{1+\alpha H}_{[s,t]}L^m_{\Omega;w}}$ on the left and right hand sides for any $w\in \mcW^-_{\gamma,\delta}$, recalling the notation of \eqref{eq:disintegrated_expectation}. If $\psi \in \mcC^{\gamma}_{\RR_+;\lambda}$ and the right hand side is finite $(s,t) \in [0,+\infty]^2_{\leq}$, then the bound holds for all $0\leq s<t <+\infty$ with $|t-s|\leq 1$.
\end{lem}
\begin{proof}
	We give the proof in the case $\psi \neq 0$ and for the disintegrated moments $L^m_{\Omega;w}$ for $w\in \mcW^-_{\gamma,\delta}$ to highlight why these do not show up in the dependencies on the right hand side. We highlight where one should change the proof in order to treat obtain estimates on the in terms on the full moments $L^m_\Omega$.

\noindent For $w\in \mcW^{-}_{(\gamma,\delta)}$, any $T>0$ and $\psi \in \mcC^\gamma_{T;\lambda}$ we define the two parameter process $A :[0,T]^2_{\leq }\to \mbR^d$,
		\begin{equation*}
			A_{s,t} \coloneqq \int_{s}^{t}  \mbE^B_{w,s} \left[ g(\varphi_s + W^H_r +\psi_r)\right] \dd r, \quad \text{for all }\, (s,t)\in [0,T]^2_\leq.
		\end{equation*}
    First, appealing to \eqref{eq:disintegrated_heat_kernel_fbm}, we have
	\begin{equation*}
		A_{s,t} = \int_{s}^{t} \mcG^H_{|r-s|} g\, \left(\varphi_s + \mbE^B_{w,s} W^H_{r}+\psi_r\right) \dd r. 
	\end{equation*}
    Hence, taking $\|\,\cdot\,\|_{L^m_{\Omega;w}}$ on both sides applying \eqref{eq:heat_kernel_fbm} and noting that due to Assumption~\ref{ass:tightness_singularity} it holds that $\alpha H > \frac{H-1}{2}>-\frac{1}{2}$, for any $m\in [1,+\infty)$, for all $|t-s|\, \leq 1$ we have 
\begin{align}\label{eq:averaging_integral_A_st_bound}
    \|A_{s,t}\|_{L^m_{\Omega;w}} \, \lesssim_{\,\alpha,H,d} &\, \|g\|_{\mcC^\alpha_x} \int_s^t  |r-s|^{\alpha H
    } \dd r
    \leq  \,\|g\|_{\mcC^\alpha_x}  |t-s|^{1+\alpha H}.
\end{align}
 We now show that $A$ satisfies the assumptions of Lemma~\ref{lem:sewing_stochastic_integral}, after suitably defining the probability space and associated Brownian motion to accommodate the conditioning to the left of zero.  Concerning the three point difference, by addition and subtraction we have
	\begin{align*}
		\delta A_{s,u,t}  =\, & \int_u^t \EE^B_{w,s} \left[g(\varphi_s + \wh_r+\psi_r) \right] \dd r - \int_u^t \EE^B_{w,u} \left[g(\varphi_u + \wh_r +\psi_r)\right] \dd r \\
	=	\, & \int_u^t \left(\EE^B_{w,s}-\EE^B_{w,u}\right)\left[g(\varphi_s + \wh_r+\psi_r) \right] \dd r\\
    & +\int_u^t \EE^B_{w,u} \left[g(\varphi_u+\wh_r+\psi_r) - g(\varphi_s+\wh_r+\psi_r) \right]\dd r.
        %
	\end{align*}
Applying \eqref{eq:malliavinfdx} followed by the tower law and Lemma~\ref{lem:lnd_to_heat_kernel} to the function $x \mapsto f(x)\coloneqq \mbE^B_{w,u}[g(x+W^H_r+ \psi_r)]$ and some $s\in [0,u]$ we see that 
\begin{align*}
    \mbE^B_{w,u}[g(W^H_r+\varphi_s)] 
    = & \,\mbE^B_{w,s}\left[g(\varphi_s+W^H_r+\psi_r)\right] +c_H \int_s^{u} (r-z)^{H-\nicefrac{1}{2}}\mbE^B_{w,z}\left[\nabla g(\varphi_s+W^H_r+\psi_r) \right]  \dd B_z\\
    =& \, \mbE^B_{w,s}\left[g(\varphi_s+W^H_r+\psi_r)\right] +\int_s^{u} (r-z)^{H-\nicefrac{1}{2}} \nabla \mcG^H_{|r-z|}g(\varphi_s+\mbE^B_{w,z} [W^H_r]+\psi_r)   \dd B_z.
\end{align*}
Hence, subtracting $\mbE^B_s\left[g(\varphi_s+W^H_r+\psi_r)\right]$ from both sides and applying the stochastic Fubini theorem, justified since $g$ is smooth and bounded, we have obtained the identity
\begin{equation}\label{eq:integral_estimate_conditionals_diff}
\begin{aligned}
 \int_u^t \left(\EE^B_{w,s}-\EE^B_{w,u}\right)\big[g(\varphi_s + &\wh_r+\psi_r) \big] \dd r \\
 &= - \int_s^{u}\int_u^t (r-z)^{H-\nicefrac{1}{2}} \nabla \mcG^H_{|r-z|}g(\varphi_s+\mbE^B_{w,z}[W^H_r]+\psi_r)\,\dd r  \dd B_z.
 \end{aligned}
\end{equation}
As such, we have shown that
\begin{align*}
    \delta A_{s,u,t} =&\, - \int_s^{u}\int_u^t (r-z)^{H-\nicefrac{1}{2}} \nabla \mcG^H_{|r-z|}g(\varphi_s+\mbE^B_{w,z} W^H_r+\psi_r)\,\dd r  \dd B_z\\
    &\,+ \int_u^t \EE^B_{w,u} \left[g(\varphi_u+\wh_r+\psi_r) - g(\varphi_s+\wh_r+\psi_r) \right]\dd r\\
    \eqqcolon &\, \int_s^u h^s_z \dd B_z + J^u_{s,t}.
\end{align*}
To estimate $h$ in mean square we first apply \eqref{eq:heat_kernel_fbm} to give for each $z \in [s,u]$
\begin{align*}
    |h^s_z|^2 \,\lesssim_{\, \alpha,H,d} &\, \left( \norm{g}_{\mcC^{\alpha}_x} \int_u^t \absv{r-z}^{H(\alpha-1)} \absv{r-z}^{H-\nicefrac{1}{2}} \dd r \right)^2 \\
\leq  &\, \norm{g}^2_{\mcC^{\alpha}_x} \left(\int_u^t  |r-z|^{H\alpha -\nicefrac{1}{2}}\, \dd r \right)^2\\
    \lesssim &\, \norm{g}_{\mcC^{\gamma}_x}^2\absv{t-u}^{1+2\alpha H} ,
\end{align*}
where we used \eqref{eq:tightness_singular_parameters} to check that $H\alpha - \frac{1}{2} > \frac{H}{2} - 1>-1$ to evaluate the integral and the fact that $u-z <t-z$ to bound $|t-z|^{1+2\alpha H} - |u-z|^{1+2\alpha H} \leq |t-u|^{1+2\alpha H}$. As a result, there exists a constant $C\coloneqq C(\alpha, H,d)>0$ such that
\begin{equation}
    \|h^s\|_{L^2_{[s,t]}} = \left(\int_s^t |h_z|^2 \dd z\right)^{\nicefrac{1}{2}} \leq C \norm{g}_{\mcC^{\alpha}_x}  |t-s|^{1+\alpha H }.
\end{equation}
Appealing again to \eqref{eq:tightness_singular_parameters} we check that $ H\alpha  +1 > \frac{H}{2}+\frac{1}{2}>\frac{1}{2}$ so that \eqref{eq:stoch_integral_sewing_h_bound} is satisfied with $\eps_2 = \frac{1}{2}+\alpha H>0$ and $C_{h,s} = C \|g\|_{\mcC^\alpha_x}$, in particular deterministic and equal for all $s \in [0,T]$.

\noindent To treat $J$ we apply \eqref{eq:disintegrated_lnd_to_heat_kernel_stoch_integral} of  Lemma~\ref{lem:lnd_to_heat_kernel} to the function $(x,y) \mapsto g(x+W^H_r+\psi_r)- g(y+W^H_r+\psi_r)$ to see that
\begin{equation*}
    \mbE_u^B[g(\varphi_u + W^H_r+\psi_r) - g(\varphi_s +W^H_r+\psi_r)] = \mcG^H_{|r-u|}\left( g(\varphi_u + \mbE_u W^H_{u,r} +\psi_r) - g(\varphi_s + \mbE_u W^H_{u,r}+\psi_r)\right).
\end{equation*}
Hence, applying \eqref{eq:heat_kernel_fbm},
\begin{align*}
    |J^u_{s,t}| & \leq  \,\int_u^t \mcG^H_{|r-u|^{2H}} \left( g(\varphi_u + \mbE_u W^H_{u,r}+\psi_r) - g(\varphi_s + \mbE_u W^H_{u,r}+\psi_r)\right)|\,  \dd r \\
    & \lesssim_{\alpha,H,d} \, \|g\|_{\mcC^\alpha_x} \int_u^t |r-u|^{H(\alpha-1)}\,|\varphi_u-\varphi_s| \dd r.
    \end{align*}
So taking moments on both sides we have
\begin{align*}
   \|J^u_{s,t}\|_{L^m_{\Omega;w}} & \lesssim_{\, \alpha,H,d} \,  \|g\|_{\mcC^\alpha_x} \int_u^t |r-u|^{H(\alpha-1)}\,\|\varphi_u-\varphi_s\|_{L^m_{\Omega;w}} \dd r,\\
   & \leq \, \llbracket\varphi\rrbracket_{\mcC^{1+\alpha H}_{[s,u]}L^m_{\Omega;w}} \|g\|_{\mcC^\alpha_x} |u-s|^{1+\alpha H} \int_u^t|r-u|^{H(\alpha-1)} \dd r\\
   & \leq \, \llbracket\varphi\rrbracket_{\mcC^{1+\alpha H}_{[s,t]}L^m_{\Omega;w}}\|g\|_{\mcC^\alpha_x} |t-s|^{1+\alpha H} |t-s|^{H(\alpha-1)+1},
\end{align*}
where we used \eqref{eq:tightness_singular_parameters} to check that $H(\alpha-1) >-H - \frac{1}{2} >-1$ in order to evaluate the integral. Hence, \eqref{eq:stoch_integral_sewing_J_bound} is satisfied with $\eps_1 = (1+\alpha H) + (H(\alpha-1)+1)-1 = 2H\alpha-H +1> 0$ and $C_J = C(\alpha,H,d)\llbracket\varphi\rrbracket_{\mcC^{1+\alpha H}_{\mbR_+}L^m_{\Omega;w}}\|g\|_{\mcC^\alpha_x}$. We have shown that the conditions of Lemma~\ref{lem:sewing_stochastic_integral} are satisfied and it is readily checked that $t\mapsto \int_0^t g(\varphi_r +W^H_r)\dd r$ is the resulting stochastic process (see for example the proof of \cite[Lem.~2.13]{galeati_le_mayorcas_24_quantitative} for similar steps). Furthermore, combining \eqref{eq:sewing_stoch_integral_bnd_1} with \eqref{eq:averaging_integral_A_st_bound} we see that there exist constants $C\coloneqq C(\alpha,H,d)>0$ such that for all $(s,t)\in [0,T]^2_{\leq}$ with $|t-s|\, \leq 1$,
\begin{align*}
    \left\|\int_s^t g(\varphi_r +W^H_r)\dd r\,  \right\|_{L^m_{\Omega;w}}  \leq & \, \left\|\int_s^t g(\varphi_r +W^H_r)\dd r - A_{s,t} \right\|_{L^m_{\Omega;w}} + \|A_{s,t}\|_{L^m_{\Omega;w}}\\
    \leq & \, C \|g\|_{\mcC^\alpha_x} |t-s|^{1+\alpha H}\Big( \sqrt{m}  +\llbracket\varphi\rrbracket_{\mcC^{1+\alpha H}_{[s,t]}L^m_{\Omega;w}}\absv{t-s}^{1+H(\alpha-1)}  + 1\Big),
\end{align*}
from which the claim follows with $\llbracket\,\cdot\, 
\rrbracket_{\mcC^{1+\alpha H}_{[s,t]}L^m_\Omega}$ replaced by $\llbracket\,\cdot\, 
\rrbracket_{\mcC^{1+\alpha H}_{[s,t]}L^m_{\Omega;w}}$ on both the left and right hand sides. From the proof above it is clear that the former case holds in the same manner without any change in dependency of the constants.
\end{proof}
The next lemma gives us an a priori H\"older bound on the slow remainder of any solution to \eqref{eq:psi_sde}.
\begin{lem}\label{lem:global_theta_bound}
 Let $g\in C^\infty(\mbR^d;\mbR^d)$, $(\alpha,H)\in (-\infty,0)\times (0,\nicefrac{1}{2})$ satisfy Assumption~\ref{ass:tightness_singularity}, $u$ satisfy Assumption~\ref{ass:confining_assumption} and $m\in [2,+\infty)$. Then, there exists a constant $C\coloneqq C(\alpha,H,d) >0$ such that for any interval $\mcI\subseteq \mbR_+$,  $\psi \in \mcC^{\gamma}_{\tint;\lambda}$, $X^{\psi}$ any solution to \eqref{eq:psi_sde} on $\tint$, $\theta^\psi \coloneqq X^\psi - W^H-\psi$ and any pair $(s,t) \in \tint^2$ that satisfies $|t-s|\leq 1$.
\begin{equation}\label{eq:stationary_theta_bound}
     \big\llbracket \theta^{\psi} \big\rrbracket_{\mcC^{1+\alpha H}_{[s,t]}L^m_{\Omega}} \leq \sqrt{m}C\left(1+\|g\|_{\mcC^\alpha_x}\right)^{1-\frac{\alpha H}{1+H(\alpha-1)}} +  \lambda  \|X^{\psi} \|_{L^\infty_{\tint}L^m_{\Omega}}.
\end{equation}
Furthermore, the statement holds without modification if $\llbracket\,\cdot\, 
\rrbracket_{\mcC^{1+\alpha H}_{[s,t]}L^m_\Omega}$ is replaced with $\llbracket\,\cdot\, 
\rrbracket_{\mcC^{1+\alpha H}_{[s,t]}L^m_{\Omega;w}}$ on both the left and right hand sides.
\end{lem}
\begin{proof}
 We give the proof in the case of the full moments $L^m_\Omega$, noting that the only modification to treat the disintegrated moments $L^m_{\Omega;w}$ comes from taking the modified version of \eqref{eq:decay_integral_holder_estimate}. By definition, for any $0\leq s<t <+\infty$ one has:
\begin{equation*}
    \theta^\psi_t - \theta^\psi_s = \int_s^t g(\theta^\psi_r + W^H_r+\psi)\dd r + \int_s^t u(X^\psi_r)\dd r.
\end{equation*}
Since by assumption $g\in C^\infty(\mbR^d;\mbR^d)$ and $u \in \dot{W}^{1,\infty}(\mbR^d;\mbR^d)$ it follows that $\theta^\psi \in \dot{W}^{1,\infty}([0,T];L^m_\Omega)$ for any $T\in \mbR_+$ and hence we may apply Lemma~\ref{lem:integral_estimate} with $\varphi = \theta$ to obtain, for any $|t-s|\, \leq 1$
\begin{equation*}
    \llbracket \theta^\psi \rrbracket_{\mcC^{1+\alpha H}_{[s,t]}L^m_\Omega} \leq C\|g\|_{\mcC^\alpha_x} \left(\sqrt{m}+ \llbracket\theta^\psi\rrbracket_{\mcC^{1+\alpha H}_{[s,t]}L^m_\Omega}\absv{t-s}^{\kappa}   +  1\right) + \lambda |t-s|^{-\alpha H} \sup_{r\in [s,t]}\|X^\psi_r\|_{L^m_\Omega}
\end{equation*}
for some $C\coloneqq C(\alpha,H,d)>0$ and with $\kappa \coloneqq 1+H(\alpha-1)$. Hence, choosing $(s',t')\in [0,1]^2_\leq$ such that 
\begin{equation*}
    |t'-s'| = \left(\frac{1}{2( C \|g\|_{\mcC^\alpha_x}+1)}\right)^{\frac{1}{\kappa}},
\end{equation*}
it follows that
\begin{align*}
    \llbracket \theta^\psi \rrbracket_{\mcC^{1+\alpha H}_{[s',t']}L^m_\Omega} \leq  &\,  2C\|g\|_{\mcC^\alpha_x}\left(1+\sqrt{m}\right)      + 2\lambda \left(\frac{1}{2 (C \|g\|_{\mcC^\alpha_x}+1)}\right)^{\frac{-\alpha H}{\kappa}} \sup_{r\in [s,t]}\|X^\psi_r\|_{L^m_\Omega}.
\end{align*}
Hence, by a simple argument (see e.g. \cite[Ex.~4.5]{friz_hairer_20_book}), for all $|t-s|\, \leq 1$, it holds that
\begin{align*}
    \llbracket \theta^\psi \rrbracket_{\mcC^{1+\alpha H}_{[s,t]}L^m_\Omega}  \leq &\,2C\|g\|_{\mcC^\alpha_x}\left(1+\sqrt{m}\right)\left(\frac{1}{2( C \|g\|_{\mcC^\alpha_x}+1)}\right)^{\frac{\alpha H}{\kappa}} + 2\lambda  \sup_{r\in [s,t]}\|X^\psi_r\|_{L^m_\Omega}\\
    \leq &\, \sqrt{m}C(1+\|g\|_{\mcC^\alpha_x})^{1-\frac{\alpha H}{\kappa}} + 2\lambda  \sup_{r\in [s,t]}\|X^\psi_r\|_{L^m_\Omega},
\end{align*}
which concludes the proof for the full moments $L^m_\Omega$ upon taking suprema over $(s,t) \in \tint^2$ such that $|t-s|\, \leq 1$. Crucially we observe that the constant will not depend on $\psi$ nor $w$ in the disintegrated moments case.
\end{proof}
The following lemma is an analogue of Lemma~\ref{lem:integral_estimate} with an exponential damping factor inside the integral.
\begin{lem}\label{lem:product_integral_estimate}
	Let $d  \in \mbN$, $(\alpha,H) \in (-\infty,0)\times (0,\nicefrac{1}{2})$ satisfy Assumption~\ref{ass:tightness_singularity}, $m\in [1,+\infty)$, $\rho \in \mcC^{1+\alpha H}_{\mbR^+}L^{2m}(\Omega;\mbR^{d}),\, \varphi \in \mcC^{1+\alpha H}_{\mbR_+}L^{2m}(\Omega;\mbR^d)$ and $g\in C^\infty(\mbR^d)$. Then,  there exist constants $C\coloneqq C(\alpha,H,d)>0$ and $c\coloneqq (\alpha,H) \in (0,1)$ such that for any interval $\mcI\subseteq \mbR_+$, $\psi\in \mcC^\gamma_{\mcI;\lambda}$ and $V>1$,
\begin{equation}\label{eq:decay_integral_supremum_estimate}
        \begin{aligned}
			\sup_{t \in \tint}  \bigg\|\int_0^t  e^{-\lambda (t-r)} \rho_r \cdot g(\varphi_r + & W^H_r + \psi_r) \, \dd r\bigg\|_{L^m_{\Omega}} \\
            &\, \leq \, \,  C \frac{ 2^{-c  V }}{ \lambda \land \lambda^{\nicefrac{1}{2}} }\|g\|_{\mcC^\alpha_x}  \|\rho\|_{L^\infty_{\tint}L^{2m}_{\Omega}}\bigg( \sqrt{m}   +  \llbracket \varphi \rrbracket_{\mcC^{1+\alpha H}_{\tint}L^{2m}_{\Omega}} \bigg)\\
    &\quad +C \frac{ 2^{-c V} }{ \lambda \land \lambda^{\nicefrac{1}{2}} }\|g\|_{\mcC^\alpha_x} \llbracket \rho\rrbracket_{\mcC^{1+\alpha H}_{\tint}L^m_{\Omega}}\\
    &\quad +\left(1+ C \frac{2^{(1-c) V}}{\lambda \wedge \lambda^{\nicefrac{1}{2}}}\right) \|g\|_{\mcC^\alpha_x}  \| \rho\|_{L^\infty_{\tint}L^{m}_{\Omega}}.
            \end{aligned}
	\end{equation}
	Furthermore, the same statement holds without modification if $L^m_{\Omega}$ is replaced by $L^m_{\Omega;w}$ everywhere on the left and right hand sides for any $w\in \mcW^-_{\gamma,\delta}$.
\end{lem}
	\begin{proof} We give the proof for general $\psi \in \mcC^\gamma_{\mcI;\lambda}$ but only for the full moments $L^m_\Omega$. The modifications required to prove the equivalent statement for the disintegrated moments $L^m_{\Omega;w}$ are essentially notational and mirror those in the proof of Lemma~\ref{lem:integral_estimate}. For any $T\in \mcI\setminus \{+\infty\}$ we define the two parameter process
\begin{equation}
    A^T_{s,t} = e^{-\lambda T} \int_s^t e^{\lambda r}\,  \rho_s\cdot  \EE^B_{s} \left[g(\varphi_s+\wh_r+\psi_r ) \right] \,\dd r,\quad \text{for all } (s,t)\in [0,T]^2_\leq. 
\end{equation}
We will check that the family $\{A^T\}_{T\in \mbR_+}$ satisfies the assumptions of Lemma~\ref{lem:sewing_decay_stochastic_integral}.

\noindent Firstly, appealing to \eqref{eq:heat_kernel_fbm_identity} we write
\begin{equation*}
     A^T_{s,t} = e^{-\lambda T} \int_s^t e^{\lambda r}\,  \rho_s\cdot  \mcG^H_{|r-s|}\, g(\varphi_s+\mbE^B_s[\wh_r] +\psi_r)\,\dd r.
\end{equation*}
Applying \eqref{eq:heat_kernel_fbm}, in a similar manner as used to obtain \eqref{eq:averaging_integral_A_st_bound}, for all $(s,t)\in [0,T]^2_{\leq}$, we directly have
\begin{align*}
    \|A^T_{s,t}\|_{L^m_{\Omega}} \,    &\lesssim_{\, \alpha,H,d} \, e^{-\lambda(T-t)} \| \rho\|_{L^\infty_{I}L^{m}_\omega} \|g\|_{\mcC^\alpha_x}\left(|t-s|^{1+\alpha H}\wedge 1\right).
\end{align*}
On the other hand, first by direct computation and then recognizing that we can apply the same arguments as used to derive \eqref{eq:integral_estimate_conditionals_diff} in the first term developed below, we find that for all $(s,u,t)\in [0,T]^3_{\leq}$,
\begin{align*}
 \delta A^T_{s,u,t} = &\,  e^{-\lambda T}\int_u^t e^{\lambda r} \left(\rho_s \cdot \mbE^B_s[g(\varphi_s + W^H_r+\psi_r)] - \rho_u \cdot \mbE^B_u[g(\varphi_u + W^H_r+\psi_r)]\right)  \, \dd r\\
 %
 %
%
%
= &\, e^{-\lambda T}\int_s^{u} \int_u^t e^{\lambda r} (r-z)^{H-\nicefrac{1}{2}}  \rho_u \cdot \nabla \mcG^H_{|r-z|}g(\varphi_s+\mbE^B_z [W^H_{r}]+\psi_r)\, \dd r \, \dd B_z   \\
&\, + e^{-\lambda T}\int_u^t e^{\lambda r} \left(\rho_s- \rho_u\right) \cdot  \mcG^H_{|r-s|}\, g(\varphi_s+\mbE^B_s[\wh_r]+\psi_r )  \, \dd r\\
&\, + e^{-\lambda T}\int_u^t e^{\lambda r} \rho_u \cdot \mcG^H_{|r-u|}\left( g(\varphi_u + \mbE_u [W^H_r]+\psi_r) - g(\varphi_s + \mbE_u [W^H_r]+\psi_r)\right)  \, \dd r \\
\eqqcolon&\, \int_s^u h^{T,u,t}_z \dd B_z + J^{T,s}_{u,t},
\end{align*}
where
\begin{equation*}
    h^{T,s,u,t}_z \coloneqq \indic_{[s,u]}(z) e^{-\lambda T}  \rho_u \cdot \int_u^t e^{\lambda r} (r-z)^{H-\nicefrac{1}{2}} \,  \nabla \mcG^H_{|r-z|}g\Big(\varphi_s+\mbE^B_z [W^H_r]+\psi_r\Big)\, \dd r ,
\end{equation*}
and
\begin{align*}
    J^{T,s}_{u,t} \coloneqq &\, e^{-\lambda T} \bigg(\int_u^t e^{\lambda r} \left(\rho_s- \rho_u\right) \cdot  \mcG^H_{|r-s|}\, g(\varphi_s+\mbE^B_s[\wh_r]+\psi_r )  \, \dd r \\
    &\qquad \quad   + \int_u^t e^{\lambda r} \rho_u \cdot \mcG^H_{|t-s|}\Big( g(\varphi_u + \mbE_u [W^H_r]+\psi_r) - g(\varphi_s + \mbE_u [W^H_r]+\psi_r)\Big)  \, \dd r \bigg).
\end{align*}
To show \eqref{eq:sewing_stoch_integral_decay_h_estimate}, recalling that we are only required to consider $|s-u|\leq |t-s|\leq 1$, we apply \eqref{eq:heat_kernel_fbm} to see that for all $z\in [s,u]$
\begin{align*}
    |h^{T,s,u,t}_z|^2 \,&\lesssim_{\alpha,H,d} \, \|g\|^2_{\mcC^\alpha_x} |\rho_u|^2 e^{-2\lambda (T-t)} \left(\int_u^t |r-z|^{\alpha H-\nicefrac{1}{2}}\,  \dd r\right)^2\\
    &\lesssim \, \|g\|^2_{\mcC^\alpha_x} |\rho_u|^2 e^{-2\lambda (T-t)}|t-u|^{1+2H\alpha}.
\end{align*}
As a result, for $(s,t)\in \mcI^2_{\leq}$
\begin{equation*}
    \|h^{T,s,u,t}\|_{L^2_{[s,t]}} \lesssim_{\alpha,H,d} \, \|g\|_{\mcC^\alpha_x} \|\rho\|_{L^\infty_\mcI}  e^{-2\lambda (T-t)}|t-s|^{1+\alpha H}.
\end{equation*}
To obtain \eqref{eq:sewing_stoch_integral_decay_J_estimate}, by the triangle inequality we directly split the estimate and apply \eqref{eq:heat_kernel_fbm} followed by the Cauchy--Schwartz inequality in the second term, to give, for any $(s,u,t)\in \mcI^3_\leq$
\begin{align*}
   \left\|J^{T,s}_{u,t}\right\|_{L^m_\Omega} \lesssim_{\alpha,H,d} & \, e^{-\lambda T} \|g\|_{\mcC^\alpha_x} \bigg(\, \int_u^t e^{\lambda r}\|\rho_s-\rho_u\|_{L^m_\Omega}\,  |r-s|^{\alpha H} \dd r  \\
    &\qquad \qquad\qquad  \, + \int_u^t e^{\lambda r} \|\rho_u\|_{L^{2m}_\omega} \|\varphi_u-\varphi_s\|_{L^{2m}}|r-u|^{H(\alpha-1)} \,\dd r\, \bigg)\\
    &\, \leq e^{-\lambda (T-t)} |t-s|^{2+H(2\alpha -1)}\|g\|_{\mcC^\alpha_x}\left(\llbracket\rho\rrbracket_{\mcC^{1+\alpha H}_{\mcI} L^{m}_\omega} + \|\rho\|_{L^\infty_{\mcI}L^{2m}_\omega} \llbracket\varphi\rrbracket_{\mcC^{1+\alpha H}_{\mcI}L^{2m}_\omega}\right).
\end{align*}
By an application of Lemma~\ref{lem:sewing_stochastic_integral} (which we don't detail) one readily checks that the unique $\mbF^B$-adapted family resulting from Lemma~\ref{lem:sewing_decay_stochastic_integral} is equal to $\int_0^t  e^{-\lambda (T-r)} \rho_r \cdot g(\varphi_r + W^H_r) \, \dd r$. Hence, we are only left to obtain the estimate \eqref{eq:decay_integral_supremum_estimate} which we do by applying the bound \eqref{eq:uniform_time_bound_ssl} in Lemma~\ref{lem:sewing_decay_stochastic_integral}. Setting, $\eps_1 = \eps_3= \frac{1}{2}+\alpha H >0$, $\eps_2 = 1+H(2\alpha-1)>0$ to see that $\eps_2 \wedge \frac{\eps_3}{2} =  (1+H(2\alpha-1)) \wedge (\frac{1}{4}+\frac{\alpha H}{2}) = \frac{1}{4} + \frac{\alpha H}{2} \in (0,\nicefrac{1}{4})$ and $C\coloneqq C(\alpha,H,d)$ the collected implicit constants from above, we see that due to \eqref{eq:uniform_time_bound_ssl} we have the bound 
\begin{align*}
	\sup_{T\in \tint} \sup_{t \in [0,T]} \Bigg\| \int_0^t  e^{-\lambda (T-r)} \rho_r \cdot g(\varphi_r &+ W^H_r+\psi_r) \, \dd r\,  \Bigg\|_{L^m_{\Omega}} \\
    \leq\,&\,  \sup_{T\in\mcI } \sup_{t \in [0,T]} \norm{ \int_0^t  e^{-\lambda (T-r)} \rho_r \cdot g(\varphi_r + W^H_r+\psi_r) \, \dd r  - A^T_{s,t}}_{L^m_{\Omega}} \\
    &\, + \sup_{T\in\mcI } \sup_{t \in [0,T]} \norm{ A^T_{s,t} }_{L^m_{\Omega}}\\
   \leq \, &\,  C \frac{ 2^{- \left(\frac{1}{4}+\frac{\alpha H}{2}\right) V }}{ \lambda \land \sqrt{\lambda} }  \|g\|_{\mcC^\alpha_x}  \bigg( \sqrt{m}\|\rho\|_{L^\infty_{\mcI} L^m_\Omega } + \llbracket\rho\rrbracket_{\mcC^{1+\alpha H}_{\mcI} L^{m}_\omega}\\
    &\qquad \qquad \quad\qquad \qquad   + \|\rho\|_{L^\infty_{\mcI}L^{2m}_\omega} \llbracket\varphi\rrbracket_{\mcC^{1+\alpha H}_{\mcI}L^{2m}_\omega} \bigg)\\
    &\, + \left(1+ C 2^{\left(\frac{3}{4}-\frac{\alpha H}{2}\right) V} \right)\| \rho\|_{L^\infty_{\mcI}L^{m}_\omega} \|g\|_{\mcC^\alpha_x},
	\end{align*}
which implies the claimed estimate.
\end{proof}
We are now ready to prove the main theorem of this section.
\begin{proof}[Proof of Theorem~\ref{th:tightness_main}]
We provide the full proof in the case of the disintegrated expectation $\mbE_{w}$ for $w\in \mcW^-_{\gamma,\delta}$, giving details at the end of the  modifications required in the case of the full expectation. In particular, we show how \eqref{eq:sol_global_tightness} is obtained if $\psi \in \mcC^{\gamma}_{\mbR_+;\lambda}$. In the proof we only track dependence on $\sqrt{m}$ and $\norm{ g }_{\mcC^{\alpha}_x}$. The dependence on the rest of the parameters are going to be kept inside the multiplicative constant $C$ that may change line by line.

The strategy of proof is going to be to rewrite $X^{\psi} = Y^{\psi} + \rho$, where $Y^{\psi}$ is a perturbed Ornstein-Uhlenbeck process, and obtain desired estimate for both of them separately. We let $Y^\psi$ be the solution to
\begin{equation}\label{eq:ypsi_def_ode}
Y^\psi_t = - \lambda \int_0^t Y^\psi_s \dd s + W^H_t  + \psi_t, \quad \text{for all }t\in \tint.
\end{equation}
By variation of constants we obtain:
\[ Y_t^{\psi} = \int_0^t e^{-\lambda(t-r)} \dd \left( \wh_r + \psi_r \right). \]
Triangle inequality and the decomposition \eqref{eq:wh_decomp_with_awop} yields the following bound 
\begin{equation}\label{eq:ypsi_lm_bound_triangle}
 \big\| Y^{\psi} \big\|_{L^{\infty}_{\tint}\llmw} \leq  \norm{ \int_0^{\cdot} e^{-\lambda(\cdot-r)} \dd W^{H;1}_{0,r} }_{L^{\infty}_{\tint}\llmw} + \norm{ \int_0^{\cdot} e^{-\lambda(\cdot-r)}\dd\left( \msA w_r + \psi_r \right) }_{L^{\infty}\llmw}.
 \end{equation}
  By independence of Brownian increments we can apply Lemma \ref{lem:rl_nice_bd} to the first term, which is identical to the same expression but $\llmw = L^m_{\Omega}$ ($W^{H;1}_{0,t}$ is identical to Riemann-Liouville fBm), and the second term is just deterministic by measurability of $\msA w$. We therefore obtain
\begin{equation}\label{eq:ou_tightness}
    \big\| Y^{\psi} \big\|_{L^{\infty}_{\tint}\llmw} \ls \sqrt{m} + \norm{ \msA w + \psi }_{\mcC^{\gamma}_{\tint;\lambda}} \leq  \sqrt{m} + \norm{ \msA w}_{\mcC^{\gamma}_{\tint;\lambda}} + \norm{\psi }_{\mcC^{\gamma}_{\tint;\lambda}}. 
\end{equation} 
Therefore we can estimate H\"older norm \eqref{eq:uniform_holder_moment} with $E = \llmw$, by applying triangle inequality to increments of \eqref{eq:ypsi_def_ode}, and using \eqref{eq:ou_tightness} in the first term, along with H\"older embeddings \eqref{eq:global_holder_ordering}: 
\begin{equation}\label{eq:ypsi_holder}
\llbracket Y^{\psi}\rrbracket_{\mcC^{\gamma}_{\tint}\llmw} \ls \norm{ \msA w}_{\mcC^{\gamma}_{\tint;\lambda}} + \norm{\psi }_{\mcC^{\gamma}_{\tint;\lambda}} + C_H \llbracket W^{H;1}_{0,\cdot} \rrbracket_{\mcC^{\gamma}_{\tint}\llmw}.
\end{equation}
We can now estimate $\rho$, which is going to be more involved part of the proof. First we observe that this satisfies the identity:
\begin{equation*}
	\rho_t = \int_0^t g(X^\psi_r)\dd r + \int_0^t u(X^\psi_r)\dd r + \lambda \int_{0}^{t} Y^\psi_r \dd r.
\end{equation*}
Since $g$ is taken to be smooth and $u$ is Lipschitz the process $t\mapsto \rho_t$ is of finite variation and we have
\begin{equation*}
	\frac{\dd}{\dd t } |\rho_t|^2 = \rho_t \cdot g(X^\psi_t) + \rho_t \cdot (u(X^\psi_t)+ \lambda Y^\psi_t).
\end{equation*}
Since $u$ satisfies Assumption~\ref{ass:confining_assumption} we appeal to monotonicity and the bound on its Lipschitz norm \eqref{eq:confining_assumption} along with Young's product inequality to write
\begin{align*}
\rho_t \cdot (u(X^\psi_t)+ \lambda Y^\psi_t)= \rho_t \cdot \left( u(X^\psi_t) - u(Y^\psi_t) \right) + \rho_t  \cdot \left( \lambda Y^\psi_t + u(Y^\psi_t) \right) \leq &- 2\lambda |\rho_t|^2 + 2 \lambda |\rho_t|  |Y^\psi_t|\\
\leq & - \lambda |\rho_t|^2 +  \lambda   |Y^\psi_t|^2.
\end{align*}
So that by variation of constants we have the following inequality for all $t\in [0,T]$
\begin{equation}\label{eq:rho_squared}
	|\rho_t|^2 \,\leq\,  e^{-\lambda t} |x|^2  \, +\,   \left|\int_0^t e^{-\lambda (t-r)}  \rho_r\cdot  g(X^\psi_r)   \dd r \right| + \lambda \int_0^t  e^{-\lambda (t-r)} | Y^\psi_r |^2 \dd r.	
\end{equation}
The last term is immediately bounded in $\norm{ \,\cdot\, }_{L^{\infty}_{\tint}\llmw}$ by \eqref{eq:ou_tightness}, so we are left to estimate the second one. Since $g$ is assumed to be smooth, the process $\theta^{\psi}_{\cdot} = x + \int_0^{\,\cdot\,} g(X^\psi_r)\dd r + \int_0^{\,\cdot\,} u(X^\psi_r)\dd r = X^\psi_{\,\cdot\,} - W^H_{\,\cdot\,}+\psi_{\,\cdot\,}$  satisfies the assumptions of Lemma~\ref{lem:global_theta_bound} so that we can apply \eqref{eq:decay_integral_supremum_estimate} from Lemma~\ref{lem:product_integral_estimate}, to see that for any $V>1$ and some $c \coloneqq c(\alpha,H) \in (0,1)$ one has
\begin{equation}\label{eq:rhogx_bd}
\begin{aligned}
    \sup_{t\in \tint }\left\|\int_0^t e^{-\lambda (t-r)}  \rho_r\cdot  g(X^{\psi}_r )   \dd r \right\|_{L^{m}_{\Omega;w}} \leq \, &\,  C 2^{- c V}\|g\|_{\mcC^\alpha_x}  \|\rho\|_{L^\infty_{\tint}L^{2m}_{\Omega;w}}\bigg( \sqrt{m}   +  \llbracket \theta^{\psi} \rrbracket_{\mcC^{1+\alpha H}_{\tint}L^{2m}_{\Omega;w}} \bigg)\\
    &\, +C 2^{-c V} \|g\|_{\mcC^\alpha_x} \llbracket \rho\rrbracket_{\mcC^{1+\alpha H}_{\tint}L^m_{\Omega;w}}\\
    &\, +\left(1+ C 2^{(1-c) V} \right) \|g\|_{\mcC^\alpha_x}  \| \rho\|_{L^\infty_{\tint}L^{m}_{\Omega;w}} \\
    &\!\!\!\!\!\!\!\!\!\!\!\!\!\!\!\!\!\!\!\!\!\!\!\! \leq \, C 2^{-cV} \left( 1 + \norm{ g }_{\mcC^{\alpha}_x} \right)^{2 - \frac{\alpha H}{1-H(\alpha-1)}} \left( \sqrt{m} + \llbracket \rho \rrbracket_{\mcC^{1+\alpha H}_{\tint}\llmw} + \big\| X^{\psi} \big\|_{L^{\infty}_{\tint}\lldmw} \right) \\ 
    & + \left(1 + C 2^{(1-c)V} \right) \norm{ g }_{\mcC^{\alpha}_x} \norm{ \rho }_{L^{\infty}_T\lldmw}, 
\end{aligned}
\end{equation}
 where in the second line we used Lemma \ref{lem:global_theta_bound} to bound $\llbracket \theta^{\psi} \rrbracket_{\mcC^{1+\alpha H}_{\tint}\llmw}$. By the the triangle inequality, the assumption $\norm{ \nabla u }_{L^{\infty}_x} \leq \lambda$ and that $X^{\psi} = Y^{\psi} + \rho$, there holds (recall that $C \coloneqq C(\alpha, H, \lambda, d)>0$, but crucially it does not depend on $g$ or $m$)
\begin{equation}\label{eq:rho_holder}
\begin{aligned}
    \llbracket\rho \rrbracket_{\mcC^{1+\alpha H}_{\tint} L^m_{\Omega;w}} \, = \,&\,  \left\llbracket x + \int_0^{\,\cdot\,}g(X^\psi_r)\dd r+ \int_0^{\,\cdot\,}u(X^\psi_r)\dd r + \lambda \int_0^{\,\cdot\,} Y^\psi_r\, \dd r \right\rrbracket_{\mcC^{1+\alpha H}_{\tint} L^m_{\Omega;w}}\\
    \leq \,&\,   \left\llbracket \int_0^{\,\cdot\,} g(X^\psi_r)\dd r \right\rrbracket_{\mcC^{1+\alpha H}_{\tint} L^m_{\Omega;w}} + \left\llbracket \int_0^{\,\cdot\,} u(X^\psi_r)\dd r \right\rrbracket_{\mcC^{1+\alpha H}_{\tint} L^m_{\Omega;w}} + \lambda \left\llbracket\int_0^{\cdot} Y^\psi_r \dd r \right\rrbracket_{\mcC^{1+\alpha H}_{\tint} L^m_{\Omega;w}}\\
    \lesssim \, &\, \sqrt{m}\, C\;(1+\|g\|_{\mcC^{\alpha}_x})^{1-\frac{\alpha H}{1-H(\alpha-1)}} +  \left(\|X^\psi\|_{L^\infty_{\tint} L^m_{\Omega;w}} + \|Y^\psi\|_{L^\infty_{\tint}L^m_{\Omega;w}}\right)\\
    \leq \,&\, \sqrt{m}\, C\; (1+\|g\|_{\mcC^{\alpha}_x})^{1-\frac{\alpha H}{1-H(\alpha-1)}} +  \left(\|\rho\|_{L^\infty_{\tint} L^m_{\Omega;w}} + 2\|Y^\psi\|_{L^\infty_{\tint}L^m_{\Omega;w}}\right).
\end{aligned}
\end{equation}
   We plug \eqref{eq:rho_holder} into \eqref{eq:rhogx_bd}, and set $\iota = 3 - \frac{2\alpha H}{1 - H(\alpha-1)}$, so that using $X^{\psi} = Y^{\psi} + \rho$ again and consolidating various terms we obtain 
\begin{align*}
\inv{ (1+ \norm{ g }_{\mcC^{\alpha}_x})^{\iota} } \sup_{t\in \tint } \left\|\int_0^t e^{-\lambda (t-r)}  \rho_r\cdot  g(X_r)   \dd r \right\|_{L^{m}_{\Omega;w}}  \leq \, &\, C 2^{-cV} \sqrt{m} \norm{ \rho }_{L^{\infty}\llmw} \\ 
&\,  +  C 2^{-cV} \left( \big\| Y^{\psi} \big\|_{L^{\infty}_{\tint}\llmw} + \norm{ \rho }_{L^{\infty}_{\tint}\llmw}^2 + \big\| Y^{\psi} \big\|^2_{L^{\infty}_{\tint}\llmw} \right) \\
&\, +  \left(1+C 2^{(1-c)V} \right) \norm{ \rho }_{L^{\infty}_{\tint}\llmw},
\end{align*}
where the square terms in the second bracket came from using Young's product inequality after applying the bound \eqref{eq:stationary_theta_bound} to the first line in \eqref{eq:rhogx_bd}. We come back to \eqref{eq:rho_squared}. We use \eqref{eq:ou_tightness}, plug the bound above into \eqref{eq:rho_squared}, fix $V$ such that $(1+\norm{ g }_{\mcC^{\alpha}_x} )^{\iota} C 2^{-cV} = \inv{2}$ and thus obtain, for some new $\mu > 0$, which comes from $2^{(1-c)V}$, but we do not specify, that
\begin{align*}
\sup_{t\in \tint } \left\|\int_0^t e^{-\lambda (t-r)}  \rho_r\cdot  g(X_r)   \dd r \right\|_{L^{m}_{\Omega;w}} & \leq \inv{2} \norm{ \rho }^2_{L^{\infty}\lldmw} + C \norm{ \rho }_{L^{\infty}_{\tint}\lldmw} \left( \sqrt{m} + \left( 1 + \norm{ g }_{\mcC^{\alpha}_x} \right)^{\mu} \right) + C m \\ 
\;\; & + C\; (1+\norm{g }_{\mcC^{\alpha}_x})^{\iota}\;\left( \norm{ \msA w}_{\mcC^{\gamma}_{\tint;\lambda}} + \norm{\psi }_{\mcC^{\gamma}_{\tint;\lambda}}\right) \lor \left( \norm{ \msA w}_{\mcC^{\gamma}_{\tint;\lambda}} + \norm{\psi }_{\mcC^{\gamma}_{\tint;\lambda}}\right)^2.
 \end{align*}
Therefore, we apply $\norm{\,\cdot\,}_{\llmw}$ to \eqref{eq:rho_squared}, apply the bound above to the second term, plug \eqref{eq:ou_tightness} into the last term and finally we obtain the closed form bound using the following implication:
\[ \left( \forall x \geq 0 \;\; x^2 \leq A x + B \right) \Rightarrow x \leq 2A + \sqrt{ 2B}. \]
Let us set $\bar{\mu} \coloneqq \mu \lor \iota$ for convenience. Then we have, using the simple bound $\sqrt{x} \leq 1 + x$ on the $\norm{ \,\cdot\,}_{\mcC^{\gamma}_{\tint;\lambda}}$ term, that
\begin{equation}\label{eq:rho_closed_bound}
\norm{ \rho }_{L^{\infty}_\mcI\lldmw} \ls_{H,\alpha,d,\lambda}  \left( \sqrt{m} + 1 + (1+\norm{g }_{\mcC^{\gamma}_x} )^{\bar{\mu} } +  \norm{ \msA w}_{\mcC^{\gamma}_{\tint;\lambda}} + \norm{\psi }_{\mcC^{\gamma}_{\tint;\lambda}} \right). 
\end{equation}
The bound on the H\"older norm follows by writing (recall $1+\alpha H> H$, as $\alpha > \inv{2} - \inv{2H} > 1 - \inv{H}$ for all $H \in (0,1)$):
\[ \big\|X^{\psi} \big\|_{\mcC^{H}_{\tint}\llmw} \leq \norm{ \rho }_{\mcC^{1+\alpha H}_\mcI\llmw} + \big\| Y^{\psi } \big\|_{\mcC^H_{\tint}\llmw} \]
and plugging  \eqref{eq:rho_closed_bound} back into \eqref{eq:rho_holder} in order to bound the first term, and bounding the second term with \eqref{eq:ypsi_holder}. The fact that $\kappa_0 \to 0$ as $\norm{ g }_{\mcC^{\alpha}_x} \to \infty$ enters through the first term in \eqref{eq:rho_holder}.

\noindent The final step is to specify the modifications necessary to obtain the analogous bound in $L^m_{\Omega}$, that is \eqref{eq:sol_global_tightness}. The steps are identical, except that in \eqref{eq:ypsi_lm_bound_triangle} and \eqref{eq:ypsi_holder} we replace $W^{H;1}_{0,\cdot}$ with $\wh$, and we observe that every bound above holds with $w = 0$, and we additionally assume that $\psi \in \mcC^{\gamma}_{\RR_+;\lambda}$, which allows us to obtain the same estimate with $\tint = \RR_+$. Note that the first term in \eqref{eq:ypsi_lm_bound_triangle} is finite by Lemma \ref{lem:gaussian_ou_bound}. The first assumption of that Lemma is then satisfied with $\beta=\beta'=H$, the second one is trivially satisfied because increments of fBm  on disjoint intervals are negatively correlated (see Lemma 10.8 in \cite{friz_hairer_20_book} for instance).
\end{proof}
\begin{rem}\label{rem:alpha_H_weak_existence}
Existence of weak solutions to \eqref{eq:intro_sde} under Assumptions~\ref{ass:tightness_singularity} and~\ref{ass:confining_assumption} follows directly from Theorem~\ref{th:tightness_main} by a compactness argument. For comparison see \cite[Sec.~8]{galeati_gerencser_22_subcritical} for a proof in the case without the unbounded coefficient $u$ that we include here. However, since our focus is on unique ergodicity which we obtain by an application of the Doob--Khas'minskii type result for fBm driven SDE give by \cite{hairer_ohashi_07_ergodic}. To this end we require stability estimates which we are only able to show under the more restrictive assumption $\alpha >1-\frac{1}{2H}$. Let us remark that recently \cite{butkovsky_mytnik_24_weak} obtained weak uniqueness of solutions to fBm driven SDEs under exactly our Assumption~\ref{ass:tightness_singularity}but without allowing for an unbounded coefficient. It is an interesting question whether unique ergodicity also holds under this more general assumption.
\end{rem}
We have two almost direct corollaries to Theorem~\ref{th:tightness_main}, first giving us global control on the remainder and secondly existence of weakly stationary solutions.
\begin{cor}\label{cor:theta_global_closed_bound}
	Let $(\alpha,H)\in (-\infty,0)\times (0,\nicefrac{1}{2})$ satisfy Assumption~\ref{ass:tightness_singularity} and $u:\mbR^d \to \mbR^d$ satisfy Assumption~\ref{ass:confining_assumption}, $g\in C^\infty(\mbR^d;\mbR^d)$, $x \in L^m_\Omega \mbR^d$ for all $m\geq 1$ be and be such that $\mbE[\exp(\tilde{\kappa}|x|^2)] <\infty$ for some $\tilde{\kappa}>0$, $\psi \in C_{\mbR_+}$ and $X^\psi$ be an associated strong solution to the SDE \eqref{eq:psi_sde}.  Then, setting $\theta^{\psi}\coloneqq X^{\psi}-W^H - \psi$, for any $T\geq 0$ such that $\psi\in \mcC^{\gamma}_{T;\lambda}$, there exists a $C\coloneqq C(\alpha,H,d,\|g\|_{\mcC^{\alpha}_x},\lambda,\|\psi\|_{\mcC^{\gamma}_{T;\lambda}})>0$ such that
\begin{equation}\label{eq:theta_global_closed_bound}
           \sup_{t \in [0,T]}\;\llbracket \theta^{\psi} \rrbracket_{\mcC^{1+\alpha H}_{[t,t+1 \land T - t]}L^m_{\Omega,w}} \leq  C \Big(1+\|x\|_{L^{2m}_\Omega}  +\sqrt{m}\, \Big) + \norm{ \psi }_{\mcC^\gamma_{T;\lambda}} + \norm{ \msA w }_{\mcC^\gamma_{T;\lambda}}.
     \end{equation}
     Moreover, the same inequality holds if we replace $L^m_{\Omega;w}$ with $\llm$ everywhere and remove the dependence on $w$ from the right hand side. If $\psi \in \mcC^{\gamma}_{\RR_+;\lambda}$, then under the full moments $L^m_\Omega$ we may take the supremum over $t\in \mbR_+$ on the left hand side and rewrite analogous bound on the right hand side, but with $w=0$.
 \end{cor}
\begin{proof}
  The proof follows directly from Lemma~\ref{lem:global_theta_bound}, and applying \eqref{eq:rho_closed_bound} and \eqref{eq:ou_tightness} in the proof of Theorem \ref{th:tightness_main} to bound $\norm{X^{\psi}}_{L^{\infty}_{[0,T]}\lldmw}$, applying triangle inequality where necessary. The modification to $\llmw$ holds through the analogous one in Lemma \ref{lem:global_theta_bound}. 
\end{proof}
\begin{cor}\label{cor:weak_existence}
   Let $(\alpha,H)\in (-\infty,0)\times (0,\nicefrac{1}{2})$ satisfy Assumption~\ref{ass:tightness_singularity}, $u:\mbR^d \to \mbR^d$ satisfy Assumption~\ref{ass:confining_assumption}, $g\in B^{\alpha'}_{\infty,\infty}(\mbR^d;\mbR^d)$ for any $\alpha'\in (1-\nicefrac{1}{(2H)},\alpha)$, $x \in L^m_\Omega \mbR^d$ for all $m\geq 1$ be and be such that $\mbE[\exp(\tilde{\kappa}|x|^2)] <\infty$. Then there exists a weakly stationary solution $\{X_t\}_{t\geq 0}$ to \eqref{eq:psi_sde} (with $\psi =0$).  In addition, for any $t>0$, $X_t$ has Gaussian tails.
\end{cor}
\begin{proof}
The a priori estimate of Theorem~\ref{th:tightness_main} is enough to construct a global weak solution by the same arguments as the proof of \cite[Thm.~8.2]{galeati_gerencser_22_subcritical}. In addition, it follows from \eqref{eq:sol_global_tightness} that the family of measures
\begin{equation*}
    \mu_T \coloneqq \frac{1}{T}\int_0^T \mcL\left(X|_{[t,t+1\wedge(T-t)]}\, \right) \dd t \in \mcP(C(\mbR_+;\mbR^d))
\end{equation*}
is tight. By Prokhorov's theorem, therefore, there exists a weakly converging subsequence with limit $\mu_\infty \in \mcP(C(\mbR_+;\mbR^d))$. By construction $\mu_\infty$ is invariant with respect to the shift operator on $C(\mbR_+;\mbR^d)$. Choosing $X$ to be the canonical process under this measure gives a weakly stationary solution.
\end{proof}
\section{Construction of the SDS}\label{sec:sds_construction}
 While Assumption~\ref{ass:tightness_singularity} was sufficient to prove a priori Gaussian tightness of solutions to \eqref{eq:psi_sde} (and a fortiori existence of solutions to \eqref{eq:intro_fbm_sde} see Remark~\ref{rem:alpha_H_weak_existence}) in order to construct a bona fide SDS we require the following stronger assumption which will be in force throughout the remainder of the paper.
\begin{ass}\label{ass:stability_singularity}
	Let $(\alpha,H) \in (-\infty,0)\times (0,\nicefrac{1}{2})$ be such that 
\begin{equation}\label{eq:stability_singular_parameters}
	\alpha > 1 -\frac{1}{2H}.
	\end{equation}
\end{ass}
\noindent For $(\alpha,H)$ satisfying \eqref{ass:stability_singularity} we specify once and for all in this and subsequent sections 
\begin{equation}\label{eq:defn_beta_eps}
     \beta \coloneqq 1 + H(\alpha-1) \in (\nicefrac{1}{2},1) \quad \text{and} \quad     \eps \, \in (0,\beta-\nicefrac{1}{2}),
\end{equation}
%
%
In addition we consider a tuple $(\Omega,\mcF,\mbF^B,\mbP,B)$ of a fixed probability space $(\Omega,\mcF,\mbP)$, a standard $\mbR^d$-valued $\mbP$-Brownian motion $B$ and $\mbF^B$ its natural filtration.
\subsection{Local Stability Estimates}\label{sec:local_stability}
In this section we provide necessary estimates to conclude local uniqueness of  solutions and stability in relevant arguments. We start with a simple gradient estimate on the averaged field with perturbation that satisfies certain regularity hypothesis. 
\begin{lem}\label{lem:averaged_field_tightness}
Let $T>0$, $(\alpha,H)\in (-\infty,0)\times (0,\nicefrac{1}{2})$ satisfy Assumption~\ref{ass:stability_singularity}, $u:\mbR^d \to \mbR^d$ satisfy Assumption~\ref{ass:confining_assumption}, $g\in C^\infty_b(\mbR^d;\mbR^d)$ and $\varphi \in \mcC^{1+\alpha H}_{T} L^m(\Omega;\mbR^d)$ be an $\mbF^B$-adapted process for which there exists a constants $C_1,\,C_2 >0$ such that for all $m\geq 1$ and $w\in \mcW^{-}_{(\gamma,\delta)}$
\begin{equation}\label{eq:averaged_field_tightness_assumption}
\;\llbracket \varphi \rrbracket_{\mcC^{1+\alpha H}_{T}\llmw}  \vee \llbracket \varphi \rrbracket_{\mcC^{1+\alpha H}_{T}L^m_\Omega}  \leq C_1 + C_2 \sqrt{m}.
\end{equation} 
Then, for $(\beta,\eps)$ satisfying \eqref{eq:defn_beta_eps}, there exists a $\kappa_1 \coloneqq \kappa_1(\eps,\alpha,H,d,\|g\|_{\mcC^\alpha_x}, C_2) > 0$ (monotone non-increasing in $\|g\|_{\mcC^\alpha_x}$) such that for all $\kappa  \in [0,\kappa_1)$ and $\psi \in C_T$ there exists a\\ $C_3\coloneqq C_3 (\eps,\alpha,H,d,\|g\|_{\mcC^\alpha_x}, C_2,C_1,\|\psi\|_{C_T},\|\msA w\|_{C_T})>0$ such that
\begin{equation}\label{eq:avg_grad_field_time_holder}
	\sup_{t \in [0,T]}\EE^B_w \Bigg[\exp\Bigg( \kappa  \left\llbracket\int_t^{\,\cdot\,} (\nabla g)(\varphi_r +\wh_r +\psi_r) \,\dd r \right\rrbracket^2_{\mcC^{\beta-\eps}_{[t,t+1\wedge(T-t)]}}   \Bigg) \Bigg] \leq  C_3.
\end{equation}
Moreover, the same estimate holds without change if $\mbE^B_w$ is replaced by $\mbE^B$ on the left hand side with the dependence on $\|\msA w\|_{C_T}$ removed from $C_3$ on the right hand side. Finally, a bound analogous to \eqref{eq:avg_grad_field_time_holder} holds in both cases considered for intervals of length $\tau > 1$, in which case the constant $\kappa_1$ (resp. $C_3$) also depends in a non-increasing (resp. non-decreasing) manner  on $\tau$.
\end{lem}
\begin{proof}
	The proof is similar in spirit to that of Theorem~\ref{th:tightness_main}, only we are required to first prove a version of Lemma~\ref{lem:integral_estimate} with $g$ replaced by $\nabla g$ and under the new condition \eqref{eq:stability_singular_parameters}. We give the full proof in the case of $\psi\neq 0$ and for the disintegrated expectation $\mbE^B_w$, the other case being similar. We also assume $T>1$ throughout, the proof for $T\leq1$ being almost identical. Let,  us define the two parameter process $A^\psi:[0,T]^2_{\leq}\to \mbR^d$, by setting
	\begin{equation*}
		A^\psi_{s,t} \coloneqq \int_s^t \mbE^B_{w,s}[\nabla g(\varphi_s + W^H_r+\psi_r)]\,\dd r,  \quad \text{for all } \, (s,t)\in [0,T]^2_{\leq}.
	\end{equation*}
Appealing to Lemma~\ref{lem:lnd_to_heat_kernel} and \eqref{eq:heat_kernel_fbm}, for any $|t-s|\leq 1$, we directly have 
	\begin{align*}
		\|A^\psi_{s,t}\|_{\llmw} \,&\leq\,  \int_s^t \| \mcG^H_{|r-s|} \nabla g\, \|_{L^\infty_x} \dd r \lesssim_{\,\alpha,H,d} \|g\|_{\mcC^\alpha_x}  |t-s|^{1+(\alpha-1) H},
	\end{align*}
	where we used \eqref{eq:stability_singular_parameters} to check that $(\alpha-1)H > -\frac{1}{2}$ allowing us to evaluate the integral.
	
\noindent To check the conditions of Lemma~\ref{lem:sewing_stochastic_integral} we proceed as in the remainder of the proof of Lemma~\ref{lem:integral_estimate}. So that applying \eqref{eq:disintegrated_lnd_to_heat_kernel_stoch_integral} in Lemma~\ref{lem:lnd_to_heat_kernel} we now have
	\begin{align*}
		\delta A^\psi_{s,u,t} =&\, \int_s^u h_z \dd B_z + J_{s,t}
	\end{align*}
	with 
	\begin{align*}
		h^s_z = &\, \int_u^t (r-z)^{H-\nicefrac{1}{2}} \nabla \mcG^H_{|r-z|} \nabla g(\varphi_s+\mbE^B_{w,z} [W^H_r]+\psi_r)\,\dd r,\\
		J^u_{s,t} = &\, \int_u^t \EE^B_{w,u} \left[\nabla g(\varphi_u+\wh_r+\psi_r) - \nabla g(\varphi_s+\wh_r+\psi_r) \right]\dd r.
	\end{align*}
	Using the same estimates as in the proof of Lemma~\ref{lem:integral_estimate} but appealing to \eqref{eq:stability_singular_parameters} instead of \eqref{eq:tightness_singular_parameters} to check that $1+(\alpha-1)H>\frac{1}{2}$, there exists a constant $C\coloneqq C(\alpha, H,d)>0$ such that for all $|t-s|\leq 1$,
	\begin{equation}
		\|h^s\|_{L^2_{[s,t]}} \lesssim_{\alpha,H,d} C \norm{g}_{\mcC^{\alpha}_x}  |t-s|^{1+(\alpha-1) H }.
	\end{equation}
	Similarly, for all $|t-s|\leq 1$ we obtain the bounds 
	\begin{align*}
		\|J^u_{s,t}\|_{\llmw} & \lesssim_{\, \alpha,H,d} \,  \|g\|_{\mcC^\alpha_x} \int_u^t |r-u|^{H(\alpha-2)}\,\|\varphi_u-\varphi_s\|_{L^m_\Omega} \dd r,\\
		& \leq \, \llbracket\varphi\rrbracket_{\mcC^{1+\alpha H}_{[s,u]}\llmw} \|g\|_{\mcC^\alpha_x} |u-s|^{1+\alpha H} \int_u^t|r-u|^{H(\alpha-2)} \dd r\\
		& \leq \, \llbracket\varphi\rrbracket_{\mcC^{1+\alpha H}_{[s,t]}\llmw}\|g\|_{\mcC^\alpha_x} |t-s|^{2+2(\alpha-1)H},
	\end{align*}
	where we again use \eqref{eq:stability_singular_parameters} to check that $2+2(\alpha-1)H = 2(1+(\alpha-1)H) >1$.
	
	Therefore, by Lemma~\ref{lem:sewing_stochastic_integral}, any $0\leq s<t\leq T$ such that $|t-s|\leq 1$ we have
	\begin{align*}
		\left\llbracket\int_s^{\,\cdot\,} \nabla g(\wh_r + \varphi_r)\dd r \right\rrbracket_{\mcC^{1+(\alpha-1) H}_{[s,t]}\llmw} \lesssim_{\alpha,H,d} & \|g\|_{\mcC^\alpha_x}  \left(1+\sqrt{m}+ \llbracket\varphi\rrbracket_{\mcC^{1+\alpha H}_{[s,t]}\llmw} |t-s|^{1+H(\alpha-1)} \right)\\
		\leq & C\|g\|_{\mcC^\alpha_x}  \left(1+\sqrt{m}(C_1+ C_2 |t-s|^{1+H(\alpha-1)}) \right),
	\end{align*}
	where we inserted the assumed bound on $\varphi$ in the second line. The conclusion then follows in a similar manner as that of Theorem~\ref{th:tightness_main}, in particular appealing to \cite[Lem.~A.2]{galeati_gerencser_22_subcritical}.  It is then clear, using \eqref{eq:lnd_to_heat_kernel_stoch_integral} from Lemma~\ref{lem:lnd_to_heat_kernel} that the same statement holds with $\mbE^B_w$ replaced by $\mbE^B$. The independence the final bound and that of $\kappa$ from the parameter $T$ can be inferred from the last inequality. The last statement follows either by scaling or by patching together the subintervals of size one, using the bound above. The claim that new $\kappa_1$ will be decreasing follows from the fact that this procedure will create a factor depending in the increasing way on $\tau$ which will multiply $\sqrt{m}$.
\end{proof}
Building upon the previous estimate we give a useful stability estimate that will be applied here and in the next section.
\begin{lem}\label{lem:local_stability}
Let $T> 0$, $(\alpha,H)\in (-\infty,0)\times (0,\nicefrac{1}{2})$ satisfy  Assumption~\ref{ass:stability_singularity}, $g^1,\,g^2 \in C^\infty_b(\mbR^d;\mbR^d)$ $u:\mbR^d\to \mbR^d$ satisfy Assumption~\ref{ass:confining_assumption} and $x^1,\,x^2 \in \mbR^d$ and $\psi \in \mcC^\gamma_{T}$ (or $\psi \in \mcC^{\gamma}_{\mbR_+;\lambda}$). Then, given $X^1,\,X^2$ solutions to \eqref{eq:psi_sde} with initial data $x_1,\,x_2$ and $g$ replaced by $g^1$ and $g^2$ respectively, for any $1 \leq m < m' <+\infty$, there exists a constant $C\coloneqq C(\alpha,H,d,\lambda,\|g^2\|_{\mcC^\alpha_x},m,m',\norm{\psi}_{\mcC^\gamma_{T}})>0$ (such that $C\to+\infty$ as $\|g\|_{\mcC^{\alpha}_x}\to 0$)  such that for all $t\in [0,T]$,
    \begin{equation}\label{eq:sol_drift_initial_stability}
        \Bigg\| \sup_{s \in [0,1\wedge T-t]} \absv{ X^1_{t+s} - X^2_{t+s} } \Bigg\|_{L^m_{\Omega}} \leq C \left( \norm{ g^1 - g^2 }_{\mcC^{\alpha}_x}+\norm{ X^1_t - X^2_t}_{L^{m'}_{\Omega}}   \right). 
    \end{equation}  
Where the same estimate holds for any $t>0$ if instead $\psi \in \mcC^{\gamma}_{\mbR_+;\lambda}$ and the constant depends on $\|\psi\|_{\mcC^{\gamma}_{\mbR_+;\lambda}}$ instead of $\|\psi\|_{\mcC^{\gamma}_T}$.
\end{lem}
\begin{proof}
For $i=1,\,2$ let us define $\tilde{\theta}^i_{\,\cdot\,} \coloneqq \theta^i_{t+\,\cdot\,} = X^i_{t+\,\cdot\,} - W^H_{t,v+\,\cdot\,} - X^i_{t} - \psi_{t,v+\,\cdot\,}$ so that for any $(u,v)\in [0,1]^2_\leq$
\begin{equation}\label{eq:theta_xy_diff}
\begin{aligned}
     \tilde{\theta}^1_{u,v} - \tilde{\theta}^2_{u,v}   = &\,   X^1_{t+s,t+v} - X^2_{t+u,t+v}\\
     =  &\, \int_{t+u}^{t+v} \left(g^1(X^1_r) - g^2(X^2_r)\right) \dd r  + \int_{t+u}^{t+v} \left( u(X^1_r) - u(X^2_r) \right) \dd r. 
     \end{aligned}
     \end{equation}
Due to Assumption \ref{ass:confining_assumption} it holds that
\begin{align}
    \left| \int_{t+u}^{t+v} \left( u(X^1_r) - u(X^2_r) \right) \dd r\right|  \leq &\,  \lambda |t-s| \left\|X^1-X^2\right\|_{L^\infty_{[t,t+1]}}\notag \\
    \leq&\,  \lambda |t-s| \Big\|\tilde{\theta}^1-\tilde{\theta}^2\Big\|_{ L^\infty_{[0,1]}} + \lambda |t-s|\left|X^1_t-X^2_t\right|. \label{eq:sol_stable_u_bnd}
\end{align}
Concerning the difference on the right hand side we write 
\begin{align*}
	 \int_{t+u}^{t+v} \left(g^1(X^1_r) -g^2(X^2_r)  \right) \dd r    =&\, \int_{t+u}^{t+v} \left(g^1(\tilde{\theta}^1_r +\wh_{t,r} + X^1_t + \psi_{t,r}) - g^1(\tilde{\theta}^2_r +\wh_{t,r} + X^2_t + \psi_{t,r}) \right) \dd r  \\
     & + \int_{t+u}^{t+v} \left(g^1-g^2\right)(\tilde{\theta}^2_r +\wh_{t,r} + X^2_t + \psi_{t,r}) \dd r  \\
     \eqqcolon\,  & \, \RN{1}_{u,v} + \RN{2}_{u,v}.
	\end{align*}
By Taylor's theorem we have
\begin{align*}
    \RN{1}_{u,v} =&\, \int_u^v (\tilde{\theta}^1_r - \tilde{\theta}^2_r + X^2_t -X^1_t) \cdot \int_0^1 \nabla g^1(\zeta (\tilde{\theta}^1_r+X^1_t) + (1-\zeta)(\tilde{\theta}^2_r+X^2_t) + W^H_{t,r} +\psi_{t,r}) \dd \zeta \dd r.
\end{align*}
Interpreting the above expression as the Young integral (see \cite{You36}) of $r\mapsto \tilde{\theta}^1_r - \tilde{\theta}^2_r + X^2_t -X^1_t$ against the process
\begin{equation*}
    v\mapsto A^1_v \, \coloneqq\,  \int_0^v \int_0^1 \nabla g^1\left(\zeta (\tilde{\theta}^1_r+X^1_t) + (1-\zeta)(\tilde{\theta}^2_r+X^2_t) + W^H_{t,r} +\psi_{t,r}\right) \dd \zeta \dd r
\end{equation*}
by standard estimates on Young integrals (e.g \cite[Eq.~(4.3)]{friz_hairer_20_book}) it holds that
\begin{equation}\label{eq:sol_stable_young_bnd}
\begin{aligned}
    |\RN{1}_{u,v}|
      \, \leq \,  \llbracket A^1\rrbracket_{\mcC^{\beta-\eps}_{[t,t+1\wedge (T-t)]}}\Big(\,\,  & \big\llbracket \tilde{\theta}^1 - \tilde{\theta}^2 \, \big\rrbracket_{\mcC^{\beta-\eps}_{[0,1]}} \absv{v-u}^{2(\beta-\eps)} + \big\| \tilde{\theta}^1 - \tilde{\theta}^2 \big\|_{L^{\infty}_{[0,1]}} \absv{v-u}^{\beta-\eps} \\
      &+ |X^1_t-X^2_t||v-u|^{\beta-\eps} \Big).
      \end{aligned}
\end{equation}
Recall that since we take $g\in C^\infty_b(\mbR^d;\mbR^d)$ the right hand side is non-trivial. Similarly, we directly have
 \begin{equation}\label{eq:sol_stable_diff_g_bnd}
	    \absv{ \RN{2}_{u,v} } \leq \left\llbracket \int_0^{\cdot} (g^1 - g^2)(\tilde{\theta}^2_r +\wh_{T,r} + X^2_t + \psi_{T,r}) \dd r \right\rrbracket_{\mcC^{\beta+H-\eps}_{[t,t+1\wedge (T-t)]}} \absv{v-u}^{\beta+H-\eps}.
	\end{equation} 
Therefore, combining \eqref{eq:sol_stable_u_bnd}, \eqref{eq:sol_stable_young_bnd}, \eqref{eq:sol_stable_diff_g_bnd} and combining terms, we find that for all $(u,v)\in [0,1]^2_{\leq}$
\begin{align*}
    \left| \tilde{\theta}^1_{u,v} - \tilde{\theta}^2_{u,v} \right| \leq\, &\,  \llbracket A^1\, \rrbracket_{\mcC^{\beta-\eps}_{[t,t+1\wedge (T-t)]}} \llbracket \tilde{\theta}^1 - \tilde{\theta}^2 \rrbracket_{\mcC^{\beta-\eps}_{[0,1]}} \absv{v-u}^{2(\beta-\eps)}\\
    &\, + \left(\llbracket A^1\rrbracket_{\mcC^{\beta-\eps}_{[t,t+1\wedge (T-t)]}}+\lambda\right)\big\| \tilde{\theta}^1 - \tilde{\theta}^2 \big\|_{L^{\infty}_{[0,1]}} \absv{v-u}^{\beta-\eps} \\
    &\, + \left(\left\llbracket \int_0^{\cdot} (g^1 - g^2)(\tilde{\theta}^2_r +\wh_{T,r} + X^2_t + \psi_{T,r}) \dd r \right\rrbracket_{\mcC^{\beta+H-\eps}_{[T,T+1]}} + (1+\lambda)|X^1_t-X^2_t|\right) \absv{t-s}^{\beta-\eps}.
\end{align*}
 We set
\begin{equation}\label{eq:constants_for_gronwall}
\begin{aligned}
	& C_1 \coloneqq \llbracket A^1\rrbracket_{\mcC^{\beta-\eps}_{[t,t+1\wedge (T-t)]}}, \quad C_2 \coloneqq \llbracket A^1\rrbracket_{\mcC^{\beta-\eps}_{[t,t+1\wedge (T-t)]}} + \lambda ,\\
    &C_3 \coloneqq \left\llbracket \int_0^{\cdot} (g^1 - g^2)(\tilde{\theta}^2_r +\wh_{T,r} + X^2_t + \psi_{T,r}) \dd r \right\rrbracket_{\mcC^{\beta+H-\eps}_{[t,t+1\wedge (T-t)]}} + (1+\lambda)|X^1_t-X^2_t|, \\ 
    & \alpha_1 = 2\alpha_2 = 2\alpha_3 = 2( \beta-\eps),\quad \eta \coloneqq \beta - \eps,
    \end{aligned}
\end{equation}
so that by a combination of Lemma~\ref{lem:averaged_field_tightness}, Lemma~\ref{lem:integral_estimate} and Corollary~\ref{cor:theta_global_closed_bound}
we check that this collection satisfies the assumptions of Corollary~\ref{cor:gaussian_gronwall} and so there exists a $C\coloneqq C(\alpha,H,d,\lambda^{-1},\|g^2\|_{\mcC^{\alpha}_x},m,m',\|\psi\|_{\mcC^{\gamma}_{T}})>0$ (monotone non-decreasing in $\|g\|_{\mcC^\alpha_x}$) such that
\begin{equation*}
    \left\| \sup_{s\in [0,1\wedge (T-t)]} |\tilde{\theta}^1_{s} - \tilde{\theta}^2_s| + \left\|\tilde{\theta}^1 - \tilde{\theta}^2 \right \|_{\mcC^{\beta-\eps}_{[0,1]}}\right\|_{L^m_\Omega} \leq C\left(\norm{ g^1 - g^2 }_{\mcC^{\alpha}_x}+\norm{ X^1_t - X^2_t}_{L^{m'}_{\omega}}   \right)
\end{equation*}
which by a further application of the triangle inequality and the identity $\tilde{\theta}^1_t - \tilde{\theta}^2_t = X^1_{t+v} - X^2_{t+v} - X^1_t + X^2_t$ yields the claimed result. The fact that the constant $C$ in the statement is independent from $T$ or $t$ follows from the properties of $C_1, C_2, C_3$ established in Lemmas \ref{lem:averaged_field_tightness}, \ref{lem:integral_estimate} and Theorem \ref{th:tightness_main}.
\end{proof}
\begin{rem} By working with conditional moments we expect that \eqref{eq:sol_drift_initial_stability} could be obtained with both $m$-moments on the left and right hand sides, see for example \cite[Thm.~3.2]{galeati_gerencser_22_subcritical} and \cite[Cor.~3.6]{galeati_le_mayorcas_24_quantitative}. However, since we are essentially focused on constructing solutions from deterministic initial data and propagating $m$-moments for all times we do not pursue this extra argument.
\end{rem}
\begin{prop}\label{prop:perturbation_continuity_slow}
    Let $(\alpha,H)\in (-\infty,0)\times (0,\nicefrac{1}{2})$ satisfy Assumption~\ref{ass:stability_singularity}, $u:\mbR^d \to \mbR^d$ satisfy Assumption~\ref{ass:confining_assumption}, $g\in C^\infty(\mbR^d;\mbR^d)$, $T>0$, $x\in \mbR^d$, $\psi \in \mcC^\gamma_{T}$ and $X,\, X^\psi$ be the associated strong solutions to \eqref{eq:psi_sde} on $[0,T]$. Then, setting $\theta = X - \wh $,   $\theta^{\psi} = X^{\psi} - \wh - \psi$ there exists a $\chi_0\coloneqq \chi_0(\alpha,H) \in (0,1)$ such that for all $m\in [1,+\infty)$, $\chi \in (0,\chi_0)$ and $w\in \mcW^{-}_{\gamma,\delta}$  there exists a $C\coloneqq C(\alpha,H,d,m,\lambda, \|g\|_{\mcC^{\alpha}_x},|x|,\norm{ \msA w }_{\mcC^\gamma_{T}}, T )>0$ such that
\begin{equation*}
   \bigg\| \sup_{t\in [0,T]} \absv{ X^{\psi}_t - X_t } \bigg\|_{\llmw} + \,  \bigg\| \sup_{t\in [0,T]} \absv{ \theta^{\psi}_t - \theta_t } \bigg\|_{\llmw} \leq C \left(   \norm{ \psi }_{\mcC^\gamma_{T}} + \norm{ \psi}_{L^{\infty}_T}^{\chi} \right).
\end{equation*} 
  Moreover, up to removing the dependence of $C$ on $\norm{\msA w }_{\mcC^\gamma_{T}}$, the result remains unchanged if $L^m_{\Omega;w}$ is replaced with $L^m_{\Omega}$ on the left hand side.
\end{prop}
\begin{proof}
The proof follows along similar lines as that of Lemma~\ref{lem:local_stability}, only we use an interpolation argument to obtain the dependence on $\|\psi\|_{L^\infty_T}$ on the right hand side. We directly have
\begin{equation}\label{eq:theta_stability_bnd_1}
    \theta^{\psi}_{s,t} - \theta_{s,t}  \, =  \int_s^t\left( g(X^{\psi}_r)- g(X_r)\right) \dd r +\int_s^t \left(u(X^{\psi}_r) -  u(X_r)\right) \dd r.
\end{equation}
Due to the Lipschitz assumption on $u$ \eqref{eq:confining_assumption} and the identity $X^{\psi} - X = \theta^{\psi} - \theta + \psi$, one immediately has
\begin{equation}\label{eq:theta_stable_u_bnd_1}
    \bigg| \int_s^t u(X^{\psi}_r) \dd r - \int_s^t u(X_r) \dd r \bigg| \leq \lambda |t-s| \left(\sup_{r\in [s,t]}|\theta^{\psi}_{r} - \theta_{r}| +\sup_{r\in [s,t]}|\psi_{r}|\right).
\end{equation}
Regarding the first term on the hand side of \eqref{eq:theta_stability_bnd_1} we separate it into two further terms 
\begin{align*}
& \int_s^t \left(g(X^\psi_r) - g(X_r)\right) \dd r \\ 
%
%
& = \underbrace{ \int_s^t \Big(g(\theta^{\psi}_r +\wh_r  + \psi_r) -  g(\theta^{\psi}_r+\wh_r )\Big) \dd r }_{\eqqcolon \RN{1}_{s,t}} +\underbrace{ \int_s^t \Big(g( \theta^{\psi}_r+\wh_r ) -g(\theta_r+\wh_r  )\Big) \dd r }_{\eqqcolon \RN{2}_{s,t}}. %
\end{align*}
By the same arguments used to obtain \eqref{eq:sol_stable_young_bnd} we find that
\begin{equation}\label{eq:sol_psi_cont_sing_2}
    |\RN{2}_{s,t}| \, \leq \, \llbracket A^1\rrbracket_{\mcC^{\beta-\eps}_{T}}\, \left( \llbracket \theta^{\psi} - \theta \rrbracket_{\mcC^{\beta-\eps}_{T}} \absv{t-s}^{2(\beta-\eps)} + \big\| \theta^\psi - \theta\big\|_{L^{\infty}_{T}} \absv{t-s}^{\beta-\eps} \right), 
\end{equation}
with $A^1_t \, \coloneqq\,  \int_0^t \int_0^1 \nabla g^1\big(\zeta (\theta^\psi_r + (1-\zeta)\theta^\psi_r + W^H_{r} +\psi_{r}\big) \dd \zeta \dd r$. Hence, for any $(s,t) \in [0,T]^2_\leq$ such that $|t-s|\leq 1$ we have
\begin{equation}\label{eq:sol_psi_cont_omegawise_bnd}
	\begin{aligned}
	 |\theta^{\psi}_{s,t} - \theta_{s,t} | \leq \,  &\, \lambda |t-s| \|\psi\|_{L^\infty_T} +  |\RN{1}_{s,t}| \\
	 & \, + \llbracket A^1\rrbracket_{\mcC^{\beta-\eps}_{T}} \,  \llbracket \theta^{\psi} - \theta \rrbracket_{\mcC^{\beta-\eps}_{T}} \absv{t-s}^{2(\beta-\eps)} + \Big(\llbracket A^1\rrbracket_{\mcC^{\beta-\eps}_{T}} +\lambda \Big) \| \theta^\psi - \theta \|_{L^{\infty}_{T}} \absv{t-s}^{\beta-\eps}.
	 \end{aligned}
\end{equation}
Concerning $\RN{1}$, we cannot apply a Taylor expansion since we only have $\psi \in \mcC^\gamma_T$ for $\gamma\in (0,H)$ and Assumption~\eqref{eq:stability_singular_parameters} does not guarantee that $\beta-\eps + \gamma > 1$ so that Young integration is not applicable. Instead we combine stochastic sewing with an interpolation argument. We detail the remainder of the proof only for the disintegrated moments $\|\,\cdot\,\|_{L^m_{\Omega;w}}$, the  case of full moments being similar. 

\noindent First we define the two parameter process
\begin{equation}\label{eq:bara_st_def}
A_{s,t} := \int_s^t \EE^B_{w,s} \left[g(\theta^{\psi}_s + \wh_r +  \psi_r)-g(\theta^{\psi}_s+\wh_r)\right]  \dd r
\end{equation}
to which we will apply Lemma~\ref{lem:SSL}. First, appealing to Lemma~\ref{lem:lnd_to_heat_kernel} and the heat kernel estimate, \eqref{eq:heat_kernel_fbm} for any $(s,t)\in [0,T]^2_\leq$ such that $|t-s|\leq 1$, we have
\begin{equation}
    \absv{ A_{s,t} } \ls_{d,\alpha, H} \absv{t-s}^{1+H(\alpha-1)} \sup_{r \in [s,t]} \absv{ \psi_r }\|g\|_{\mcC^{\alpha}_x} \,  \leq \, \absv{t-s}^{1+H(\alpha-1)} \| \psi\|_{L^\infty_T}\|g\|_{\mcC^{\alpha}_x}.
\end{equation}
Simply applying the triangle inequality also gives, for any $\bar{m}\in [1,+\infty)$,
\begin{align}\label{eq:bara_st_time_half_bound}
\norm{ \EE^B_{w,s} \delta_{s,u,t} A }_{L^{\bar{m}}_{\Omega;w}} \ls &\, \absv{t-s}^{\beta} \norm{ \psi }_{L^{\infty}_T} \norm{ g }_{\mcC^{\alpha}_x}, \quad (s,u,t)\in [0,T]^2_\leq, \quad |t-s|\leq 1.
\end{align}
Note that again here, $\beta <1$ and so this is not enough to apply Lemma~\ref{lem:SSL}. However, expanding $\delta A$ we also see that for any $(s,u,t)\in [0,T]^3_{\leq}$
\begin{equation}\label{eq:three_point_bara}
\begin{aligned}
    \EE^B_{w,s} \delta A_{s,u,t}  = &\underbrace{ \int_u^t \EE^B_{w,s} \left[g(\theta_s^{\psi} +\wh_r +  \psi_r)-g(\theta^{\psi}_u +\wh_r +  \psi_r)\right] \dd  r}_{\eqqcolon \RN{1}^1_{s,t}} \\
    &- \underbrace{  \int_u^t \EE^B_{w,s} \left[g(\theta^\psi_s+\wh_r  ) - g(\theta^\psi_u+\wh_r  ) \right]\dd r. }_{\eqqcolon \RN{1}^2_{s,t}}
\end{aligned}
\end{equation}
Treating $\RN{1}^1$ and $\RN{1}^2$ separately, applying Lemma~\ref{lem:lnd_to_heat_kernel} and \eqref{eq:heat_kernel_fbm} to each, for any $(s,u,t)\in [0,T]^3_\leq$ suh that $|t-s|\leq 1$, we see that
\begin{equation}\label{eq:goodtime_bound_ione_itwo}
 \norm{ \EE^B_{w,s} \delta \bar{A}_{s,u,t} }_{L^{\bar{m}}_{\Omega;w}} \leq \norm{ \RN{1}^1_{s,t} }_{L^{\bar{m}}_{\Omega}} + \norm{ \RN{1}^2_{s,t} }_{L^{\bar{m}}_{\Omega}} \ls \absv{t-s}^{2\beta + H} \norm{ g }_{\mcC^{\alpha}_x} \llbracket \theta^{\psi} \rrbracket_{\mcC^{\beta+H}_T L^{\bar{m}}_{\Omega}}.
 \end{equation}
Therefore, interpolating between \eqref{eq:bara_st_time_half_bound} and \eqref{eq:goodtime_bound_ione_itwo}, for any $\chi \in (0,1)$ it holds that 
\begin{equation}\label{eq:interpolated_bound}
\norm{ \EE^B_s \delta_{s,u,t} \bar{A} }_{L^{\bar{m}}_{\Omega;w}} \lesssim_{d,\alpha,H,\chi} \absv{t-s}^{(2\beta+H)(1-\chi) + \beta\chi} \llbracket \theta^{\psi} \rrbracket_{\mcC^{\beta+H}_TL^{\bar{m}}_{\Omega;w}}^{1-\chi} \norm{ \psi}_{L^{\infty}_T}^{\chi} \norm{ g }_{\mcC^{\alpha}_x}.
\end{equation}
Choosing $\chi$ such that $(2\beta+H) ( 1 - \chi) + \beta \chi > 1$ (which is always possible due to Assumption~\ref{ass:stability_singularity}) we see that both conditions of Lemma~\ref{lem:SSL} are verified. We let $\chi_0\coloneqq \chi_0(\alpha,H) \in (0,1)$ be the saturation point of this inequality. Hence, after checking that $\RN{1}$ is the unique object obtained from $A$ by Lemma~\ref{lem:SSL}, one sees that there exists some $\iota>\beta-\eps$ such that for any $(s,t)\in [0,T]^2_\leq$  such that $|t-s|\leq 1$
\begin{equation*}
    \norm{ \RN{1}_{s,t} }_{L^{\bar{m}}_{\Omega;w}} \ls_{d,\alpha,H} \left( \norm{ g }_{\mcC^{\alpha}_x} \norm{ \psi }_{L^{\infty}_T} + \llbracket \theta^{\psi} \rrbracket_{\mcC^{\beta+H}_TL^{\bar{m}}_{\Omega;w}}^{1-\chi} \norm{ \psi}_{L^{\infty}_T}^{\chi} \norm{ g }_{\mcC^{\alpha}_x} \right) \absv{t-s}^{\iota}, \quad \iota \in (0,1).
\end{equation*}

Since $\bar{m}\in  [1,+\infty)$ was arbitrary we can take it large enough so that $\iota -\nicefrac{1}{\bar{m}} >\beta-\eps >0$ and so appealing Kolmogorov's continuity theorem (e.g. \cite[Thm.~A.11]{friz_victoir_10_multidimensional}  or \cite[Lem.~A.3]{galeati_gerencser_22_subcritical}) we see that for some $\beta-\eps<\iota' < \iota - \nicefrac{1}{\bar{m}}$, 
\begin{equation*}
	\norm{ \RN{1}}_{L^{\tilde{m}}_{\Omega;w}\mcC^{\iota^\prime}_T} \ls \norm{ g }_{\mcC^{\alpha}_x}\left( \norm{ \psi }_{L^{\infty}_T} + \llbracket \theta^{\psi} \rrbracket_{\mcC^{\beta+H}_TL^{\bar{m}}_{\Omega;w}}^{1-\chi} \norm{ \psi}_{L^{\infty}_T}^{\chi} \right).
\end{equation*} 
Hence, applying Corollary~\ref{cor:theta_global_closed_bound} to bound $\llbracket \theta^{\psi} \rrbracket_{\mcC^{\beta+H}_T}$, there exists a $C\coloneqq C(d,\alpha,H,\lambda,\|g\|_{\mcC^\alpha_x},|x|,\bar{m})>0$ (which explodes as $\|g\|_{\mcC^\alpha_x}\to +\infty$) such that
\begin{equation}\label{eq:sol_psi_cont_sing_1}
\norm{ \RN{1}}_{L^{\tilde{m}}_{\Omega;w}\mcC^{\iota^\prime}_T}   \lesssim_C  \norm{ \psi }_{L^{\infty}_T} +(1+\norm{ \psi }^{1-\chi}_{\mcC^\gamma_{T;\lambda}} +\norm{ \msA w }_{\mcC^\gamma_{T;\lambda}}^{1-\chi}) \norm{ \psi}_{L^{\infty}_T}^{\chi} .
\end{equation}
Feeding this into \eqref{eq:sol_psi_cont_omegawise_bnd} and collecting some terms with the embedding $\mcC^{\gamma}_{T} \hookrightarrow \mcC^{\gamma}_{T;\lambda} \hookrightarrow L^\infty_T$, there exists a $C\coloneqq C(d,\alpha,H,\lambda,\|g\|_{\mcC^\alpha_x},|x|,\bar{m})>0$ such that for all $|t-s|\leq 1$ have
\begin{equation}\label{eq:sol_psi_cont_omegawise_bnd_2}
	\begin{aligned}
		|\theta^{\psi}_{s,t} - \theta_{s,t} |\,  \lesssim_C \,  &\, \left((1+\lambda)   \norm{ \psi }_{\mcC^\gamma_{T}} +(1+\norm{ \msA w }_{\mcC^\gamma_{T}}^{1-\chi}) \norm{ \psi}_{L^{\infty}_T}^{\chi} \right)|t-s|^{\iota'}\\
		& \, + \llbracket A^1\rrbracket_{\mcC^{\beta-\eps}_{T}} \,  \llbracket \theta^{\psi} - \theta \rrbracket_{\mcC^{\beta-\eps}_{T}} \absv{t-s}^{2(\beta-\eps)} \\
		&\, + \Big(\llbracket A^1\rrbracket_{\mcC^{\beta-\eps}_{T}} +\lambda \Big) \| \theta^\psi - \theta \|_{L^{\infty}_{T}} \absv{t-s}^{\beta-\eps}.
	\end{aligned}
\end{equation}
Appealing to Lemma~\ref{lem:averaged_field_tightness} for finiteness of Gaussian moments of $\llbracket A^1\rrbracket_{\mcC^{\beta-\eps}_{T}}$ (by the last statement of Lemma~\ref{lem:averaged_field_tightness} the estimate Gaussian estimate holds for time intervals longer than one) and Corollary~\ref{cor:theta_global_closed_bound} to verify \eqref{eq:averaged_field_tightness_assumption}, we apply Corollary~\ref{cor:gaussian_gronwall} in the same manner as in the proof of Lemma~\ref{lem:local_stability} with
\begin{equation}\label{eq:constants_for_gronwall}
	\begin{aligned}
		& C_1 \coloneqq \llbracket A^1\rrbracket_{\mcC^{\beta-\eps}_{T}}, \quad C_2 \coloneqq \llbracket A^1\rrbracket_{\mcC^{\beta-\eps}_{T}} + \lambda ,\\
		&C_3 \coloneqq (1+\lambda)   \norm{ \psi }_{\mcC^\gamma_{T}} +(1+\norm{ \msA w }_{\mcC^\gamma_{T}}^{1-\chi}) \norm{ \psi}_{L^{\infty}_T}^{\chi}, \\ 
		& \alpha_1 = 2\alpha_2 =  2( \beta-\eps),\quad \alpha_3 = \iota' >\beta -\eps, \quad \eta \coloneqq \beta - \eps.
	\end{aligned}
\end{equation}
 Therefore there exists a $C\coloneqq C(\alpha,H,d,\lambda,\|g\|_{\mcC^{\alpha}_x},m,|x|, T)>0$ such that
\begin{equation*}
	\left\| \sup_{t\in [0,T]} |\theta^\psi_{t} - \theta_t| + \left\|\theta^\psi - \theta \right \|_{\mcC^{\beta-\eps}_{[0,T]}}\right\|_{L^m_\Omega} \lesssim_C   \norm{ \psi }_{\mcC^\gamma_{T}} +\Big(1+\norm{ \msA w }_{\mcC^\gamma_{T}}^{1-\chi}\Big) \norm{ \psi}_{L^{\infty}_T}^{\chi}.
\end{equation*}
The conclusion in the case of disintegrated moments follows immediately, in particular since $X^\psi-X = \theta^\psi -\theta +\psi$. The case of full moments is almost identical, only appealing to the appropriate statement of Lemma~\ref{lem:lnd_to_heat_kernel} and the appropriate conclusion of Corollary~\ref{cor:theta_global_closed_bound}.
\end{proof}
\subsection{Solution Concept for Singular SDE}
When $\alpha<0$, and $g$ is allowed to be a genuine distribution, we cannot naively evaluate the right hand side of \eqref{eq:psi_sde} and so we require a specified notion of solutions. The following is similar to the definitions given in \cite[Sec.~4]{galeati_le_mayorcas_24_quantitative}.
\begin{defn}\label{def:singular_sde}
	Let $\alpha<0$, $g\in \cC^\alpha(\mbR^d;\mbR^d)$, $u:\mbR^d \to \mbR^d$ be a Lipschitz continuous map and $\psi \in C_T$. Then, we say that a tuple $(\Omega,\mcF,\FF,\PP;X_0, W^H, X^\psi)$ is a global weak solution to \eqref{eq:psi_sde} if the following hold:
	\begin{enumerate}[label=\roman*)]
		\item $(\Omega,\mcF,\FF,\PP)$ is a filtered probability space, $W^H$ is an $\RR^{d}$-valued fBm adapted to $\mbF$ and $X_0$ is an $\cF_0$-measurable $\mbR^d$ valued random variable;
		\item $X^\psi$ is $\FF$-adapted;
		\item there exists a sequence $\{g^n\}_{n\geq 1}\subset C^\infty_b(\mbR^d;\mbR^d)$ such that $g^n\to g$ in $\cC^{\alpha}(\mbR^d;\mbR^d)$, and 
		\begin{equation}\label{eq:singular_sde_limiting}
			X^\psi_{\,\cdot\,} -X_0 - W_{\,\cdot\,}  -\psi_{\,\cdot\,}- \int_0^{\,\cdot\,} u(X^\psi_s)\dd s = \lim_{n\to \infty} \int_{0}^{_{\,\cdot\,} } g^n(X^\psi_s)\dd s,
		\end{equation}
		where the limit is taken in probability with respect to $\PP$, locally uniformly on $\mbR_+$.
	\end{enumerate}
\end{defn}
Note that for any solution $X^\psi$ in the sense of Definition~\ref{def:singular_sde} where the limit on the right hand side of \eqref{eq:singular_sde_limiting} defines a bona fide continuous path (which we denote by an abuse of notation as $t\mapsto \int_0^t g(X^\psi_s)\dd s$), the process
\begin{equation}\label{eq:singular_sde_limiting}
	t\mapsto X^\psi_{t} -X_0  - \int_0^{t} u(X^\psi_s)\dd s - \int_0^t g(X^\psi_s)\dd s - \psi_t
\end{equation}
is distributed under $\mbP$ as an $\mbR^d$ valued fBm on $[0,+\infty)$.  A restricted class of solutions for which we can show that this limit is well defined and we have uniqueness of solutions are the $(\delta,m)$-solutions for parameters $(\delta,m)$ suitably chosen, see Proposition~\ref{prop:sde_well_posed}. 
\begin{defn}\label{def:beta_m_sol}
	In the setting of Definition~\ref{def:singular_sde}, given a pair $\delta \in (0,1]$, $m\in [1,\infty]$ and an interval $\mcI\subseteq \mbR_+$, a weak solution $X^\psi$ to \eqref{eq:psi_sde} on $\mcI$ is said to be $(\delta,m)$-regular if $X^\psi-X_0-W-\psi \in \mcC^{\delta}_{\mcI}L^m_\Omega$.
\end{defn}
\begin{rem}\label{rem:beta_m_sol_embeddings}
	Note that by Jensen's inequality, if a solution $X$ is $(\delta,m)$-regular, then it is also $(\delta,\tilde m)$-regular for any $\tilde m\in [1,m]$ and similarly, due to the embeddings of H\"older spaces, it must be $(\tilde\delta,m)$-regular for any $\tilde\delta<\delta$.
\end{rem}
Proving that there exists a unique strong solution that is $(\delta,m)$-regular requires tools developed in two following section. We give the proof in Proposition \ref{prop:sde_well_posed}.
\begin{prop}\label{prop:sde_well_posed}
    Let $(\Omega,\mcF,\mbF,\mbP)$ be a filtered probability space with $W^H$ an $\mbF$ adapted fBm under $\mbP$, $(\alpha,H) \in (-\infty,0)\times (0,\nicefrac{1}{2})$ satisfy \eqref{eq:stability_singular_parameters}, $g\in \mcC^\alpha(\mbR^d;\mbR^d)$, $u:\mbR^d\to \mbR^d$ satisfy \eqref{eq:confining_assumption}, $\psi\in \mcC^{\gamma}_{T;\lambda}$ for some $T>0$ and $x\in \mbR^d$. Then, there exists a pathwise unique solution $X^{\psi}$ to \eqref{eq:singular_sde} on $[0,T]$ which is $(1+\alpha H,m)$-regular for any $m\in [1,+\infty)$ and $\mbF$-adapted. If $\psi \in \mcC^\gamma_{\mbR_+;\lambda}$ then the same holds on all of $\mbR_+$.
%
\end{prop}
\begin{proof}
Let us first work on the interval $[0,1]$ and set $\psi =0$. Given $g\in \mcC^\alpha(\mbR^d;\mbR^d)$ (recall \eqref{eq:defn_cC_alpha}) we take a sequence $\{g_n\}_{n\geq 1}\subset C^\infty_b(\mbR^d;\mbR^d)$ such that $\|g-g_n\|_{\mcC^{\alpha}_x} \to 0$. To each $n\geq 1$ there exists a unique strong solution $X^n$ to \eqref{eq:singular_sde} with $g$ replaced by $g_n$. Applying Corollary~\ref{cor:theta_global_closed_bound} we immediately see that $\theta^n \coloneqq X^n- W^H$ satisfies the assumptions of Lemma~\ref{lem:local_stability} uniformly in $n\geq 1$ so that for any $1\leq n< m <+\infty$ we have
\begin{equation}
    \Big\|\sup_{t\in [0,1]}|X^m_{t}-X^n_{t}|\Big\|_{L^m_\Omega} \lesssim \|g^m-g^n\|_{\mcC^{\alpha}_x}.
\end{equation}
Since $\{g^n\}_{n\geq 1}$ is by definition a Cauchy sequence in $\mcC^{\alpha}(\mbR^d;\mbR^d)$ it follows that there exists an $X\in L^m(\Omega;L^\infty([0,1];\mbR^d))$ such that $X^n \to X$ in the relevant topology. By construction we also have $\|\theta^n-\theta\|_{L^m_\Omega C_{[0,1]}}\to 0$ where $\theta\coloneqq X-W^H$ and passing to the limit in \eqref{eq:theta_global_closed_bound} we see that $\theta$ also satisfies the same bound and since each $X^n$ is $\mbF^H$ adapted the same holds for $X$ and $\theta$. Then we take Lemma \ref{lem:integral_estimate}, with $\varphi = \theta = X - \wh$, and $g-g^n$ instead of $g$. We then have the bound 
\[ \norm{ \int_0^{\cdot} \left( g - g^n \right)\left( \wh_r + \theta_r \right) \dd r }_{\mcC^{1+\alpha H}_{[0,1]}L^m_{\Omega}} \ls \norm{ g - g^n }_{\mcC^{\alpha}_x} \left( \sqrt{m } + \norm{ \theta }_{\mcC^{1+\alpha H}L^m_{\Omega}} \right). \]
We then apply Corollary~\ref{cor:theta_global_closed_bound} to estimate $\norm{ \theta }_{\mcC^{1+\alpha H}L^m_{\Omega}}$ above and as a result we obtain
\begin{equation*}
	 \norm{ \int_0^{\,\cdot\,}(g-g^n)(X_r)\dd r  }_{\mcC^{1+\alpha H}_{[0,1]}L^m_\Omega}  \leq  C(\|g-g^n\|_{\mcC^{\alpha}_x})\Big(1+\|x\|_{L^{2m}_\omega}  +\sqrt{m}\, \Big),
\end{equation*}
 where $\lim_{r\to 0}C(r) =0$, in particular the process
 \begin{equation*}
 	t\mapsto \int_0^t g(X_r)\dd r,
 \end{equation*}
is well defined. We stress that the steps above work with no modification if $\psi \neq 0$ as we can take $\theta^{\psi} = X - \wh - \psi$. From here one may proceed as in the conclusion of the proof of \cite[Lem.~4.6]{galeati_le_mayorcas_24_quantitative} to conclude that $X$ is a $(1+\alpha H,m)$ solution in the sense of Definition~\ref{def:beta_m_sol}, only replacing $\mcC^{\kappa-\var}_{[0,T]}$ therein with $\mcC^{1+\alpha H}_{[0,1]}$  (recall \eqref{eq:uniform_holder_moment}) and \cite[Lem.~2.13]{galeati_le_mayorcas_24_quantitative} therein with our Corollary~\ref{cor:theta_global_closed_bound}. Equally, one may follow the remaining steps of the same proof, replacing \cite[Cor.~3.6]{galeati_le_mayorcas_24_quantitative} with our Lemma~\ref{lem:local_stability}, to conclude that the following limit
\begin{equation*}
	\theta - \int_0^t u(X_r)\dd r = \lim_{n\to \infty} \int_0^t \tilde{g}^n(X_r)\dd r,
\end{equation*}
holds in the $\mcC^{1+\alpha H}_{[0,1]}L^m_\Omega$ topology where $\{\tilde{g}^n\}_{n\geq 1}$ is any sequence of smooth bounded functions which converge to $g$ in $\mcC^{\alpha}(\mbR^d;\mbR^d)$. Finally, to argue that $X$ is the unique $(1+\alpha H,2)$ solution one may follow the proof of \cite[Lem.~4.8]{galeati_le_mayorcas_24_quantitative} replacing  \cite[Prop.~2.17]{galeati_le_mayorcas_24_quantitative}  with our Lemma~\ref{lem:local_stability}. Here we conclude the proof of well-posedness on the interval $[0,1]$. To extend to $[0,T]$ for $T>1$ we recognize that the result of Lemma~\ref{lem:local_stability} holds from any starting time $t\in [0,T]$ and so the above argument may be repeated, restarting at $t=1,\,2,\, \ldots, \floor{T}$ using a priori Gaussian tightness (Theorem~\ref{th:tightness_main}) to ensure that the intervals of existence do not shrink at future times. The final claim follows in the same manner, restarting at all $t\in \mbN$.
%
%
\end{proof}
With Proposition~\ref{prop:sde_well_posed} in hand we are able to construct an SDS in the sense of Definition~\ref{def:sds} associated to the SDE 
\begin{equation}\label{eq:singular_sde}
	X_t = x_0 + \int_0^t g(X_s)\dd s + \int_0^t u(X_s)\dd s + W^H_t, \qquad t \geq 0.
\end{equation}
For every $T>0$ we define the solution map
\begin{equation}
	[0,T] \times \mbR^d  \ni (t,x_0) \mapsto  \hat{\Phi}_T(x_0,W^H)(t) = X_t,
\end{equation}
where for $W^H$ a realization of the fBm, $\hat{\Phi}_T(x_0,W^H)$ is the unique solution obtained by Proposition~\ref{prop:sde_well_posed} to \eqref{eq:psi_sde} with $\psi=0$.

Recall the construction of the stationary noise process $(\mcW^-_{\gamma,\delta},\{\msP_t\}_{t\in \mbR_+}, \bP_H,\{\theta_t\}_{t\in \mbR_+})$ from Lemma~\ref{lem:stationary_noise} and for each $T>0$,  the continuous map $\msR_T$ defined by \eqref{eq:negative_to_positive_flip}. Then we set,
\begin{equation}\label{eq:stoch_sds_definition}
	\begin{aligned}
		\Lambda :\mbR_+ \times \mbR^d \times \supp(\bP_H) &\to \mbR^d\\
		(t,x,w) &\mapsto \hat{\Phi}_t(x,\msR_t w)(t),
	\end{aligned}
\end{equation}
where we observe that for $w\in \supp(\bP_H)$ the process $t\mapsto \msR_T w(t)$ is distributed according to an fBm restricted to $[0,T]$. The notation $\hat{\Phi}$ here is a cosmetic change to distinguish our solution map $\hat{\Phi}_t : \mbR^d \times \mcW^+_{\gamma,\delta}\to \mbR^d$ from the map $\Phi$ in Definition~\ref{def:sds} which takes arguments from the noise space itself.
\begin{thm}\label{th:sds_well_defined}
 Let $(\alpha,H) \in (-\infty,0)\times (0,\nicefrac{1}{2})$ satisfy Assumption~\ref{ass:stability_singularity},  $g\in \mcC^\alpha(\mbR^d;\mbR^d)$, $u:\mbR^d\to \mbR^d$ satisfy Assumption~\ref{ass:confining_assumption}. Then $\Lambda$ as given by \eqref{eq:stoch_sds_definition} defines a stochastic SDS in the sense of Definition~\ref{def:sds} over the stationary noise process $(\mcW^-_{\gamma,\delta}, \{\msP_t\}_{t\geq 0}, \bP^-_H, \{\theta_t\}_{t\geq 0})$. Moreover, for any generalised initial data $\mu$ (see \cite[Def.~2.3]{hairer_ohashi_07_ergodic}) the process generated by $\Lambda$ is a  weak solution to the SDE \eqref{eq:singular_sde}, adapted to the naturally enhanced filtration $\tilde{\mbF}^H$. 
\end{thm}
\begin{proof}
Firstly, for any initial condition $\mu = \delta_x \otimes \bP_H$, $T>0$ and $w\in \supp(\bP_H)$, by Proposition~\ref{prop:sde_well_posed} there exists a pathwise unique solution $\hat{\Phi}_T(x,\msR_t w)$ adapted to the natural filtration generated by the process $[0,T]\ni t\mapsto \msR_t w(t)$. Properties \ref{it:sds_path_reg} and \ref{it:sds_cocycle} of Definition~\ref{def:sds} follow directly from Proposition~\ref{prop:sde_well_posed} and a direct computation. The extension to any generalised initial condition is direct along with adaptation to the enhanced natural filtration via disintegration of measures. Property \ref{it:sds_continuous}, continuous dependence at every $(x,w)\in \mbR^d \times \supp(\bP^-_H)$ follows from Proposition~\ref{prop:perturbation_continuity_slow}.

\end{proof}
\section{Uniqueness of the Invariant Measure}\label{sec:uniqueness_inv_measure}
So far we have established a stochastic dynamical system associated to the SDE
\begin{equation}\label{eq:strong_feller_sde}
     X_t = x + \int_0^t (g(X_s)+u(X_s))\dd s + W^H_t,
\end{equation}
for
\begin{equation*}
    g \in \mcC^{\alpha}(\mbR^d;\mbR^d) \quad \text{with}\quad \alpha >1-\frac{1}{2H},
\end{equation*}
and $u$ satisfying Assumption~\ref{ass:confining_assumption}, as well as existence of ergodic, invariant measures in the sense of \cite{hairer_ohashi_07_ergodic}. In this section we show uniqueness of the ergodic invariant measures. Before starting the proof, let us introduce some notions and notations.

For each $T > 0$ we recall the solution map
\begin{equation*}
\begin{aligned}
    \hat{\Phi}_T :\mbR^d \times \mcW^+_{T;\gamma,\delta} \,&\to\, \mcW^+_{T;\gamma,\delta},
\end{aligned}
\end{equation*}
given in Theorem~\ref{th:sds_well_defined} and such that
\begin{equation*}
    \hat{\Phi}_T(x,\msR_T w )(t) = x + \int_0^t \left(g(\hat{\Phi}_T(x,\msR_T w)(s)) + u(\Phi_T(x,\msR_T w)(s)) \right) \dd s + (\msR_T w)_t,
\end{equation*}
with $\msR$ the map defined by \eqref{eq:negative_to_positive_flip}. The fact that the image of $\hat{\Phi}_T$ is contained in $\mcW^+_{T;\gamma,\delta}$ is a consequence of the existence of invariant measures and some regularity results on the solution map, see the comment just before \cite[Lem.~5.1]{hairer_ohashi_07_ergodic} along with our tightness result Theorem~\ref{th:tightness_main}.

Similarly for a given measurable map $\varphi : C([1, +\infty), \RR^d) \to \RR$ we define the functional $\bar{\mcQ} \varphi : \RR^d \times \mcW^{-}_{(\gamma,\delta)} \to \RR$ by setting
\begin{equation}\label{eq:semi_group}
    (\bar{\mcQ} \varphi)(x,w) := \int_{C([1,\infty),\RR^d)} \varphi(z) R^*_1 \bar{\mcQ} \delta_{(x,w)} (\dd z).
\end{equation}
The main result of this section is the following theorem.
\begin{thm}\label{th:unique_ergodicity}
    Let $(\alpha,H)\in (-\infty,0)\times (0,\nicefrac{1}{2})$, satisfy Assumption~\ref{ass:stability_singularity}, $g\in \mcC^{\alpha}(\mbR^d;\mbR^d)$ and $u:\mbR^d\to \mbR^d$ satisfy Assumption~\ref{ass:confining_assumption}. Then, there exists a unique ergodic invariant measure associated to the semi-group $\bar{\mcQ}$.
\end{thm}
\noindent Appealing to Theorem \ref{th:doob_khasminskii}, uniqueness of the ergodic invariant measure follows if we show that $\hat{\Phi}_T$ and the functional $\bar{Q}$ satisfy 
\begin{enumerate}
\item stochastic strong Feller (Definition \ref{def:strong_feller}),
\item topological irreducibility (Definition \ref{def:irreducible}),
\item quasi-Markovianity (\cite[Def.~3.7]{hairer_ohashi_07_ergodic}).
\end{enumerate}
We prove strong-Feller in Section \ref{subsec:strong_feller}, irreducibility in Sections~\ref{subsec:irreduc}  and quasi-Markovianity in Section~\ref{subsec:quasimarkov}. Existence of ergodic invariant measures in the sense of Definition~\ref{def:invariant_ergodic_measures} follows from Theorem~\ref{th:tightness_main} and the arguments of \cite[Lem.~2.20]{hairer_ohashi_07_ergodic}.
\subsection{The Strong--Feller Property}\label{subsec:strong_feller}%
We obtain the strong Feller property in the sense of Definition~\ref{def:strong_feller}, for the semi-group \eqref{eq:semi_group}. In order to prove that the resulting modulus of continuity satisfies our relaxed continuity requirement (compare with \cite[Def.~3.3]{hairer_ohashi_07_ergodic}) we first require some preparatory results on the Jacobian associated to \eqref{eq:strong_feller_sde}.

We retain the assumptions on $(\beta,\eps)$ from \eqref{eq:defn_beta_eps} and a fixed tuple $(\Omega,\mcF,\mbF^B,\mbP,B)$ as described at the start of  Section~\ref{sec:sds_construction}.

Fixing $T\in (0,+\infty)$ and formally differentiating \eqref{eq:strong_feller_sde} in the initial condition $x\in \mbR^d$ we define the matrix valued process
\begin{equation}\label{eq:jacobian_definition}
  t\mapsto   \mfJ_t \coloneqq D_x \hat{\Phi}_T(x,\msR_T w)(t) \in \mbR^{d\times d},
\end{equation}
and define its inverse $\mfJ^{-1}$ to be such that $\mfJ\mfJ^{-1} = \mfJ^{-1}\mfJ \equiv \mbI_{d}$. For $\psi \in C_T\mbR^d$ let us also introduce the notation
\begin{equation*}
    \mfJ^\psi \coloneqq D_x \hat{\Phi}_T(x,(\msR_T w)+\psi), \quad \mfJ^{\psi;-1} \mfJ^{\psi} = \mfJ^\psi \mfJ^{\psi;-1} = \mbI_d.
\end{equation*}
so that in particular $\mfJ = \mfJ^0$. We then have the following composite result, which is a variant of \cite[Lem.~6.3]{galeati_gerencser_22_subcritical} allowing for the unbounded drift $u$ as well as the singular $g$.
\begin{lem}\label{lem:jacboian_bounds}
   Let $T>1$, $(\alpha,H)\in (-\infty,0)\times (0,\nicefrac{1}{2})$ satisfy \eqref{eq:stability_singular_parameters}, $g \in C_b^{\infty}(\mbR^d;\mbR^d)$  and $u:\mbR^d\to \mbR^d$ satisfy Assumption~\ref{ass:confining_assumption}, $x\in \mbR^d$, $\psi \in \mcC^{\gamma}_{T}$ and $X^\psi$ be the associated strong solution to \eqref{eq:psi_sde}. Then, for any $t\in [0,T]$ it holds that
   \begin{align}
    \mfJ^\psi_t = &\, \mbI_{d} + \int_0^t \left(\nabla g(X^\psi_s) + \nabla u(X^\psi_s) \right) \mfJ^\psi_s \dd s, \label{eq:jacobian}\\ 
    \mfJ^{\psi;-1}_t = &\, \mbI_d - \int_0^t \mfJ^{\psi;-1}_s \left(\nabla g(X^\psi_s) + \nabla u(X^\psi_s) \right) \dd s. \label{eq:jacobian_inverse}
\end{align}
 In addition, for any $m\in [1,\infty)$, $(\beta,\eps)$ satisfying \eqref{eq:defn_beta_eps} and $w \in \mcW^-_{\gamma,\delta}$, there exists a constant \\ $C\coloneqq C(\alpha,H,d,\lambda,\|g\|_{\mcC^{\alpha}_x},T,\eps,m, |x|, \norm{ \psi}_{\mcC^\gamma_{T;\lambda}}, \norm{ w }_{\mcC^\gamma_{T;\lambda}}) > 0$ such that for all 
    \begin{equation}\label{eq:jacobian_estimates}
        \sup_{t\in [0,T]} \left( \Big\|\|\mfJ^
        \psi\|_{\mcC^{\beta-\eps}_{[t,t+1\wedge(T-t)]}}\Big\|_{\llmw}  +  \Big\|\, \| \mfJ^{\psi;-1} \|_{\mcC^{\beta-\eps}_{[t,t+1\wedge (T-t)]}} \, \Big\|_{\llmw} \right) \leq C
    \end{equation}
    In particular the identities \eqref{eq:jacobian} and \eqref{eq:jacobian_inverse} remain valid for $g\in \mcC^{\alpha}(\mbR^d;\mbR^d)$ and $X^\psi$ the associated strong solution. 
\end{lem}
\begin{proof}
   For $g\in \mcC^{\infty}_b(\mbR^d;\mbR^d)$ the identities \eqref{eq:jacobian} and \eqref{eq:jacobian_inverse} are satisfied by classical arguments and the a priori bound of Theorem~\ref{th:tightness_main}, see for example the proof of \cite[Lem.~5.1]{hairer_ohashi_07_ergodic}.

\noindent In order to prove \eqref{eq:jacobian_estimates}, let $t\in [0,T-1]$ and we see that for $(r,s)\in [t,t+1]^2_{\leq}$ we have
\begin{equation*}
    \mfJ^\psi_{r,s} = \int_r^s \left(\nabla g(X^\psi_a) + \nabla u(X^\psi_a) \right) \mfJ^\psi_a \dd a.
\end{equation*}
Defining the process $\mbR_+ \ni  t \mapsto A^\psi_t \coloneqq \int_0^t \nabla g(X^\psi_r)\dd r$ a combination of Lemma~\ref{lem:averaged_field_tightness} and Corollary~\ref{cor:theta_global_closed_bound} shows that $\mbP$-a.s. we have that $\llbracket A^\psi \rrbracket_{\mcC^{\beta-\eps}_{[t,t+1]}} <+\infty$, so that we may interpret the integral between $\nabla g(X^\psi)$ and $\mfJ^\psi$ as a Young integral. By standard estimates on Young integrals (e.g \cite[Eq.~(4.3)]{friz_hairer_20_book}) $\mbP$-a.s. it holds that
\begin{equation*}
    \left|\int_r^s \nabla g(X^\psi_a)  \mfJ^\psi_a \dd a\,  \right| \leq \|\mfJ^\psi\|_{L^\infty_{[t,t+1]}} \llbracket A^\psi \rrbracket_{\mcC^{\beta-\eps}_{[t,t+1]}} |s-r|^{\beta-\eps} + \|\mfJ^\psi\|_{\mcC^{\beta-\eps}_{[t,t+1]}}\llbracket A^\psi \rrbracket_{\mcC^{\beta-\eps}_{[t,t+1]}} |s-r|^{2(\beta-\eps)}.
\end{equation*}
On the other hand, due to Assumption~\ref{ass:confining_assumption}, we directly have
    \[ \left| \int_r^s \nabla u(X_a) \mfJ^\psi_a \dd a \, \right| \leq \lambda \int_s^r \absv{ X^\psi_a \mfJ^\psi_a } \dd a \leq \lambda \|X^\psi\|_{L^\infty_{[t,t+1]}} \left( \big\|\mfJ^\psi\big\|_{L^\infty_{[t,t+1]}} \absv{s-r} + \big\|\mfJ^\psi \big\|_{\mcC^{\beta-\eps}_{[t,t+1]}} \absv{s-r}^{1+\beta-\eps} \right). \] 
So that we conclude by an application of Lemma~\ref{lem:averaged_field_tightness} and Corollary~\ref{cor:theta_global_closed_bound} to verify Gaussianity of $\big\llbracket A^\psi \big\rrbracket_{\mcC^{\beta-\eps}_{[t,t+1]}}$, with respect to $\mbE_w$, uniformly over $t\in [0,T-1]$ and then Corollary~\ref{cor:gaussian_gronwall} with
    \[ C_1 = C_2 = \lambda \|X^\psi\|_{L^\infty_{[t,t+1]}} + \big\llbracket A^\psi \big\rrbracket_{\mcC^{\beta-\eps}_{[t,t+1]}}, \quad  C_3 = 0, \quad \alpha_1 = 2\alpha_2 = 2(\beta-\eps)\quad \text{and}\quad   \eta = \beta-\eps.
    \]
The proof for $\mfJ^{\psi;-1}$ is essentially identical. Finally, the extension to $g\in \mcC^{\alpha}(\mbR^d;\mbR^d)$ follows since all constants are uniform in $\|g\|_{\mcC^{\alpha}_x}$ and we have stability of solutions with respect to $g$, recalling Lemma~\ref{lem:local_stability}.
\end{proof}
The following lemma establishes continuity of the gradient averaged field along solutions with H\"older continuous perturbations. Let us recall our notation \eqref{eq:uniform_holder_moment} meaning that the H\"older semi-norm below is only ever computer on intervals of length $1$ or less inside $[0,T]$.

\begin{lem}\label{lem:nablag_with_perturb}
 Let $T>0$, $(\alpha,H)\in (-\infty,0)\times (0,\nicefrac{1}{2})$ satisfy \eqref{eq:stability_singular_parameters}, $g \in C_b^{\infty}(\mbR^d;\mbR^d)$  and $u:\mbR^d\to \mbR^d$ satisfy Assumption~\ref{ass:confining_assumption}, $x\in \mbR^d$, $\psi \in \mcC^{\gamma}_{T}$ and $X^\psi$ be the associated strong solution to \eqref{eq:psi_sde}. Then, there exists a $\chi_1\in (0,1)$ such that for all $\chi \in (0, \chi_1)$,  $\beta' \coloneqq \beta'(\chi) \in \left( \inv{2}, \beta \right)$ for $\beta$ as in \eqref{eq:defn_beta_eps} there exists a constant $C \coloneqq C(\alpha,H,d,\lambda,\eps,m,\|g\|_{\mcC^{\alpha}_x},|x|, \norm{ \psi}_{\mcC^\gamma_{T}}, \norm{ \msA w }_{\mcC^\gamma_{T}}, T)>0$ such that
\begin{equation*}
    \left\llbracket \int_0^{\,\cdot\,}  \left(\nabla g ( X^{\psi}_r )  - \nabla g ( X_r ) \right) \dd r\,  \right\rrbracket_{\llmw\mcC^{\beta'}_T} \leq C \left( \norm{ \psi }_{L^{\infty}_T}^{\chi^2} \vee     \norm{\psi }^\chi_{\mcC^\gamma_{T}} \right).
\end{equation*} 
\end{lem}
\begin{proof}
We will follow similar steps as in the proof of  Proposition~\ref{prop:perturbation_continuity_slow}. For $(s,t)\in [0,T]^2_\leq$, let us set
\begin{equation*}
    \bar{A}_{s,t} := \int_s^t \E^B_{w,s} \left[\nabla g(\wh_r + \theta^{\psi}_s + \psi_r)-\nabla g(\wh_r + \theta_s) \right]\, \dd r.
\end{equation*}
Applying Lemma~\ref{lem:lnd_to_heat_kernel}, \eqref{eq:heat_kernel_fbm} with $\eta = \alpha$ and $\kappa = 1+\chi$, for some $\chi>0$ which will be restricted later, and the triangle inequality, we directly see that 
\begin{equation*}\label{eq:gradient_bara_space_cont}
 \| \bar{A}_{s,t}\|_{\llmw} \ls_{\chi,\alpha,H,d} \norm{ g }_{\mcC^{\alpha}_x} \absv{t-s}^{1+H(\alpha-1-\chi)} \left( \norm{ \psi }_{L^{\infty}_T}^{\chi} + \big\| \big|\theta^{\psi} - \theta\big|^{\chi} \big\|_{L^{\infty}_T\llmw} \right).
\end{equation*}
where $\theta^\psi = X^\psi - W^H - \psi$ (and $\theta$ similarly). By Proposition \ref{prop:perturbation_continuity_slow} and Jensen's inequality there exists a $C\coloneqq C(\alpha,H,d,m,\lambda, \|g\|_{\mcC^{\alpha}_x},\chi,|x|,\norm{\msA w }_{\mcC^\gamma_{T}} )>0$ such that
\begin{equation*}\label{eq:theta_psi_diff}
\big\| \big|\theta^{\psi} - \theta \big|^{\chi} \big\|_{L^{\infty}_T\llmw} \leq \big\| \theta^{\psi} - \theta \big\|_{L^{(1\lor\chi)m}_{\Omega;w}L^{\infty}_T} \leq C \left(   \norm{ \psi }_{\mcC^\gamma_{T}} + \norm{ \psi}_{L^{\infty}_T}^{\chi} \right)
\end{equation*}
Hence, applying the triangle inequality to obtain the second estimate below, we have
\begin{align}
        \| \bar{A}_{s,t}\|_{\llmw} \leq &\, C \norm{ g }_{\mcC^{\alpha}_x} \absv{t-s}^{1+H(\alpha-1-\chi)} \left( \norm{ \psi }_{L^{\infty}_T}^{\chi} +     \norm{ \psi }_{\mcC^\gamma_{T}} \right), \label{eq:averaged_field_stable_2_diff}\\
        \|\delta  \bar{A}_{s,u,t}\|_{\llmw} \leq &\, C \norm{ g }_{\mcC^{\alpha}_x} \absv{t-s}^{1+H(\alpha-1-\chi)} \left( \norm{ \psi }_{L^{\infty}_T}^{\chi} +     \norm{\psi }_{\mcC^\gamma_{T}} \right). \label{eq:averaged_field_stable_3_diff_1}
\end{align}
Inequality \eqref{eq:averaged_field_stable_2_diff} verifies \eqref{eq:base_SSL_bnd_1} of Lemma~\ref{lem:SSL}. We will verify \eqref{eq:base_SSL_bnd_2} by an interpolation argument as employed in the proof of Proposition~\ref{prop:perturbation_continuity_slow}. To do so, we expand $\mbE^B_{w,s}\delta \bar{A}_{s,u,t}$ (by analogy see for example \eqref{eq:three_point_bara}) and again apply Lemma~\ref{lem:lnd_to_heat_kernel} and \eqref{eq:heat_kernel_fbm} with $\eta = \alpha-1$ and $\kappa =1$ to see that
\begin{equation}\label{eq:averaged_field_stable_3_diff_2}
 \norm{ \EE^B_{w,s} \delta  \bar{A}_{s,u,t} }_{\llmw} \ls \absv{t-s}^{2\beta} \norm{ g }_{\mcC^{\alpha}_x} \left( \big\| \theta^{\psi} \big\|_{\mcC^{\beta+H}_T\llmw} + \big\| \theta \big\|_{\mcC^{\beta+H}_T\llmw} \right).
\end{equation}
Therefore, interpolating between \eqref{eq:averaged_field_stable_3_diff_1} and \eqref{eq:averaged_field_stable_3_diff_2}, for any $\chi \in (0,1)$ and applying Corollary~\ref{cor:theta_global_closed_bound} (noting that $\beta+H = 1+\alpha H$) it holds that
\begin{align}
    \norm{ \EE_s^B \delta \bar{A}_{s,u,t} }_{\llmw} \lesssim_C  &\, \absv{t-s}^{(2 - \chi)\beta - \chi^2}  \norm{ g }_{\mcC^{\alpha}_x} \left( \norm{ \psi }_{L^{\infty}_T}^{\chi} +     \norm{\psi }_{\mcC^\gamma_{T}} \right)^\chi\left( \big\| \theta^{\psi} \big\|_{\mcC^{\beta+H}_T \llmw} + \norm{ \theta }_{\mcC^{\beta+H}_T \llmw} \right)^{1-\chi} \notag \\
   \lesssim_C &\, \absv{t-s}^{(2 - \chi)\beta - \chi^2}\norm{ g }_{\mcC^{\alpha}_x} \left( \norm{ \psi }_{L^{\infty}_T}^{\chi^2} +     \norm{\psi }^\chi_{\mcC^\gamma_{T}} \right) \left( 1 + \norm{ \psi }_{\mcC^\gamma_{T}} + \norm{\mcA  w }_{\mcC^\gamma_{T}} \right)^{1-\chi},\label{eq:averaged_field_stable_3_diff_2}
\end{align}
where the constant $C>0$ has the same dependencies as above. Taking $\chi_1$ to be the largest positive number satisfying the inclusion
\begin{equation*}
   q_\beta (\chi)\coloneqq   (2 - \chi)\beta - \chi^2 >1.
\end{equation*}
To see that such a $\chi_1$ exists, observe that $q_{\nicefrac{1}{2}}(0) = 1$ while $q_1(0) = 2$ (recall $\beta \in (\nicefrac{1}{2},1)$), $\beta \mapsto q_\beta(0)$ is increasing and $\mbR_+ \ni \chi \mapsto q_\beta (\chi)$ is decreasing for all $\beta>0$. Therefore, choosing $\chi \in (0,\chi_1)$ the estimate \eqref{eq:averaged_field_stable_3_diff_2} verifies \eqref{eq:base_SSL_bnd_2} and so we may apply Lemma~\ref{lem:SSL} (checking that $\int_0^{\,\cdot\,}  \big(\nabla g ( X^{\psi}_r )  - \nabla g ( X_r ) \big) \dd r$ is the unique process obtained from $\bar{A}$ in the usual manner). We conclude by an application of the Kolmogorov continuity criterion (e.g. \cite[Thm.~A.11]{friz_victoir_10_multidimensional}), where a dependence of the constant $C$ on $T$ arises from patching together the estimate obtained above on intervals of length one to cover all of $[0,T]$.
 \end{proof}
We are now able to show continuity of the Jacobian with respect to deterministic perturbations.
\begin{prop}\label{prop:jacobian_continuity}
Let $T>0$, $(\alpha,H)\in (-\infty,0)\times (0,\nicefrac{1}{2})$, satisfy Assumption~\ref{ass:stability_singularity}, $g\in \mcC^{\alpha}(\mbR^d;\mbR^d)$ and $u:\mbR^d\to \mbR^d$ satisfy Assumption~\ref{ass:confining_assumption}, $\psi \in C^\gamma_{T}\mbR^d$ and $X^\psi$ be the associated strong solution to \eqref{eq:psi_sde} obtained in Proposition~\ref{prop:sde_well_posed}. Then, there exists a $\chi_1 \coloneqq \chi_1(\alpha,H)\in (0,1)$ such that for any $m\in [1,\infty)$, $\chi \in (0,\chi_1)$,  $\beta'\coloneqq \beta'(\chi) \in [\beta-\eps,\beta)$  and $w\in \mcW^{-}_{\gamma,\delta}$ there exists a constant \\ $C\coloneqq C(\alpha,H,d,\lambda, \|g\|_{\mcC^\alpha_x},m,\beta',T, \norm{ \msA w }_{\mcC^\gamma_{T;\lambda}}, \norm{ \psi }_{\mcC^\gamma_{T;\lambda}})>0$ for which
\begin{equation*}
	\big\| \mfJ^{\psi} - \mfJ \big\|_{\llmw\mcC^{\beta'}_T} \leq C \left( \norm{ \psi }_{L^{\infty}_T}^{\chi^2} \vee     \norm{\psi }^\chi_{\mcC^\gamma_{T}} \right).
	\end{equation*}
\end{prop}
\begin{proof}
From \eqref{eq:jacobian}, for $(s,t)\in [0,T]^2_\leq$, the triangle inequality, Assumption ~\ref{ass:confining_assumption} we directly have the $\mbP$-a.s. estimate
	\begin{align*}
	\left| \mfJ^{\psi}_{s,t} - \mfJ_{s,t} \right| \lesssim  \,&\,  \left| \int_s^t \left(\nabla g(X^\psi_r) -\nabla g(X_r)  \right)\mfJ^{\psi}_r\dd r  \right| + \left| \int_s^t \nabla g(X_r)\big(\mfJ^{\psi}_r - \mfJ_r \big)\dd r  \right| \\
	& \, + \lambda |t-s| \left(\|X^\psi-X\|_{L^\infty_{T}}\|\mfJ^\psi\|_{L^\infty_{T}} +   \|X\|_{L^\infty_{T}} \|\mfJ^\psi-\mfJ\|_{L^\infty_{[s,t]}}   \right).
	\end{align*}
	Treating the first two integrals as Young integrals we have (e.g. by \cite[Eq.~(4.3)]{friz_hairer_20_book}) for any $|t-s|\leq 1$
	\begin{align*}
		\left| \int_s^t \left(\nabla g(X^\psi_r) -\nabla g(X_r)  \right)\mfJ^{\psi}_r\dd r  \right|  \leq & \, \left\llbracket  A^1 \right\rrbracket_{\mcC^{\beta'}_{[s,t]}}\left( \|\mfJ^\psi\|_{L^\infty_{[s,t]}} |t-s|^{\beta'} +\llbracket\mfJ^\psi \rrbracket_{\mcC^{\beta-\eps}_{[s,t]}}|t-s|^{2\beta'}  \right)\\
		\lesssim & \, \left\llbracket  A^1 \right\rrbracket_{\mcC^{\beta'}_{T}} \|\mfJ^\psi\|_{\mcC^{\beta-\eps}_{T}}  |t-s|^{\beta'},
	\end{align*}
	with $A^1_t \coloneqq  \int_0^t \left(\nabla g(X^\psi_r) -\nabla g(X_r)  \right) \dd r$ and
	\begin{align*}
		\left| \int_s^t \nabla g(X_r)\big(\mfJ^{\psi}_r - \mfJ_r \big)\dd r  \right| \leq\,&\,   \left\llbracket  A^2 \right\rrbracket_{\mcC^{\beta-\eps}_{[s,t]}}\left( \|\mfJ^\psi-\mfJ\|_{L^\infty_{[s,t]}} |t-s|^{\beta-\eps} +\llbracket\mfJ^\psi-\mfJ \rrbracket_{\mcC^{\beta-\eps}_{[s,t]}}|t-s|^{2(\beta-\eps)}  \right)\\
		\leq \, &\,  \left\llbracket  A^2 \right\rrbracket_{\mcC^{\beta-\eps}_{T}}\left( \|\mfJ^\psi-\mfJ\|_{L^\infty_{[s,t]}} |t-s|^{\beta'} +\llbracket\mfJ^\psi-\mfJ \rrbracket_{\mcC^{\beta-\eps}_{[s,t]}}|t-s|^{2\beta'}  \right)
	\end{align*}
	with $A^2_t \coloneqq \int_0^t \nabla g(X_r) \dd r$. Hence, appealing to Theorem~\ref{th:tightness_main}, Lemma~\ref{lem:nablag_with_perturb}, Lemma~\ref{lem:averaged_field_tightness}   and Lemma~\ref{lem:jacboian_bounds} we check that the conditions of Corollary~\ref{cor:gaussian_gronwall} are satisfied with 
	\begin{align*}
		& C_1 = \left\llbracket \int_0^{\cdot} (\nabla g )(X_r) \dd r \right\rrbracket_{\mcC^{\beta-\eps}_{T}}, \quad C_2 = \left\llbracket \int_0^{\cdot} (\nabla g)(X_r) \dd r \right\rrbracket_{\mcC^{\beta-\eps}_{T}} + \lambda \|X\|_{L^\infty_T}, \\
		&C_3 =  \bigg\llbracket \int_0^{\,\cdot\,} \left( \nabla g(X^\psi_r) -\nabla g(X_r)  \right) \dd r  \bigg\rrbracket_{\mcC^{\beta'}_T}   \|\mfJ^\psi\|_{\mcC^{\beta-\eps}_T} + \lambda \|X^\psi-X\|_{L^\infty_{T}} \|\mfJ^\psi\|_{L^\infty_{T} }, \\
		 &\alpha_1 = 2\beta', \quad  \alpha_2 = \alpha_3 = \beta',\quad  \eta = \beta - \eps.
	\end{align*}
\noindent Therefore, we conclude by an application of H\"older's inequality, Proposition~\ref{prop:perturbation_continuity_slow} and Lemma~\ref{lem:nablag_with_perturb}. 
\end{proof}
We are now in a position to prove the strong Feller property. For $T>1$, let us introduce the suggestive notation $\mcW^+_{[1,T];\gamma,\delta}$ for the space of paths in $\mcW^+_{\gamma,\delta}$ restricted to $[1,T]$. Then for $\varphi_T : \mcW^+_{[1,T];\gamma,\delta} \to \RR$, let recall $\bar{\mcQ} \varphi_T$ as defined in \eqref{eq:semi_group}.
\begin{thm}\label{th:strong_feller_proof}
Let $(\alpha,H)\in (-\infty,0)\times (0,\nicefrac{1}{2})$, satisfy Assumption~\ref{ass:stability_singularity}, $g\in \mcC^{\alpha}(\mbR^d;\mbR^d)$ and $u:\mbR^d\to \mbR^d$ satisfy Assumption~\ref{ass:confining_assumption}.  Then, there exists a measurable map  $\ell: \mcW^-_{\gamma,\delta} \to \RR$, continuous at any $w\in \supp(\bP_H)$ and such that for $\bP^-_H$-a.e. $w \in \mcW^-_{\gamma,\delta}$, any $x,\,y\in \mbR^d$, $T>1$ and $\varphi_T \in C_b(\mcW^+_{[1,T];\gamma,\delta};\mbR)$ such that $\|\varphi_T\|_{C_b} \leq 1$, then it holds that
\begin{equation}
\left| \bar{\mcQ} \varphi_T(x,w) - \bar{\mcQ} \varphi_T(y, w) \right| \leq C(w) \absv{x-y}. 
	\end{equation}
\end{thm}
\begin{proof}
We first fix $T>1$ and see that for any $\varphi_T \in W^{1,\infty}(\mcW^+_{[1,T];\gamma,\delta};\mbR)$ and $x,\,y\in \mbR^d$, it holds that 
\begin{equation}\label{eq:semi_group_difference_1}
     \bar{\mcQ} \varphi_T(x,w) -   \bar{\mcQ} \varphi_T(y,w) =\mbE_w \left[\,\int_0^1 \langle (D \varphi_T) (\hat{\Phi}_T(z_\theta,\msR_T w)),D_x \hat{\Phi}_T(z_\theta,\msR_Tw)\xi  \rangle_{L^2_{[1,T]}}  \dd \theta   \right] ,
\end{equation}
where we recall that $\mbE_w$ is a expectation over $\mcW^+_{\gamma,\delta}$ with respect to the measure $\msH(w,\, \cdot\, )$ and we set
\begin{equation*}
    z_\theta \coloneqq \theta x +(1-\theta)y\quad \text{and}\quad \xi \coloneqq x-y.
\end{equation*}
Similarly to the definition of Jacobian \eqref{eq:jacobian}, we define the Fr\'echet derivative of $\hat{\Phi}_T$ in the noise component in any direction $v\in \mfH^+_{H}$ as
\begin{equation*}
    t\mapsto \mfK^v_t \coloneqq D_{w} 
    \hat{\Phi}_T(x,\msR_T w+\psi)[v](t) \in \mbR^d,
\end{equation*}
which for $g\in C^\infty_b(\mbR^d;\mbR^d)$ and $X$ the associated strong solution to \eqref{eq:psi_sde}, is directly seen to satisfy, 
\begin{equation}\label{eq:noise_derivative}
    \mfK^v_t = \int_0^t \left(\nabla g(X_s) + \nabla u(X_s)\right) \mfK^v_s \dd s + v_t, \quad \text{for all }\, t\in [0,T].
\end{equation}
By a similar argument as the conclusion of the proof of Lemma~\ref{lem:jacboian_bounds} this remains valid for $g\in \mcC^{\alpha}(\mbR^d;\mbR^d)$ and $X$ the associated strong solution obtained in Proposition~\ref{prop:sde_well_posed}. Note that for $v\in \mfH^+_{[1,T];H}$ (see \eqref{eq:cameron_martin_positive} with the natural notational modification here) we have the identity 
\begin{equation}\label{eq:K_to_malliavin}
\langle \mfK^v,f\rangle_{L^2_T} = \langle \mfD \hat{\Phi}_T(x,\msR_T w)[v],f\rangle_{L^2_T}, \quad \text{for all } f \in L^2([0,T];\mbR^d),
\end{equation}
where $\mfD$ denotes the Mallaivin derivative (see for example \cite[Thm.~6.8]{galeati_gerencser_22_subcritical}). Appealing to \eqref{eq:jacobian}-\eqref{eq:jacobian_inverse} and \eqref{eq:noise_derivative} one observes the identity
\begin{equation}\label{eq:malliavin_by_jacobian}
    \mfK^v_t =\mfJ_t   \int_0^t \mfJ^{-1}_s \dd v_s. 
\end{equation}
This can be obtained for example by expanding $\frac{\dd}{\dd t}\left( \mfJ_t   \int_0^t \mfJ^{-1}_s \dd v_s \right)$. Arguing as in the proof of Lemma~\ref{lem:jacboian_bounds}, working from \eqref{eq:noise_derivative}, for $v\in \mfH^+_{T;H}\cap \mcC^{\eta}_T\mbR^d$ for some $\eta >\nicefrac{1}{2}$, there exists a $C \coloneqq C(\alpha,H,d,\lambda,\|g\|_{\mcC^{\alpha}_x},T,\eps,m)>0$ such that\footnote{If one works with spaces of $p$-variation, instead of H\"older regularity it is possible to obtain an analogous statement for generic $v\in \mfH^+_{H;T}$, see \cite[Thm.~6.8]{galeati_gerencser_22_subcritical}. Since, in the end, we will only need this bound for a $v\in W^{1,\infty}_T\mbR^d$ this restriction is of no consequence here.}
\begin{equation}\label{eq:k_v_estimates}
        \sup_{t\in [1,T]}  \left\|\|\mfK^
        v\|_{\mcC^{\eta \wedge (\beta-\eps)}_{[t,t+1\wedge (T-t)]}}\right\|_{L^m_\Omega} \leq C\|v\|_{\mcC^{\eta}_T\mbR^d}.
    \end{equation}
Hence, by stochastic Fubini we can re-write \eqref{eq:semi_group_difference_1}, for any $x,\,y\in \mbR^d$ and $w\in \supp(\bP^-_H)$ as
\begin{equation}\label{eq:semi_group_difference_2}
	\bar{\mcQ} \varphi_T(x,w) -   \bar{\mcQ} \varphi_T(y,w) =\,\int_0^1 \mbE_w \left[\langle (D \varphi_T) (\hat{\Phi}_T(z_\theta,\msR_T w)),\mfJ^\psi_T \xi  \rangle_{L^2_T}  \right] \dd \theta   .
\end{equation}
We fix some $h:[0,1]\to \mbR$, smooth, such that $\supp(h)\subseteq (0,1)$  and $\int_0^1 h_s \dd s =1$ and define the process 
\begin{equation*}
	t\mapsto v^\xi_t \coloneqq \int_0^t h_s \mfJ_s \xi \dd s.
\end{equation*}
It follows from Lemma~\ref{lem:jacboian_bounds} that $v^\xi\in W^{1,\infty}_T\mbR^d$ and by  \eqref{eq:malliavin_by_jacobian} one observes that
\begin{equation*}
	\mfK^{v^\xi}_t = \mfJ_t \int_0^t  \mfJ^{-1}_s (h_s \mfJ_s \xi) \dd s = \mfJ_t \xi \int_0^t h_s \dd s.
\end{equation*}
So that in particular, for $t\geq 1$, due to our assumptions on $h$,
\begin{equation}\label{eq:k_v_J_equal}
  \mfK^{v^\xi}_t = \mfJ_t \xi = (D_x \hat{\Phi}_T(z_\theta,\msR_T w)\xi)(t).
\end{equation}
Now, since we also assumed that the test function $\varphi_T$ only acts on the path of the noise restricted to $[1,T]$, applying \eqref{eq:k_v_J_equal} in the first line and then the chain rule for Malliavin derivatives, we have
\begin{align*}
	\bar{\mcQ} \varphi_T(x,w) -   \bar{\mcQ} \varphi_T(y,w) =&\, \int_0^1 \mbE_w \left[\langle (D (\varphi_T)(\hat{\Phi}_T(z_\theta,\msR_T w),\mfK^{v^\xi} \rangle_{L^2_T}  \right] \dd \theta\\
 =&\, \int_0^1 \mbE_w \left[\langle (D (\varphi_T)(\hat{\Phi}_T(z_\theta,\msR_T w),\mfD \hat{\Phi}_T(z_\theta,\,\cdot\,)[v^\xi] \rangle_{L^2_T}  \right] \dd \theta \\
 =&\, \int_0^1 \mbE_w \left[\langle (D (\varphi_T)(\hat{\Phi}_T(z_\theta,\msR_T w) \mfD \hat{\Phi}_T(z_\theta,\,\cdot\,),v^\xi \rangle_{L^2_T}  \right] \dd \theta\\
    = &\, \int_0^1 \mbE_w \left[\langle \mfD (\varphi_T\circ \hat{\Phi}_T)(z_\theta,\,\cdot\,),v^\xi \rangle_{L^2_T}  \right] \dd \theta.
\end{align*}
We now apply \cite[Thm.~5.2]{hairer_ohashi_07_ergodic} to obtain the estimate
\begin{align}
	\left| \bar{\mcQ} \varphi_T(x,w) -   \bar{\mcQ} \varphi_T(y,w) \right| \leq &\, \sup_{\theta\in [0,1]} \left| \mbE_w \left[\langle \mfD (\varphi_T\circ \hat{\Phi}_T)(z_\theta,\,\cdot\,),v^\xi \rangle_{L^2_T}  \right] \right| \notag \\
	&\leq \left(\mbE_w \left[\|v^\xi\|^2_{\mfH^+_{H;[1,T]}}\right]\right)^{\nicefrac{1}{2}} \underbrace{\sup_{t\in [1,T]}\big\|(\varphi_T\circ \hat{\Phi}_T)(z_\theta,\,\cdot\,)(t)\big\|_{L^2_\omega}}_{\leq 1}. \label{eq:final_strong_feller}
\end{align}
Appealing to the fact that $\frac{\dd}{\dd t} v^\xi|_{[1,+\infty)} =0$ (due to the choice of $h$), we see that for any $\gamma \in (0,\beta-\eps)$,
\begin{equation*}
	\|v^\xi\|^2_{\mfH^+_H} = \int_0^\infty  |\msD^{\nicefrac{1}{2}+H} v_t |^2 \dd t  \lesssim \bigg\| \frac{\dd}{\dd t}v^\xi \bigg\|^2_{\mcC^\gamma_{[0,1]} \mbR^d } = \big\| h \mfJ^\psi  \xi\big\|^2_{\mcC^{\gamma}_{[0,1]}\mbR^d} \lesssim_h \|\mfJ\|^2_{\mcC^{\beta-\eps}_{[0,1]}} |x-y|^2,
\end{equation*}
where the final bound is due to Lemma~\ref{lem:jacboian_bounds}, the properties of $h$ and in particular is independent of $T$. This allows us to define $\ell$ by setting
\[ \supp(\bP^-_H)\ni w\mapsto \ell(w) \coloneqq \left( \EE_w \norm{ \mfJ }^2_{\mcC^{\beta-\eps}_{[0,1]}} \right)^{1/2}.\]
 To show continuity of $\ell$ at $\bP_H^-$-a.e. $w\in \mcW^{-}_{\gamma,\delta}$ let us take $\mfv \in \mcW^-_{\gamma,\delta}$ and recalling the definition of $\msH$ (see \eqref{eq:disintegration_measure}) we observe that
 \begin{align*}
\EE_{w+\mfv} \big\| \mfJ \big\|^2_{\mcC^{\beta-\eps}_{[0,1]}} 
& = \int_{\mcW^+_{[0,1];\gamma,\delta}} \norm{ D_x \hat{\Phi}_1\left(x,z-z_0 \right)(t)}^2_{\mcC^{\beta-\eps}_{[0,1]}} \left( \left( \tau_{\msA\left(w + \mfv\right)} \circ \msD^{1/2-H} \right)^* \bW\right) (\dd z ) \\ 
& = \int_{\mcW^+_{[0,1];\gamma,\delta}} \norm{ D_x \hat{\Phi}_1\left(x,z-z_0 - (\msA \mfv)|_{[0,1]}\, \right)(t)}^2_{\mcC^{\beta-\eps}_{[0,1]}} \left( \left( \tau_{A w } \circ \msD^{1/2-H} \right)^* \bW\right) (\dd z)\\
& = \mbE_w \big\|\mfJ^{-\msA \mfv} \big\|^2_{\mcC^{\beta-\eps}_{[0,1]}}.
\end{align*}
 Hence, by the triangle inequality an application of \eqref{eq:jacobian_estimates} and Proposition \ref{prop:jacobian_continuity} (with $\beta'=\beta-\eps$) we have
\begin{align*}
    \left| \ell(w)^2 - \ell(w+\mfv)^2 \right|\leq \,  &\, \EE_w \left[\Big| \norm{ \mfJ }^2_{\mcC^{\beta-\eps}_{[0,1]}} -   \big\| \mfJ^{-\msA \mfv} \big\|^2_{\mcC^{\beta-\eps}_{[0,1]}}\Big|\right]\\
     \lesssim \, &\, \Big( \norm{ \mfJ }^2_{L^2_{\Omega;w}\mcC^{\beta-\eps}_{[0,1]}} + \big\| \mfJ^{-\msA \mfv} \big\|^2_{L^2_{\Omega;w}\mcC^{\beta-\eps}_{[0,1]}}\, \Big)  \big\| \mfJ - \mfJ^{-\msA \mfv} \big\|^2_{L^2_{\Omega;w}\mcC^{\beta-\eps}_{[0,1]}} \\
\leq \, & \, C \left( \norm{ \msA \mfv }_{L^{\infty}_T}^{\chi^2} \vee     \norm{\msA \mfv }^\chi_{\mcC^\gamma_{T}} \right)^2.
\end{align*}
for some $C\coloneqq C(\alpha,H,d,\lambda,\|g\|,m,\beta,\, \eps,\|w\|_{\mcC^\gamma_{T}},\|\psi\|_{\mcC^\gamma_T})>0$. Since $\msA$ is a bounded linear operator from $\mcW^-_{\gamma,\delta}\to \mcW^+_{\gamma,\delta}$ we conclude that $\ell(w+\mfv)\to \ell(w)$ as $\|\mfv\|_{\gamma,\delta}\to 0$.
\end{proof}
\subsection{Irreducibility}\label{subsec:irreduc}
Let us first introduce a variant of Girsanov theorem which applies to the Mandelbrot van--Ness representation of the fBm. Although the result is surely known, we could not find a self-contained reference amenable to our setting. See for comparison  \cite{NO02, galeati_gubinelli_20_Noiseless, CG16} where a similar result is shown for the Volterra representation of fBm. We begin with a short exposition of fractional integral and differential operators with negative exponents necessary for our analysis. 

Recall that  \eqref{eq:fractional_integral} and \eqref{eq:fractional_derivative} define $\msI^\alpha$ and $\msD^\alpha$ for $\alpha>0$. Given $f\in C^\infty([0,T];\mbR^d)\cap L^1([0,T];\mbR^d)$ and $\alpha \in (-1,0)$ we set
\begin{equation}\label{eq:frac_ops_negative_def}
    \msI^\alpha f = \partial_t \msI^{1+\alpha} f \quad \text{and}\quad \msD^{\alpha} f = \msD^{1+\alpha} \int_0^{\,\cdot\,} f_r \dd r.
\end{equation}
The fact that the right hand sides of both expressions in \eqref{eq:frac_ops_negative_def} are bona fide continuous maps follows from \cite[Thm.~4]{Pic11}. In addition the semi-group property remains, so that in particular (see \cite[Thm.~5]{Pic11}), for $\alpha \in (-1,1)$
\begin{equation}\label{eq:frac_dervi_inverses}
    \msI^\alpha \msD^\alpha f = \msD^\alpha \msI^\alpha f =f.
\end{equation}
Note that if $\alpha \in (-1,0)$, $f(0)=0$ and $f\in \mcC^{\alpha'}_T \mbR^d$ for $\alpha' > -\alpha$ one may write
\begin{equation}\label{eq:bar_msI_def}
\begin{aligned}
    (\msI^\alpha f)_t = (\partial_t \msI^{1+\alpha}f)_t = &\, \partial_t \left(\frac{1}{\Gamma(\alpha+2)} t^{1+\alpha} f(t) - \frac{1}{\Gamma(\alpha+1)} \int_0^t (t-s)^\alpha (f_t-f_s) \dd s\right) \\
    =&\, \frac{1}{\Gamma(\alpha)}\int_0^t (t-s)^{\alpha -1} f_s \dd s\\
    =&\, \frac{1}{\Gamma(\alpha+1)} \int_0^t (t-s)^{\alpha} \dd f_s\\
    \eqqcolon &\, (\bar{\msI}^\alpha f)_t.
    \end{aligned}
\end{equation}
We then have the following version of Girsanov's theorem.

\begin{lem}\label{lem:girsanov_applies}
Let $(\Omega, \mcF, \mbP, \mbF^B)$ be a filtered probability space with a two sided $\mbF^B$-adapted $\mbP$-Brownian motion $(B_t)_{t \in \RR}$, $g\in \mcC^\infty_b(\mbR^d;\mbR^d)$, $x\in \mbR^d$ and $(X_t)_{t \geq 0}$ be any solution to \eqref{eq:singular_sde} in the sense of Definition \ref{def:singular_sde}. Then, there exists a $T >0$, a probability measure $\mbQ$, equivalent to $\mbP$, and an $\mbF^B$-adapted $\mbQ$-Brownian motion $\tilde{B}$ such that under $\mbQ$ the process $(X_t)_{t\in[0,T]}$ is an fBm with representation \eqref{eq:fbm_mvn} relative to $\tilde{B}$ and for any $m\geq 1$ 
\begin{equation*}
    \EE \bigg[ \left(\frac{ \dd\mbP }{ \dd\mbQ}\right)^m \bigg]  \vee \EE \left[ \left(\frac{ \dd \mbQ }{ \dd \mbP}\right)^{m} \right] <+\infty. 
\end{equation*}
\end{lem}
\begin{proof}
For brevity let us set $b(x) = g(x) + u(x)$. The proof broadly follows  \cite[Sec.~3.2 \& 3.3]{NO02}, but with necessary modifications to treat the Mandelbrot van--Ness representation of  the fBm, \eqref{eq:fbm_mvn}. For $t\geq 0$ and some $\tilde{c}_H>0$, let us set
\begin{equation*}
    h_t \coloneqq \tilde{c}_H \msD^{H-\nicefrac{1}{2}} \left( \int_0^{\cdot} b\left( X_s \right) \dd s  \right) 
\end{equation*}
and define the process
\begin{equation}\label{eq:btilde_plush_eqb}
\tilde{B}_t \coloneqq \begin{cases}   B_t + h_t & t \geq 0, \\
B_t & t < 0. \end{cases}
\end{equation}
Note that since we assume $g\in C^\infty_b(\mbR^d;\mbR^d)$ it holds that $\{t\mapsto \int_0^t b(X_s)\dd s \} \in W^{1,\infty}_T\mbR^d$ so that the application of $\msD^{H-\nicefrac{1}{2}}$ is justified and results in a Lipschitz continuous process, see \eqref{eq:frac_ops_negative_def}. By definition, therefore, up to choosing $\tilde{c}_H$ appropriately, we have $    (\msI^{H-\nicefrac{1}{2}} h)_t = \int_0^t b(X_r)\dd r$. A combination of \eqref{eq:bar_msI_def}, \eqref{eq:frac_dervi_inverses} and \eqref{eq:btilde_plush_eqb} and H\"older regularity of the Brownian motion gives
\begin{equation*}
    (\bar{\msI}^{H-1/2} B) = (\msI^{H-1/2} B) = \msI^{H-1/2} \tilde{B} - \msI^{H-1/2} h = \msI^{H-1/2} \tilde{B} - \int_0^{\cdot} b(X_r) \dd r,
\end{equation*}
so that we have
\begin{equation}\label{eq:x_def_tilde}
X_t - \int_{-\infty}^0 G(t,r) \dd B_r = (\bar{\msI}^{H-1/2} B)_t + (\bar{\msI}^{H-1/2} h)_t = (\bar{\msI}^{H-1/2} \tilde{B})_t.
\end{equation}
We therefore check that Novikov's condition applies to $h$ so that we can find a new measure $\mbQ$ under which $\tilde{B}$ is a Brownian motion.

\noindent By the semi-group property of fractional derivatives (see \cite[Thm.~5]{Pic11})
\[ \partial_t \msD^{\alpha} = \msD^{2+\alpha} \msI^1 = \msD^{1+\alpha}. \]
Appealing to the Poincare inequality on $[0,T]$, the semi-group property of fractional and integer derivatives and characterization of their images (see \cite[Thm.~5]{Pic11}) and the triangle inequality, for some $\eta \in(H+\nicefrac{1}{2},1)$ and subsequently $\delta >0$ such that $\eta +\delta <1$,
\begin{equation}\label{eq:fractional_embedding_holder}
\begin{aligned}
    \norm{ h }_{W^{1,2}_{T}\mbR^d} \lesssim \left\| \msD^{H+\nicefrac{1}{2}} \int_0^{\,\cdot\,} b(X_s)\dd s \right\|_{L^2_{T}\mbR^d} \ls &\, \left\llbracket \int_0^{\,\cdot\,} b(X_r)\dd r \right\rrbracket_{\mcC^{\eta}_{T}\mbR^d} \\
    %
    %
    \leq & \,T^\delta \left\llbracket \int_0^{\cdot} g(X_r) \dd r +\int_0^{\cdot} u(X_r) \dd r\right\rrbracket_{\mcC^{\eta+\delta}_{[0,1]}\mbR^d} .
    \end{aligned}
\end{equation}  
Note that the spaces $\mathbb{H}^{\beta,0}([0,T];
\mbR^d)$ used in \cite[Thm.~5]{Pic11} are nothing but $\{f \in \mcC^{\beta}([0,T];\mbR^d)\,:\, f_0=0\}$. Choosing $\delta<0$ such that $ 1+\alpha H -2\delta \in (H+\nicefrac{1}{2},1)$ (this is always possible due to \eqref{eq:stability_singular_parameters}) we can apply Corollary~\ref{cor:theta_global_closed_bound} along with \cite[Lem.~A.2]{galeati_gerencser_22_subcritical} to see that there exists a $\kappa_0\coloneqq \kappa_0(d,\alpha,H,\lambda,\|g\|_{\mcC^{\alpha}_x},\delta) >0$ such that for all $\kappa \in (0,\kappa_0)$ such that 
\begin{align*}
    \mbE\left[\exp\left(\kappa \left\llbracket \int_0^{\cdot} g(X_r) \dd r +\int_0^{\cdot} u(X_r) \dd r\right\rrbracket_{\mcC^{\eta+\delta}_{[0,1]}\mbR^d}\right)\right] <\infty.
\end{align*}
Hence, in combination with \eqref{eq:fractional_embedding_holder} we see that choosing $T\in (0,1]$ sufficiently small as a function of $\kappa_0$ we ensure that
\begin{equation}\label{eq:novikov}
\EE \left[\exp\left( \inv{2} \norm{ h }_{W^{1,2}_T \mbR^d}^2 \right) \right] < \infty,
\end{equation}
so that Novikov's condition applies and the first claim is shown. To get the bound on Radon--Nikodym derivative we proceed in an analogous way to the proof of \cite[Thm.~14]{galeati_gubinelli_20_Noiseless} and exploit the estimate obtained on $\norm{ h }_{W^{1,2}_T\mbR^d}$  above. 
\end{proof}
As a result we have the following proof of topological irreducibility for the SDS constructed in Theorem~\ref{th:sds_well_defined}.
\begin{cor}
There exists a time $t > 0$ such that the SDS $\Lambda$ constructed in Theorem~\ref{th:sds_well_defined}, induced by \eqref{eq:singular_sde} is \emph{topologically irreducible} at time $t$ in the sense of Definition~\ref{def:irreducible}. 
\end{cor}
\begin{proof}
We first recall the identity for any $U\in \mcB(\mbR^d)$ and $(x,w)\in \mbR^d \times \mcW^{-}_{\gamma,\delta}$
\begin{equation*}
    \mcQ_t(x,w;Y\times \mcW^-_{\gamma,\delta}) = \int_{\mcW^-_{\gamma,\delta}} \delta_{\Lambda_t(x,w')}(U)\msP_t(w,\dd w') =  \int_{\mcW^-_{\gamma,\delta}} \delta_{\Lambda_t(x,w')}(U)\bP^-_H(\dd w')  = \mbP(X_t^x \in U).
\end{equation*}
By Lemma~\ref{lem:girsanov_applies} there exists a $T>0$ and probability measure $\mbQ$ equivalent to $\mbP$ such that $X$ is distributed as an fBm and all positive and negative moments of the density $\nicefrac{\dd \mbP}{\dd \mbQ}$ are finite. Hence, for any $t<T$ we may write
    \begin{equation*}
    \mbP(X_t^x \in U) = \mbE[\indic_{\{X^x_t\in U\}}] = \mbE_{\mbQ}\left[\indic_{\{W^H_t\in U\}} \frac{\dd \mbP}{\dd \mbQ}\right]  >0.
\end{equation*}
\end{proof}
\subsection{Quasi-Markovianity}\label{subsec:quasimarkov}
Finally, we show that the SDS constructed in Theorem~\ref{th:sds_well_defined} is quasi-Markovian over the stationary noise process of Lemma~\ref{lem:stationary_noise}. Note that we make one small modification to \cite[Def.~3.7]{hairer_ohashi_07_ergodic}, only requiring the map $w\mapsto \mcP^{V,U}_s(w,\,\cdot\,)$ be defined on $\supp(\bP^-_H)$.

Similar to \cite[Cor.~5.10]{hairer_ohashi_07_ergodic} we have the following corollary to our Lemma~\ref{lem:L_2_image_deriv_A}. Given an SDS $\Lambda$ build on top of the stationary noise process, we recall the definition of the sets 
\begin{equation*}
    \mcN^t_{\mcW_{\gamma,\delta}} \coloneqq \bigg\{(w,\tilde{w})\in \Big(\mcW_{\gamma,\delta}\Big)^2 \,|\, R^*_t \bar{\mcQ} \delta_{x,w} \sim R^*_t \bar{\mcQ} \delta_{x,\tilde{w}}\, \text{for all }\, x\in \mcX\, \bigg\}\quad \text{for } t\geq 0,
\end{equation*}
where the equivalence relation is that of equivalence of measures.
\begin{cor}\label{cor:smooth_difference_N_set}
    The set of pairs $(w,\tilde{w})\in (\mcW^-_{\gamma,\delta})^2$ such that $\tilde{w}-w \in C^\infty_0(\mbR_-;\mbR^d)$ is contained in the set $\mcN^t_{\mcW_{\gamma,\delta}}$ for every $t>0$.
\end{cor}
\begin{proof}
    Unpacking the definitions, $R^*_t\bar{\mcQ}\delta_{x,w}$ denotes the law of the solution to
    \begin{equation*}
        X_{\,\cdot\,} = x + \int_0^{\,\cdot\,} g(X_s)\dd s + \int_0^{\,\cdot\,} u(X_s)\dd s + (\msA w)_{\,\cdot\,}|_{[0,+\infty)},
    \end{equation*}
    as a path restricted to $[t,+\infty)$, with $\tilde{X}$ defined analogously, using the fact that addition of a smooth path does not affect local nondeterminism (see Lemma \ref{lem:lnd}). It follows from a trivial extension of Lemma~\ref{lem:L_2_image_deriv_A} that if the difference $\tilde{w}-w \in C^\infty_0(\mbR_-;\mbR^d)$ then $\msD^{H+\nicefrac{1}{2}}\msA(\tilde{w}-w)\in L^2(\mbR_+;\mbR^d)$, in other words the driving noises of $X$ and $\tilde{X}$ differ only by an element of the relevant Cameron--Martin space $\mfH^-_H$ (see \eqref{eq:cameron_martin_negative}) and so the are equivalent in law, from which the conclusion follows.
\end{proof}
\begin{prop}\label{prop:our_quasimarkov}
There exists a time $t > 0$ such that the SDS $\Lambda$ constructed in Theorem~\ref{th:sds_well_defined}, induced by \eqref{eq:singular_sde} is quasi-Markovian (in the sense of \cite[Def.~3.7]{hairer_ohashi_07_ergodic} appropriately modified as above) over the stationary noise process given by  Lemma~\ref{lem:stationary_noise}.
\end{prop}

\begin{proof}[Proof of Proposition \ref{prop:our_quasimarkov}]
It is clear from the proof of  \cite[Prop.~5.11]{hairer_ohashi_07_ergodic} that quasi-Markovianity does not depend on the construction of the SDS only how the noise enters the system. As such one may repeat exactly the same proof here, except for the last line, where we refer to Corollary \ref{cor:smooth_difference_N_set} instead of Corollary 5.10 in \cite[Cor.~5.10]{hairer_ohashi_07_ergodic} and recall that our solution map is only well-defined on $\supp(\bP^-_H)$.
\end{proof}
\appendix
\section{A Gr\"onwall Type Lemma}\label{app:gronwall}
For completeness we state here a variant of Gr\"onwall's lemma. The proof is known (it was introduced in \cite[Lemma 2.11]{DGHT19}, where it is dubbed "rough Gronwall lemma"), we only outline the steps so that the paper is self-contained.
\begin{lem}\label{lem:my_gronwall}
 Let $T>0$, $E$ be a normed vector space, $\rho \in C_T E$ such that for all $S\in [0,T)$ and $(s,t) \in [S,T]^2_{\leq}$ there exists some $\alpha_1 > 1,\,  1> \alpha_2,\,  \alpha_3 \geq \eta > 1/2> 0$ and positive constants $C_1,\, C_2,\, C_3\, \geq 0$ such that 
\[ \norm{ \rho_t - \rho_s }_{E} \leq \norm{ \rho }_{\mcC^{\eta}_{[S,T]}E} C_1 \absv{t-s}^{\alpha_1} +  C_2 \norm{ \rho }_{L^{\infty}_{[S,T]}E} \absv{t-s}^{\alpha_2} + C_3 \absv{t-s}^{\alpha_3}. \]
Then, there exists a $\mu  > 0$ such that
\begin{equation}\label{eq:gronwall_deterministic_bound}
\begin{aligned}
 \sup_{u \in [S,T]} \norm{ \rho_u }_{E} + \norm{ \rho }_{\mcC^{\eta}_{[S,T]}E} \leq & \exp\left( \mu \left( 1 \lor  \absv{T - S} \right) \left( \left( 2C_1 \right)^{\inv{\alpha_1-\eta}} +  \left( 2C_2 \right)^{\inv{\alpha_2}} \right) \right) \\
 &\qquad \qquad \qquad \times \left( \norm{ \rho_{S} }_E + 2 C_3 \left( 1 \land \absv{T-S}^{\alpha_3}\right) \right).
 \end{aligned}
 \end{equation}
\end{lem}

\begin{proof}
For brevity we will work with $S = 0$, the general case follows by translation in the time domain. First we divide both sides by $\absv{t-s}^{\eta}$, which for the intervals of length $\Delta \leq 1$ such that
\[ C_1 \Delta^{\alpha_1 - \eta}  \vee 2 C_2 \Delta^{\alpha_2} \leq \nicefrac{1}{2} \]
allows us to obtain via the first condition
\begin{equation}\label{eq:rho_gronwall_holder_shorttime}
\norm{ \rho }_{\mcC^{\eta}_{T} E} \leq 2 C_2 \norm{ \rho }_{L^{\infty}_{[s,t]}E} \Delta^{\alpha_2-\eta} + 2 C_3 \Delta^{\alpha_3-\eta}.
\end{equation}
Fix $k \in \NN$, then it holds that
\[ \sup_{u \in [0,\Delta]} \norm{ \rho_{k\Delta+u} }_E \leq \norm{ \rho_{k\Delta}} + \Delta^{\alpha_2}  2 C_2 \norm{ \rho }_{L^{\infty}_{[s,t]}E} + 2 C_3 \Delta^{\alpha_3}, \]
which by the second condition on $\Delta$ gives:
\[  \sup_{u \in [0,\Delta]} \norm{ \rho_{k\Delta+u} }_E \leq 2 \norm{ \rho_{k\Delta}} + 4 C_3 \Delta^{\alpha_3}. \]
Let $N = \lfloor \nicefrac{T}{\Delta} \rfloor + 1$, then the claim on the supremum norm follows by iterating the procedure $N$ times, and bounding the second term above with $\Delta \leq 1 \land T$, so that we obtain for some numerical constant $\nu > 0$ the following inequality
\begin{equation}\label{eq:rho_sup_intermediary}
\sup_{u \in [0,T]} \norm{ \rho_u }_E \leq \exp\left( \nu T \left( \left( 2 C_1 \right)^{\inv{\alpha_1-\eta}} + \left( 2C_2 \right)^{\inv{\alpha_2}} \right) \right)  \left( \norm{ \rho_0 }_E + C_3 \left( 1 \land T \right)^{\alpha_3}  \right).
\end{equation}
The H\"older norm is bounded in short intervals with \eqref{eq:rho_gronwall_holder_shorttime} and then extended to $[0,T]$ with \cite[Ex.~4.24]{friz_hairer_20_book}, which amounts to multiplying the right hand side of \eqref{eq:rho_gronwall_holder_shorttime} with $\Delta^{\eta-1}$. Using the definition of $\Delta$ one bounds very crudely, for another numerical constant $\bar{\nu} > 0$
\begin{equation}\label{eq:delta_eta_mone_est}
2 \Delta^{\eta - 1} C_2 \leq \exp\left( \bar{\nu} (1 \lor T ) \left( 2C_1 \right)^{\inv{\alpha_1-\eta}} + \left( 2C_2 \right)^{\inv{\alpha_2}} \right),
\end{equation}
so that the bound on $\norm{ \rho }_{\mcC^{\eta}_TE}$ follows by combining \eqref{eq:delta_eta_mone_est}, \eqref{eq:rho_sup_intermediary} and \eqref{eq:rho_gronwall_holder_shorttime}, up to choosing another numerical constant $\mu > \bar{\nu} \lor \nu$ if necessary.
\end{proof}
\begin{cor}\label{cor:gaussian_gronwall}
 In the setting of Lemma~\ref{lem:my_gronwall}, let $1\leq m<m'<+\infty$, $\rho \in L^m_\Omega C_TE$ with $\rho_0 \in L^{m'}_\Omega E$, $C_1,\,C_2,\,C_3$ be non-negative random variables for which there exist $\kappa_1,\,\kappa_2 >0$ (depending only on $(S,T)$ through $|T-S|$) such that 
\begin{equation*}
    G_i \coloneqq \mbE\left[e^{\kappa_i C^2_i}\right] <+\infty\quad \text{for } \,\, i=1,\,2
\end{equation*}
    and $C_3 \in L^{m'}_\Omega E$. Then, there exists some $C\coloneqq C(m,m',T-S,\alpha_1,\,\alpha_2,\alpha_3,\eta,G_1,G_2,\kappa_1,\kappa_2)>0$ such that 
\begin{equation}\label{eq:random_gronwall_bound}
\norm{ \sup_{u \in [S,T]} \norm{ \rho_u }_{E} + \norm{ \rho }_{\mcC^{\eta}_{[S,T]}E}  }_{L^m_{\Omega}} \leq C \left(\norm{\rho_{S}}_{L^{m'}_\Omega} + \norm{C_3 }_{L^{m'}_{\omega}} \right).
\end{equation}
\end{cor}

\begin{proof}
   To prove the second claim, we apply $\norm{ \,\cdot\, }_{L^m_{\Omega}}$ to \eqref{eq:gronwall_deterministic_bound}, use H\"older's inequality and  then exploit the fact that powers in \eqref{eq:gronwall_deterministic_bound} are $\inv{\alpha_1-\eta} < 2, \inv{\alpha_2} < 2$, which gives us moments of all orders thanks to assumption on $G_i$. 
\end{proof}
\section{Integrals and Moments of Gaussian Processes}\label{sec:gaussian_bounds}
In this section we provide bounds for Gaussian processes that appear in our analysis.
\begin{lem}\label{lem:gaussian_ou_bound}
Let $(X_t)_{t \geq 0}$ be a centered Gaussian process such that for all $\beta,\,\beta' \in (0,\nicefrac{1}{2})$
\[\EE\left[ \absv{ X_t - X_s }^2\right] \leq \absv{t-s}^{2\beta} \lor \absv{t-s}^{2\beta'} \quad \text{for all }\,   s<t\]
and
\[\EE\left[ (X_t - X_v)(X_v - X_s)\right] \ls_{\beta} \left( \absv{t-v}^{\beta}\absv{v-s}^{\beta} \right) \lor \left( \absv{t-v}^{\beta'} \absv{v-s}^{\beta'} \right)\quad\text{for all }\, s<v<t. \]
Then for $\lambda > 0$ and $m \geq 1$ it holds that
\[ \sup_{t \geq 0} \norm{ \int_0^t e^{-\lambda(t-r)} \dd X_r }_{L^m_{\Omega}} \ls_{\lambda,\beta} \sqrt{m}. \]
\end{lem}
\begin{proof}
Since the proof is standard we only provide a sketch. Using the fact that $X$ is Gaussian it suffices to estimate the second moment and the rest follows by hypercontractivity. We first use integration by parts (as $e^{-\lambda(t-\cdot)}$ is smooth this is just Riemann-Stjeltes integration)
\[ \int_0^t e^{-\lambda(t-r)} \dd X_r = e^{-\lambda t} X_t + \int_0^t e^{-\lambda(t-r)} (X_t - X_r) \dd r. \]
We raise to the second power and observe that the only term that needs a bit more care is the following (which we obtain by Fubini)
\begin{equation}\label{eq:fubini}
\int_0^t \int_0^t e^{-\lambda(t-r)} e^{-\lambda(t-r')} (X_t- X_r) (X_t - X_r') \dd r' \dd r \end{equation}
so we use for $0\leq r'<r <t$ (treating the other regime in a symmetric way) that 
\begin{equation}\label{eq:product_x}
(X_t - X_r)(X_t - X_r') = (X_t - X_r)^2 + (X_t- X_r)(X_r- X_{r'}) 
\end{equation}
and then by applying expectation to \eqref{eq:fubini} the resulting first term is bounded by  
\[ \int_0^t \int_0^t e^{-\lambda(t-r)} e^{-\lambda(t-r')} \left( (t-r)^{2\beta} \lor (t-r)^{2\beta'} \right) \dd r'  \dd r <+\infty. \]
The second term in \eqref{eq:product_x} is similarly bounded with the second covariance assumption and the proof is finished by grouping terms together.
\end{proof}
\begin{lem}\label{lem:rl_nice_bd}
Let $B^H_t = \int_0^t (t-r)^{H-1/2} \dd B_r$ with $H < \nicefrac{1}{2}$ be a Riemann-Liouville Brownian motion. Then $B^H$ satisfies the assumptions of Lemma \ref{lem:gaussian_ou_bound} with $\beta' \in (0, H), \beta \in (H, \nicefrac{1}{2})$
\end{lem}
\begin{proof}
For  $0\leq s<t<+\infty$ we have
\[ \int_0^t (t-r)^{H-1/2} \dd B_r - \int_0^s (s-r)^{H-1/2} \dd B_r = \int_s^t (t-r)^{H-1/2} \dd B_r + \int_0^s \left( (t-r)^{H-1/2} - (s-r)^{H-1/2} \right) \dd B_r. \]
We take second moments - clearly the two terms are independent and the first term is bounded with $\absv{t-s}^{2H}$ up to a constant. We only have to bound the second term, so by It\^o isometry
\[ \int_0^s \left( (t-r)^{H-1/2} - (s-r)^{H-1/2} \right)^2 \dd r = \int_0^{s-1} \dd r + \int_{s-1}^s \dd r. \]
We then use a basic interpolation estimate for $x,\, y>0$ and $\beta \in (0,1],\,  \theta \in \RR$ to see that 
\begin{equation}\label{eq:real_interpolation}
\big|x^{\theta} - y^{\theta} \big| \leq \absv{ x-y}^{\beta} ( x \land y )^{\theta-\beta}
\end{equation}
The first integral then is suitably bounded using $\beta = H+\delta$ for some small $\delta > 0$, the second one with $\beta = H-\delta$.

For the second hypothesis, we perform similar decomposition as above for $v<s<t$, i.e.
\begin{align*}
 B^H_t - B^H_s = & \int_s^t (t-r)^{H-1/2} \dd B_r + \underbrace{ \int_v^s (t-r)^{H-1/2} - (s-r)^{H-1/2} \dd B_r }_{=:I_2} \\ &\;\;\;+ \underbrace{\int_0^v (t-r)^{H-1/2} - (s-r)^{H-1/2} \dd B_r}_{=:I_3} \\
 B^H_s - B^H_v = &\underbrace{ \int_v^s (s-r)^{H-1/2} \dd B_r }_{=:J_1} + \underbrace{ \int_0^v (s-r)^{H-1/2} - (v-r)^{H-1/2} \dd B_r }_{=:J_2} \\
 \end{align*}
We then compute expectation of the product of two lines above, which results in six terms. However, by independence of Brownian increments, we observe that only non-zero terms are $\mbE \left[ I_2 J_1 \right]$ and $\mbE \left[ I_3 J_2 \right]$. By It\^o isometry and monotonicity of $x \to x^{H-1/2}$, we have $\mbE \left[ I_2 J_1 \right] \leq 0$. Again by It\^o isometry, decomposing to $\int_0^{v-1} + \int_{v-1}^v$ and exploiting the interpolation estimate \eqref{eq:real_interpolation}:
\begin{align*}
 \mbE \left[ I_3 J_2 \right] = & \int_0^v \left( (t-r)^{H-1/2} - (s-r)^{H-1/2} \right) \left( (s-r)^{H-1/2} - (v-r)^{H-1/2} \right) \dd r \\ 
 \leq & \absv{t-s}^{\beta} \absv{v-s}^{\beta} \int_0^{v-1} (s-r)^{H-1/2-\beta} (v-r)^{H-1/2-\beta} \dd r \\ & \;\;\; + \absv{t-s}^{\beta'} \absv{v-s}^{\beta'} \int_{v-1}^v (s-r)^{H-1/2-\beta'} (v-r)^{H-1/2-\beta'} \dd r
\end{align*}
Since $(s-r)^{\eta} \leq (v-r)^{\eta}$ for all $\eta < 0$, the two integrals above are finite for $\beta > H > \beta'$ and the conclusion follows up to a multiplicative constant.
\end{proof}
\section{Stochastic Sewing Lemmas}\label{app:stoch_sewing_proofs}
We collect some stochastic sewing lemmas (SSL) which are used variously throughout the text. Some are by now standard, some only more recently introduced and some novel. Where a sufficiently concise or clear proof is not available in the literature we have given a sketch.

In this section we let $d,\, \ell,\,  \in \mbN\setminus \{0\}$ be fixed such that $1\leq d \leq \ell$ and $(\Omega,\mcF,\mbF,\mbP,B)$ be a tuple of a probability space $(\Omega,\mcF,\mbP)$ and a standard $\mbR^k$ valued standard Brownian motion $B$ with $\mbF = \{\mcF_t\}_{t\geq 0}$ its natural filtration. 

The following version of SSL with controls and conditional norms is a particular subcase of \cite[Thm.~3.1]{le_23_banach} for $\mathfrak{p}=2$ therein. 
\begin{lem}\label{lem:SSL}
		Let $m,\,n$ be real numbers such that $2\leq m\leq n \leq \infty$ with $m<\infty$ and $0\leq \cS<\cT<\infty$.
	Assume that $A: [\mcS,\mcT]^2_{\leq} \to L^m_\Omega \mbR^\ell$ is a continuous mapping from $[\cS,\cT]^2_{\leq}$ to $L^m_\Omega\RR^\ell$ such that $A_{s,t}$ is $\cF_t$-measurable for all $(s,t) \in [\cS,\cT]^2_{\leq}$.
	Suppose that there exist $\eps_1,\,\eps_2>0$ and $\Gamma_1,\,\Gamma_2>0$ such that for all $(s,u,t)\in [\cS,\cT]^3_{\leq}$ one has the bounds
	\begin{align}
		\|A_{s,t} \|_{L^m_\Omega} &\leq \Gamma_1 |t-s|^{\frac{1}{2} + \eps_1},\label{eq:base_SSL_cond_1} \\
		\|\EE_s \delta A_{s,u,t}\|_{L^m_\Omega}&\leq \Gamma_2 |t-s|^{1+\eps_2}. \label{eq:base_SSL_cond_2}
	\end{align}
        Then, there exists a unique adapted stochastic process $\{\cA_t: t\in [\cS,\cT]\}$, continuous as a mapping from $[\cS,\cT]$ to $L^m_\Omega$, such that $\cA_{\cS} =0$
	Moreover, there exists a constant $C\coloneqq C(\eps_1,\eps_2,m,n,\ell)>0$ such that
	\begin{align}
		\|\cA_t -\cA_s \|_{L^m_\Omega} &\leq C \Gamma_1 |t-s|^{\frac{1}{2}+\eps_1} + C \Gamma_2|t-s|^{1+\eps_2}, \label{eq:base_SSL_bnd_1}\\
		\| \cA_t-\cA_s -A_{s,t} \|_{L^m_\Omega} &\leq C \Gamma_2|t-s|^{1+\eps_2}. \label{eq:base_SSL_bnd_2}
	\end{align}
\end{lem}
A second variant we appeal to allows us to obtain Gaussian tail estimates on the resulting integral when we do not have bounded integrands. The idea is to study two parameter processes which decompose into a stochastic integral and a process satisfying the a necessary moment bound. One can then apply the Burkoholder--Davis--Gundy inequality with optimal growth of the constant to the stochastic integral component. The proof is similar to that of \cite[Lem.~A.1]{GM23}, which in turn was inspired by \cite[Thm.~4.7]{ABLM24}. First let us state BDG as a proposition. Let us also remark that another method of dealing with this problem has been introduced in \cite{BDG19}, where the authors also manage to rewrite sewing in terms of continuous martingale, but at the price of introducing additional conditions that the germ needs to verify.
\begin{prop}
Let $(M_t)_{t \geq 0}$ be a continuous martingale and let $\langle M \rangle$ its quadratic variation. Then for any $m \geq 2$ and $T > 0$ there holds
\begin{equation}\label{eq:bdg_optimal}
    \norm{ M_T }_{L^m_{\Omega}} \leq 2 \sqrt{m} \norm{ \langle M \rangle_T }_{L^m_{\Omega}}.
\end{equation} 
\end{prop}
\begin{proof}
See \cite[App.~A]{CK91} for a clean proof written in English. A reader capable of reading Russian may consult older reference \cite{novikov_73_moment}, where similar proof first appeared.
\end{proof}
\begin{lem}\label{lem:sewing_stochastic_integral}
	Let $0\leq \mcS <\mcT <+\infty$, $m \in [2,+\infty)$ and $A :[\mcS,\mcT]^2_{\leq } \to L^m_\Omega \mbR^\ell$ be a continuous two-parameter random process such that
	\begin{equation*}
		\delta A_{s,u,t} = \int_s^u h_z  \dd B_z + J_{s,t},\qquad \text{for all }\, (s,u,t)\in [\mcS,\mcT]^3_{\leq},
	\end{equation*}
	where $J:[\mcS,\mcT]^2_\leq \to L^m_\Omega\mbR^\ell$ is a continuous process and $h : [\mcS,\mcT]\to L^m_\Omega \mbR^{\ell \times d}$ is a continuous, $\mbF$-adapted, matrix valued process. Suppose that there exist $\eps_1,\,\eps_2>0$, a constant $ C_J\coloneqq C_J(\mcS,\mcT,m)>0$ and a family of non-negative random variables $\{C_{h,s}\}_{s\in [S,T]}\subset L^m(\Omega;\mbR)$ where each $C_{h,s}$ is $\mcF_s$-measurable, such that 
	\begin{align}
		\|J_{s,t}\|_{L^m_\Omega} \leq &\, C_J |t-s|^{1+\eps_1},\quad \text{for all } \, (s,t)\in [\mcS,\,\mcT]^2_{\leq}, \label{eq:stoch_integral_sewing_J_bound} \\
        \|h\|_{L^2_{[s,t]}}\leq &\, C_{h,s} \absv{t-s}^{\nicefrac{1}{2} +\eps_2}, \quad \mbP\text{-a.s.  for all }\, (s,u,t) \in [\mcS,\mcT]^3_{\leq}.\label{eq:stoch_integral_sewing_h_bound}
	\end{align}
	 Then, there exists a unique $\mbF$-adapted continuous stochastic process $\mcA :  [\mcS,\mcT]\to L^m_\Omega \mbR^\ell$ such that $\mcA_S =0$ and for all $(s,t)\in [\mcS,\mcT]^2_\leq$ and is such that for some $C_1\coloneqq C_1(\eps_1,\eps_2)>0$  
\begin{equation}\label{eq:sewing_stoch_integral_bnd_1}
		\|\mcA_t -\mcA_s - A_{s,t}\|_{L^m_\Omega}   \, \leq\,   C_1 \left( \sup_{u \in [s,t]}\norm{ C_{h,u} }_{L^m_{\Omega}}  \sqrt{m}   \absv{t-s}^{\nicefrac{1}{2}+\eps_2} + C_{J} \absv{t-s}^{1+\eps_1} \right). 
	\end{equation}
\end{lem}
\begin{proof}[Sketch proof of Lemma~\ref{lem:sewing_stochastic_integral}]
	The proof is similar to that of \cite[Lem.~A.1]{GM23},  so we will only sketch the ideas. In particular we don't present the proof of uniqueness of the limit, since it is by now standard, see for example the proof of \cite[Lem.~A.3]{GM23} for the main ideas as well as \cite[Thm.~3.1]{le_23_banach}. We construct the process $\mathcal{A}_t$ as a limit of approximations $A^{n}_{t} \coloneqq \sum_{i=1}^n A_{t_i,t_{i+1}}$ for $(t_i)_{i \in \{1,...,n\}}$ a particular choice of partition of $[0,t]$. The proof of uniqueness shows that the limit is independent of this choice.
	
	Given $(s,t)\in [\mcS,\mcT]^2_\leq$ and $k\in \mbN$, let us define $t^k_i = s + i 2^{-k} \absv{t-s}$ and $u^k_i = \inv{2}(t^k_i+t^k_{i+1})$ so that we have the concrete approximations $\{t\mapsto A^k_t\}_{k\in \mbN}$ which satisfy
	\begin{equation}\label{eq:ak_akmone}
	A^{k}_{s,t} - A^{k-1}_{s,t} = \sum_{i=0}^{2^k-1} \delta  A_{t_i,u_i,t_{i+1}} = \sum_{i=0}^{2^k-1} \int_{t^k_i}^{t^k_{i+1}} h_z \dd B_z + \sum_{i=0}^{2^k-1} J(t^k_i,t^k_{i+1}) =: I^k + J^k.
	\end{equation}
 For $J^k$ it follows immediately from the assumptions and definition of the sequence $(t_i)_{i=0,...,2^k-1}$ that
	\[ \big\| J^k \big\|_{L^m_\Omega} \leq \norm{ \sum_{i=0}^{2^k-1} J(t^k_i, t^k_{i+1}) }_{L^m_\Omega} \leq 2^k \sum_{i=0}^{2^k} \norm{ J(t^k_i,t^k_{i+1}) }_{L^m_\Omega} \leq C_J 2^{-k\eps_1} \absv{t-s}^{1+\eps_1}.  \]
	For $I^k$ we use additivity of a stochastic integral and write
	\[ \sum_{i=0}^{2^k-1} \int_{t^k_i}^{t^k_{i+1}} h_z \dd B_z = \int_s^t h_z \indic_{z \in \left[t^k_i,u^k_i\right]} \dd B_z. \]
	Then by the BDG inequality \eqref{eq:bdg_optimal} applied to the stochastic integrals and the assumption \eqref{eq:stoch_integral_sewing_h_bound} on $h$ we have
	\begin{align*}
		\norm{ \int_s^t h_z \indic_{z \in[t^k_i,u^k_i]} \dd B_z }_{L^m_\Omega} \leq & \sqrt{m}\,  \norm{ \left( \int_s^t h_z^2 \indic_{z \in \left[t^k_i,u^k_i\right]} \dd z \right)^{1/2} }_{L^m_\Omega} \leq \sqrt{m} \, C_{h,s} \norm{ \left( \sum_{i=0}^{2^k-1} \absv{u^k_i - t^k_i}^{1+\eps_2} \right)^{\nicefrac{1}{2}} }_{L^m_\Omega} \\ 
		\lesssim & \sqrt{m}\,  C_{h,s} 2^{-k \eps_2} \absv{t-s}^{\nicefrac{1}{2}+\eps_2}
	\end{align*}
	Hence,
    \begin{equation*}
\|I^k\|_{L^m_\Omega} + \| J^k \|_{L^m_\Omega}  \lesssim   2^{-k(\eps_1\wedge \eps_2)} \left( \sqrt{m} \absv{t-s}^{\nicefrac{1}{2}+\eps_2} + C_J \absv{t-s}^{1+\eps_1} \right)         
    \end{equation*}
    from which it follows that the series $\norm{I^k}_{L^m_\Omega}, \norm{ J^k }_{L^m_\Omega}$ are both summable, which is enough to conclude existence of the limit $\mcA$. 
\end{proof}
The following variant of Lemma~\ref{lem:sewing_stochastic_integral} applies to germs which have some decay in time leading to a globally bounded integral.
\begin{lem}\label{lem:sewing_decay_stochastic_integral}
    Let $ m \in [2,+\infty)$ and $\{ A^T : [0,T]\to L^m_\Omega \mbR^d \, \}_{T\geq 1}$  
    be a family of continuous process such that
	\begin{equation*}
	    \delta A^T_{s,u,t} =J^T(s,t) +  \int_s^t h^{T,u}_z \dd B_z ,\quad \text{for all }\,  (s,u,t)\in [0,T]^3_{\leq},\, T\geq 1,
	\end{equation*} 
where each $h^{T,u} : [0,T] \to L^m_\Omega \mbR^{\ell \times d}$ is a continuous matrix-valued $\mbF$-adapted process and $J^T: [0,T]^2_{\leq} \to L^m_\Omega \mbR^d$ is a continuous process. Suppose that there exist $\eps_1,\,\eps_2,\,\eps_3 >0$, $\lambda>0$, $C_A, \, C_J >0$ and a family of $\mbR$-valued random variables $\{C_{h,u}\}_{u \geq 0}\subset L^m_\Omega \mbR$, such that for all  $T\geq 1$ and $(s,t)\in [0,T]^2_\leq$,
\begin{equation}
		 \norm{ A^T_{s,t} }_{L^m_\Omega} \leq \, e^{-\lambda(T-t)} C_A \left(\absv{t-s}\wedge 1\right)^{\nicefrac{1}{2}+\eps_1}, \label{eq:sewing_decay_two_index_bound}
\end{equation}
and for all  $(s,t)\in [0,T]^2_\leq$ such that $\absv{t-s} \leq 1$ 
\begin{align}
     \norm{ J^T(s,t) }_{L^m_{\Omega}} \leq &\, e^{-\lambda(T-t)} C_{J} \absv{t-s}^{1+\eps_2},\label{eq:sewing_stoch_integral_decay_J_estimate}\\
     \norm{ h^{T,u} }_{L^2_{[s,t]}} \leq &\,  e^{-\lambda(T-t)} C_{h,u} \absv{t-s}^{\nicefrac{1}{2}+\eps_3}.\label{eq:sewing_stoch_integral_decay_h_estimate}
	 \end{align}
    Then, for any interval $\tint \subseteq \RR_+$ such that $[0,1]\subseteq \mcI$, there exists a unique family of real-valued $\mbF$-adapted process $\{ \mathcal{A}^T \}_{T\in \tint}$ such that for every $(s,t) \in [0,T]^2_{\leq}$ the bounds \eqref{eq:sewing_stoch_integral_bnd_1} holds on $[s,t]$, up to multiplication by  $e^{-\lambda(T-t)}$ on the right hand side. In addition, there exists a $C \coloneqq C(\eps_1,\eps_2,\eps_3)>0$ (crucially not depending on $\mcI$) such that for any $V >1$, 
\begin{equation}\label{eq:uniform_time_bound_ssl}
	\sup_{T\in \tint} \sup_{(s,t) \in [0,T]^2_\leq} \norm{ \mathcal{A}^T_{s,t} - A^T_{s,t} }_{L^m_{\Omega}} \leq C \left( \frac{ 2^{-V\left(\eps_2\wedge \frac{\eps_3}{2} \right)}   }{ \lambda \land \sqrt{\lambda} } \left( \sqrt{m} \sup_{u\in \tint}\|C_{h,u}\|_{L^m_\Omega} + C_{J} \right) + 2^{V} C_A \right).
	\end{equation}
\end{lem}
\begin{proof}[Sketch proof of Lemma~\ref{lem:sewing_decay_stochastic_integral}]
As in the proof of Lemma~\ref{lem:sewing_decay_stochastic_integral} we will construct $\mathcal{A}^T$ as the limit of approximations for $n\geq 0$,
\begin{equation*}
    A_t^{T;n} =\sum_{i=0}^{2^n-1} A_{t^n_i,t^n_{i+1}}^T
\end{equation*} 
where $t^n_i = i 2^{-n} t $ for $i\in \{0,\ldots,2^{n}\}$. Again, we will  focus on showing that the sequence  of process $\{A^{T;n}\}_{n\geq 0}$ converges to a limit satisfying \eqref{eq:uniform_time_bound_ssl}. To check uniqueness (in the claimed sense) and invariance of the limiting family over the choice of partitions, one can directly apply Lemma~\ref{lem:sewing_stochastic_integral} for each $T\geq 1$ fixed and it is readily checked that the resulting family is consistent, i.e. $\mcA^{2T}|_{[0,T]}= \mcA^T$ for any given $T\geq 1$.

To prove that any limiting family $\{\mcA^T\}_{T\geq 1}$ must satisfy \eqref{eq:uniform_time_bound_ssl}, let us first fix a $T\geq 1$ and $(s,t)\in [0,T]^2_\leq$. We then take any $V\in  [2,\infty) \cap \mbN$ and set
\begin{equation}\label{eq:bar_k_def}
     \bar{k} = \lfloor \log_2 \absv{t-s} \rfloor + V \in \mbN.
 \end{equation}
At the end of the proof we will argue that the claimed estimate holds for $V\in [1,+\infty)$.
We then consider the decomposition, for $n\geq  \bar{k} +2$ 
\begin{equation}\label{eq:decomposition}
  A^{T;n}_{s,t} - A^{T;0}_{s,t} = \underbrace{ \sum_{k=\bar{k}+1}^{n-1} \left( A^{T;k}_{s,t} - A^{T;k-1}_{s,t} \right) }_{\eqqcolon I^n_1} + \underbrace{ A^{T;\bar{k}}_{s,t} }_{\eqqcolon I_2} - \underbrace{ A^{T}_{s,t} }_{\eqqcolon I_3}.
\end{equation}
We bound each term in sequence. For $I^n_1$ we will show that the sequence of partial sums is absolutely convergent. We set $u^k_i = \inv{2}(t^k_i + t^k_{i+1})$ and decompose, for $k\geq \bar{k} +1$,
    \begin{equation*}
        A^{T;k}_{s,t} - A^{T;k-1}_{s,t} = \sum_{i=1}^{2^k-1} \delta A^T_{t^k_i,u^k_i,t_{i+1}} = \sum_{i=1}^{2^k-1}  J^T(t^k_i,t^k_{i+1}) +  \sum_{i=1}^{2^k-1}\int_{t^k_i}^{t^k_{i+1}} h^{T,u^k_i}_z \dd B_z \eqqcolon U^T_k + L^T_k 
    \end{equation*} 
Taking the $L^m_\Omega$  norm of $U^T_k$ and applying the assumed bound on $J^T$ \eqref{eq:sewing_stoch_integral_decay_J_estimate} we have 
\begin{equation}\label{eq:uk_first_bound}
    \begin{aligned}
        \norm{ U^T_k }_{L^m_\Omega} \leq & \,C_{J} e^{-\lambda T} \sum_{i=0}^{2^k-1} e^{\lambda t_{i+1}^k} \absv{t_{i+1}^k - t_i^k}^{1+\eps_2} \\
        \leq &\, C_{J} e^{-\lambda T} \left( \absv{t-s} 2^{-k} \right)^{1+\eps_2} e^{\lambda 2^{-k}\absv{t-s}} \sum_{i=0}^{2^k-1} e^{\lambda i 2^{-k} \absv{t-s}} \\
        \leq &\, C_{J} e^{-\lambda T} \left( \absv{t-s} 2^{-k} \right)^{1+\eps_2} e^{\lambda 2^{-k}\absv{t-s}} \frac{ e^{\lambda \absv{t-s} } - 1}{ e^{\lambda 2^{-k} \absv{t-s} } -1 }.
    \end{aligned}
	\end{equation}
   Since we have $k>\bar{k}$ (recalling \eqref{eq:bar_k_def}) it holds that $e^{\lambda 2^{-k}\absv{t-s}} \leq e^{\lambda}$. Using in addition the trivial facts that $e^{-\lambda T} \left( e^{\lambda\absv{t-s}} -1 \right) \leq 1$ and $e^{x}>1+x$ for $x>0$, we see that  
    \begin{equation*}
        \|U^k\|_{L^m_\Omega} \leq C_{J} \left( 2^{-k} \absv{t-s} \right)^{1+\eps_2} \left( 2^{-k} \absv{t-s} \right)^{-1} \lambda^{-1} \leq C_{J} \lambda^{-1} \left( 2^{-k} \absv{t-s} \right)^{\eps_2}
    \end{equation*} 
To deal with $L^T_k$, we exchange summation and integration for $h$ in the same way as in the proof of Lemma \ref{lem:sewing_stochastic_integral} and appeal to \eqref{eq:bdg_optimal} to obtain 
\begin{align*}
\norm{L^T_k}_{L^m_{\Omega}} & \leq  \sqrt{m} \, \norm{ \left( \int_s^t |h^{T,u^k_i}_z|^2 \indic_{z \in [t^k_i, u^k_i] } \dd z \right)^{\nicefrac{1}{2}} }_{L^m_{\Omega}} \\ 
& \leq \sqrt{m} \, \sup_{u \geq  0} \norm{ C_{h,u} }_{L^m_{\Omega}} e^{-\lambda T} \left( \sum_{i=0}^{2^k-1} e^{2\lambda t^k_{i+1}} \absv{t^k_{i+1}-t^k_i}^{1+\eps_3} \right)^{\nicefrac{1}{2}} \\
& \leq \sqrt{m}\,  \sup_{u \geq  0} \norm{ C_{h,u} }_{L^m_{\Omega}} e^{-\lambda T} \left( \left( \absv{t-s} 2^{-k} \right)^{1+\eps_3} e^{\lambda 2^{-k}\absv{t-s}} \frac{ e^{\lambda \absv{t-s} } - 1}{ e^{\lambda 2^{-k} \absv{t-s} } -1 } \right)^{\nicefrac{1}{2}}.
\end{align*}
Using the same concluding arguments as in bounding $\|U^T_k\|_{L^m_\Omega}$ we see that we have
\[ \norm{L^T_k}_{L^m_{\Omega}} \leq \sqrt{m}\,  \sup_{u \geq  0} \norm{ C_{h,u} }_{L^m_{\Omega}}  \lambda^{-\nicefrac{1}{2}} \left( 2^{-k} \absv{t-s} \right)^{\nicefrac{\eps_3}{2}}. \]
Hence, by the triangle inequality
\begin{align*}
   \|I^n_1\|_{L^m_\Omega} \, \leq & \, \sum_{k= \bar{k}+1}^n \left(\norm{ U^T_k }_{L^m_\Omega} + \norm{L^T_k}_{L^m_{\Omega}}  \right)\\
   \leq &\, C_J \lambda^{-1}|t-s|^{\eps_2} \sum_{k= \bar{k}+1}^{n-1} 2^{-\eps_2 k} + \sqrt{m}\,  \sup_{u \geq  0} \norm{ C_{h,u} }_{L^m_{\Omega}}  \lambda^{-\nicefrac{1}{2}} |t-s|^{\nicefrac{\eps_3}{2}}\sum_{k= \bar{k}+1}^{n-1} 2^{-k \nicefrac{\eps_3}{2}}\\
   =&\, C_J \lambda^{-1}|t-s|^{\eps_2} \frac{2^{-\eps_2 (\bar{k}+1)} - 2^{-\eps_2 n}}{1-2^{-\eps_2}}\\
   &\, + \sqrt{m}\,  \sup_{u \geq  0} \norm{ C_{h,u} }_{L^m_{\Omega}}  \lambda^{-\nicefrac{1}{2}} |t-s|^{\nicefrac{\eps_3}{2}}\frac{2^{-(\bar{k}+1)\nicefrac{\eps_3}{2}} - 2^{-n\nicefrac{\eps_3}{2} }}{1-2^{-\nicefrac{\eps_3}{2}}}.
\end{align*}
Since the right hand side converge as $n\to +\infty$ we see that the sequence of partial sums $\{I^n_1\}_{n\geq \bar{k}}$ is absolutely convergent and using the definition of $\bar{k}$ \eqref{eq:bar_k_def}, a fortiori we have the bound
  \begin{align}
     \sup_{n\geq \bar{k}+2} \norm{ I^n_1 }_{L^m_{\Omega}} &\, \leq \,  C_J \lambda^{-1} \frac{2^{-\eps_2(V+1)} }{1-2^{-\eps_2}} + \sqrt{m}\,  \sup_{u \geq  0} \norm{ C_{h,u} }_{L^m_{\Omega}}  \lambda^{-\nicefrac{1}{2}}  \frac{  2^{-\frac{\eps_3}{2}(V+1)} }{1-2^{-\nicefrac{\eps_3}{2}}}, \notag \\
& \, \lesssim_{\eps_2,\,\eps_3} \, \frac{2^{-V \left(\eps_2\wedge \frac{\eps_3}{2}\right) }}{\lambda \wedge \sqrt{\lambda}} \left(C_J + \sqrt{m}\,  \sup_{u \geq  0} \norm{ C_{h,u} }_{L^m_{\Omega}} \right).  \label{eq:decay_sewing_I_1}
    \end{align}
To control $I_2$ we observe that by definition of $\bar{k}$ we may appeal to \eqref{eq:sewing_decay_two_index_bound} and the triangle inequality to see that 
  \begin{align*}
      \norm{ I_2 }_{L^m_\Omega} \leq \, \sum_{i=0}^{2^{\bar{k}}-1} \left\| A^T_{t_i^{\bar{k}},t_{i+1}^{\bar{k}}} \right\|_{L^m_\Omega} \leq &\,  e^{-\lambda T}\sum_{i=0}^{2^{\bar{k}}-1} e^{\lambda t^{\bar{k}}_i} C_A \left|t^{\bar{k}}_i-t^{\bar{k}}_{i+1}\right|^{\nicefrac{1}{2}+\eps_1}\\
      \leq &\, e^{-\lambda T} \absv{t-s}^{\nicefrac{1}{2}+\eps_1} 2^{-\bar{k}(\nicefrac{1}{2}+\eps_1)} C_A \sum_{i=0}^{2^{\bar{k}}-1} e^{\lambda (i+1) 2^{-\bar{k}}\absv{t-s}} \\
     \leq & \, e^{-\lambda T} 2^{-V(\nicefrac{1}{2}+\eps_1)} C_A e^{\lambda 2^{-\bar{k}} \absv{t-s} } \frac{ e^{\lambda\absv{t-s} }-1}{e^{\lambda 2^{-\bar{k}}\absv{t-s}} - 1}\\
     \lesssim &\, C_A   2^{-V(\nicefrac{1}{2}+\eps_1)} \frac{ e^{\lambda 2^{-V}} }{e^{\lambda 2^{-V}} - 1} \underbrace{e^{-\lambda T}(e^{\lambda\absv{t-s} }-1)}_{\leq 1}.
  \end{align*} 
In the last line we used the definition of $\bar{k}$. We now split the estimate depending on the choice of $V$. If $\lambda 2^{-V} > 1$, then we have the  bound $\frac{ e^{\lambda 2^{-V} }}{ e^{\lambda 2^{-V}}-1} \leq  \frac{1}{1-e^{-\lambda2^{-V}}}  \leq \sup_{x >1} \inv{1 - e^{-x}}<+\infty$. If $V$ and $\lambda$ are such that $\lambda 2^{-V} < 1$, then we may simply use the fact that $\frac{1}{e^x -1} \leq \frac{1}{x}$ for $x\geq 0$ to obtain
\begin{equation*}
     2^{-V(\nicefrac{1}{2}+\eps)}  \frac{e^{\lambda 2^{-V}}}{ e^{\lambda 2^{-V}} - 1}  \lesssim  \frac{2^{-V(\nicefrac{1}{2}+\eps)} }{\lambda 2^{-V}} = \lambda^{-1} 2^{\nicefrac{V}{2}-V\eps_1}.
    \end{equation*}
Hence, we find
\begin{equation}\label{eq:decay_sewing_I_2}
  \norm{ I_2 }_{L^m_{\Omega}} \lesssim C_A \left(  2^{-V\left(\nicefrac{1}{2}+\eps_1 \right)} +  \lambda^{-1} 2^{V\left(\nicefrac{1}{2}-\eps_1 \right)}  \right). 
  \end{equation}
Since $I_3 = A^T_{s,t}$ we  directly apply  \eqref{eq:sewing_decay_two_index_bound} to have
\begin{equation}\label{eq:decay_sewing_I_3}
	\|I_3\|_{L^m_\Omega} \leq e^{-\lambda(T-t)} C_A \left(\absv{t-s}\wedge 1\right)^{\nicefrac{1}{2}+\eps_1} \leq C_A.
\end{equation}
Finally, combining \eqref{eq:decay_sewing_I_1}, \eqref{eq:decay_sewing_I_2} and \eqref{eq:decay_sewing_I_3}, recombining terms and absorbing various implicit constants we obtain, for $V\in [2,+\infty)\cap \mbN$,
\begin{equation*}
	 \sup_{n\geq \bar{k}+2 } \|A^{T;n}_{s,t} - A^{T;0}_{s,t} \|_{L^m_\Omega} \lesssim_{\eps_1,\eps_2,\eps_3}  \left( \frac{ 2^{-V\left(\eps_2\wedge \frac{\eps_3}{2} \right)}   }{ \lambda \land \sqrt{\lambda} } \left( \sqrt{m} \sup_{u\in \mbR}\|C_{h,u}\|_{L^m_\Omega} + C_{J} \right) + 2^{\frac{V}{2}} C_A \right).
\end{equation*}
So that, up to modifying the multiplicative constant, any limit $\mcA^T$ of $A^{T;n}$ (taken in $L^m_\Omega$) must obey \eqref{eq:uniform_time_bound_ssl} for $V\in [1,+\infty)$. We observe that the multiplicative constant never depends on the length of the interval.
\end{proof}
\section{An Abstract Ergodicity Result}\label{app:doob_khasminski_proof}
We are required to check that the conclusion of Theorem~\ref{th:doob_khasminskii} still holds despite our relaxation of the continuity conditions on the flow map and notions of strong Feller, topological irreducibility and quasi-Markovianity. To this end we give our own version of the abstract ergodicity result \cite[Thm.~3.2]{hairer_ohashi_07_ergodic}.

\begin{thm}\label{th:abstract_ergodic_result}
	Given an SDS $\Lambda$ in the sense of Definition~\ref{def:sds} built over a stationary noise process $(\mcW,\{\msP_t\}_{t\geq 0}, \bP_\omega, \{\theta_t\}_{t\geq 0})$, let us define the map
	\begin{equation*}
		\begin{aligned}
		\hat{\Lambda}_t :\mcX^2 \times \mcW^2 &\to \mcX^2\times \mcW^2\\
		(x,y,w,w') &\mapsto (\Lambda_t(x,w),\Lambda_t(y,w'),w,w'), 
	\end{aligned}%
	\quad \text{for all }\, t>0.
	\end{equation*}
	Assume it is possible to define a jointly measurable map 
	\begin{equation*}
		(x,y,w) \mapsto \msP^{x,y}_t(w,\,\cdot\,) \in \mcM_+(\mcW^2)
	\end{equation*}
	such that
	\begin{enumerate}[label=\roman*)]
		\item \label{it:abstract_ergodict_1} the measure $\msP^{x,y}_t(w,\,\cdot\,)$ is a sub-coupling for the pair $\msP_t(w,\,\cdot\,)$ and $\msP_t(w,\,\cdot\,)$ for every $(x,y,w)\in \mcX^2 \times \supp(\bP_\omega) \subset \mcX^2 \times \mcW$;
		\item \label{it:abstract_ergodict_2} there exists an $s>0$ such that 
		\begin{equation*}
			(\hat{\Lambda}_t(x,y)^* \msP^{x,y}_t(w,\,\cdot\,))(\mcN^s) >0
		\end{equation*}
		for every $(x,y,w) \in \mcX^2\times \supp(\bP_\omega)$, where $\hat{\Lambda}_t(x,y) \coloneqq \hat{\Lambda}_t(x,y,\,\cdot\,,\,\cdot\,)$.
	\end{enumerate}
	Then, $\Lambda$ can have at most one invariant measure (in the sense of Definition~\ref{def:invariant_ergodic_measures}) up to the equivalence relationship of \cite[Def.~2.4]{hairer_ohashi_07_ergodic}. 
\end{thm}
\begin{proof}
	One may follow the proof of \cite[Thm.~3.2]{hairer_ohashi_07_ergodic} almost verbatim, except to notice that when constructing the measure $\theta^{(\mu,\nu)} \in \mcM_+(\mcN^s_\mcW)$ defined by setting
	\begin{equation*}
		\theta^{(\mu,\nu)}(A) \coloneqq \int_\mcW \int_{\mcX^2} \left(\hat{\Lambda}_t(x,y)^* \msP^{x,y}_t(w,\,\cdot\,)\right)(A\cap \mcN^s_\mcW) \mu(w,\dd x) \nu(w,\dd y) \bP_\omega(\dd w),
	\end{equation*}
	we only require that $\msP^{x,y}_t(w,\,\cdot\,)$ be a sub-coupling for $\msP_t(w,\,\cdot\,)$ and $\msP_t(w,\,\cdot\,)$ for $\bP_\omega$-almost every $w\in \mcW$. Note similarly that requiring point \ref{it:abstract_ergodict_2} to only hold for $\mbP_w$-a.e. $w\in \mcW$ is no obstacle either since $w$ is integrated against $\bP_\omega$ in the definition of $\theta^{(\mu,\nu)}$. Recall the definition of $\mcN^s_\mcW$ given in \cite[Sec.~3]{hairer_ohashi_07_ergodic} and the times $s,\,t>0$ specified in the proof of \cite[Thm.~3.2]{hairer_ohashi_07_ergodic}.
\end{proof}
We are now able to prove Theorem~\ref{th:doob_khasminskii} which follows the proof of \cite[Thm.~3.10]{hairer_ohashi_07_ergodic}.
\begin{proof}[Proof of Theorem~\ref{th:doob_khasminskii}]
	As in \cite{hairer_ohashi_07_ergodic} our strategy is to show that under our assumptions we are able to verify the conditions of Theorem~\ref{th:abstract_ergodic_result}. We proceed as in the proof of \cite[Thm.~3.10]{hairer_ohashi_07_ergodic}, first noting that the function $f:\mcW\to \mcW$ constructed therein is such that $f(\mcW) \subseteq \supp(\bP_w)$. Hence, using the notation given therein, for any $(x,w) \in \mcX\times \supp(\bP_\omega)$ we have $(\tilde{x}(x,w),f(w))\in \mcX\times \supp(\bP_w)$. Therefore, since we assumed that the function $(x,y,w)\mapsto \ell(x,y,w)$ is jointly continuous at any point of $\mcX^2 \times \supp(\bP_\omega)$ we equally conclude that $r(x,w)>0$ for every $(x,w)\in \mcX\times \supp(\bP_w)$, where
	\begin{equation*}
		r(x,w) \coloneqq \sup \left\{\rho>0\,|\, \ell(x',\tilde{x}(x,w),w') \leq \nicefrac{1}{4}\,\text{for all }\, (x',w')\in \msB_\rho(x,w)\right\}
	\end{equation*}
	with $\msB_\rho(x,w) \coloneqq \msB^\mcX_\rho(\tilde{x}(x,w))\times \msB^\mcW_\rho(f(w))$, since by definition $\ell(\tilde{x}(x,w),\tilde{x}(x,w),f(w))=0$. Similarly it follows that the quantity $n_x$ defined in \cite[Eq.~(3.6)]{hairer_ohashi_07_ergodic} is finite for every $(x,w)\in \mcX\times \supp(\bP_\omega)$. The remaining properties are checked using topological irreducibility and quasi-Markovianity, recalling that we only need consider $w\in \supp(\bP_\omega)$. Hence, we conclude that the map $\mcW \ni w \mapsto \msP^{x,y}_s(w,\,\cdot\,)$ as defined by \cite[Eq.~(3.4) \& (3.6)]{hairer_ohashi_07_ergodic} satisfies the assumptions of our Theorem~\ref{th:abstract_ergodic_result} which concludes the proof.

\end{proof}
\section{Properties of the Operator $\msA$}\label{app:A_operator_properties}
Recall the operator defined by \cite[(Eq.~4.3)]{hairer_ohashi_07_ergodic}
\begin{equation}\label{eq:A_op_def_1}
   ( \msA w)_t \coloneqq (H-\nicefrac{1}{2})c_H \alpha_{1-H} \int_0^\infty \frac{1}{r}\, f\bigg(\frac{t}{r}\bigg) w_{-r}\dd r,
\end{equation}
with 
\begin{equation}\label{eq:A_op_def_2}
    f(x) \coloneqq x^{H-\nicefrac{1}{2}}+ (H-\nicefrac{3}{2})x \int_0^1 \frac{(u+x)^{H-\nicefrac{5}{2}}}{(1-u)^{H-\nicefrac{1}{2}}
    }\dd u.
\end{equation}
The following lemma gives a statement akin to \cite[Lem.~A.1]{hairer_ohashi_07_ergodic} for $H\in (0,\nicefrac{1}{2})$. See also the discussion around \cite[Eq.~(4.9)]{hairer_pillai_11_regularity}.
\begin{lem}\label{lem:g_func_scaling}
Let $H\in (0,\nicefrac{1}{2})$ and $g:\mbR_+ \to \mbR$ be defined by \eqref{eq:A_op_def_2}. Then, it holds that
\begin{equation}
    f(x) = \begin{cases}
        O(x), & x\in  (0,1),\\
        O\big(x^{H-\frac{1}{2}}\big),& x\in [1,+\infty)
    \end{cases}
    \quad \text{and}\quad 
    f'(x) = \begin{cases}
        O(1), & x\in  (0,1),\\
        O\big(x^{H-\frac{3}{2}}\big),& x\in [1,+\infty).
    \end{cases}
\end{equation}
\end{lem}
\begin{proof}
The proof follows exactly the steps of \cite[Lem.~A.1]{hairer_ohashi_07_ergodic}. The case $x\geq 1$ follows from the estimate
\begin{equation}
    x^{H-\nicefrac{1}{2}} + (H-\nicefrac{3}{2}) x^{H-\nicefrac{3}{2}} \leq f(x)\leq x^{H-\nicefrac{1}{2}}.
\end{equation}
For $x\in (0,1)$, as in the proof of \cite[Lem.~A.1]{hairer_ohashi_07_ergodic} we re-write $g(x)$ as
\begin{equation}
    f(x) = C_1 x(1+x)^{H-\nicefrac{1}{2}} + C_2 x \int_0^1 (u+x)^{H-\nicefrac{5}{2}} \big(1-(1-u)^{\nicefrac{1}{2}-H}\big) \dd u,
\end{equation}
for two constants $C_1,\,C_2>0$. Since $H\in (0,\nicefrac{1}{2})$ the function
\begin{equation}
    (u+x)^{H-\nicefrac{5}{2}}(1-(1-u)^{\nicefrac{1}{2}-H}),
\end{equation}
is integrable on $(0,1)$ for all $x>0$ and we have
\begin{equation}
    \int_0^1 (u+x)^{H-\nicefrac{5}{2}} \big(1-(1-u)^{\nicefrac{1}{2}-H}\big) \dd u = O(x^{H-\nicefrac{3}{2}}),
\end{equation}
which concludes the proof for $g$. Similar arguments apply to $g'$.
\end{proof}
The following is an analogue of \cite[Prop.~A.2]{hairer_ohashi_07_ergodic} in the case $H\in (0,\nicefrac{1}{2})$.
\begin{prop}\label{prop:A_map_bounded}
    Let $H\in (0,\nicefrac{1}{2})$, $\gamma,\, \delta >0$ be such that $\gamma \in (0,H)$ and $\delta \in (H-\gamma,1-\gamma)$. Then $\msA$ is a bounded linear operator from $\mcW^-_{\gamma,\delta}$ into $\mcW^+_{\gamma,\delta}$.
\end{prop}
\begin{proof}
   The proof follows exactly the steps of the proof of \cite[Prop.~A.2]{hairer_ohashi_07_ergodic} only using our Lemma~\ref{lem:g_func_scaling} in place of \cite[Prop.~A.2]{hairer_ohashi_07_ergodic}.
\end{proof}
We also require the following technical lemma which the exact analogue of \cite[Lem.~5.9]{hairer_ohashi_07_ergodic}.
\begin{lem}\label{lem:L_2_image_deriv_A}
    For $H\in (0,\nicefrac{1}{2})$, $\msA$ as defined in \eqref{eq:A_op_def_2} and $\mfh:\mbR_-\to \mbR$ such that $\mfh'\in C^\infty_0(\mbR_-;\mbR)$, it holds that $\msD^{H+\nicefrac{1}{2}}\msA \mfh \in L^2(\mbR_+;\mbR)$.
\end{lem}
\begin{proof}
The proof follows mutatis-mutandis that of \cite[Lem.~5.9]{hairer_ohashi_07_ergodic} only replacing \cite[Lem.~A.1]{hairer_ohashi_07_ergodic} with our Lemma~\ref{lem:g_func_scaling}.
\end{proof}
\bibliography{refs_madry.bib} 
\bibliographystyle{alpha}

\end{document}